\documentclass[reqno,twoside,11pt]{article}
 \usepackage[english]{babel}
 \usepackage[T1]{fontenc}
 \renewcommand{\rmdefault}{ptm}
 \usepackage{times}
 \usepackage[
    paper=a4paper,
    portrait=true,
    textwidth=425pt,
    textheight=664pt,
    tmargin=2.8cm,
    marginratio=1:1
            ]{geometry}
 \usepackage[all]{xy}       
 \usepackage[centertags]{amsmath}
 \usepackage{amsfonts}
 \usepackage{amsthm}
 \usepackage{amssymb}
 \usepackage{enumerate}
 \usepackage{mathrsfs}      
 \usepackage{ifthen}
 \usepackage{graphicx}
 \usepackage[active]{srcltx}    

 \theoremstyle{plain}
 \newtheorem{thm}{Theorem}
 \newtheorem{cor}[thm]{Corollary}
 \newtheorem{lemma}[thm]{Lemma}
 \newtheorem{prop}[thm]{Proposition}
 \theoremstyle{definition}
 \newtheorem{rem}[thm]{Remark}
 \numberwithin{thm}{section}
 \numberwithin{equation}{section}

 \newcommand{\R}{\mathbf{R}}
 \newcommand{\C}{\mathbf{C}}
 \newcommand{\N}{\mathbf{N}}

 \newcommand{\g}[1][g]{\mathfrak{#1}}
 \newcommand{\LG}[1][{}]{\textbf{#1}}
 \newcommand{\Acomp}{\mathfrak{A}}
 \newcommand{\U}{\mathcal{U}}
 \newcommand{\Sym}{\mathbf{S}}
 \newcommand{\WG}[1][empty]{\def\ArgI{#1}\WRelayI}
    \newcommand{\WRelayI}
    {
        \ifthenelse{\equal{\ArgI}{empty}}
            {W}
            {W_{\ArgI}}
    }
 \newcommand{\WGs}[1][empty]{\def\ArgI{#1}\WsRelayI}
    \newcommand{\WsRelayI}
    {
        \ifthenelse{\equal{\ArgI}{empty}}
            {\mathcal{W}}
            {\mathcal{W}_{\ArgI}}
    }
 \newcommand{\Nor}{\mathcal{N}}
 \newcommand{\Cen}{\mathcal{Z}}
 \newcommand{\Cone}{\mathcal{C}}
 \newcommand{\invCone}[1]{\Gamma_{#1}^{W_{K\cap H}}}
 \newcommand{\Hor}{\textnormal{Hor}}

 \newcommand{\Psg}{\mathcal{P}}
 
 \newcommand{\diag}{\textnormal{diag}}
 \newcommand{\Ft}{\mathcal{F}}
 
 \newcommand{\Ht}[1][empty]{\def\ArgI{#1}\HtRelayI}
    \newcommand{\HtRelayI}[1][empty]{\def\ArgII{#1}\HtRelayII}
    \newcommand{\HtRelayII}
    {
        \ifthenelse{\equal{\ArgII}{empty}}
        {
            \ifthenelse{\equal{\ArgI}{empty}}
                {\mathcal{H}}
                {\mathcal{H}_{\ArgI}}
        }
        {\mathcal{H}_{\ArgI}^{\ArgII}}
    }
 \newcommand{\Rt}[1][empty]{\def\ArgI{#1}\RtRelayI}
    \newcommand{\RtRelayI}[1][empty]{\def\ArgII{#1}\RtRelayII}
    \newcommand{\RtRelayII}
    {
        \ifthenelse{\equal{\ArgII}{empty}}
        {
            \ifthenelse{\equal{\ArgI}{empty}}
                {\mathcal{R}}
                {\mathcal{R}_{\ArgI}}
        }
        {\mathcal{R}_{\ArgI}^{\ArgII}}
    }
 \newcommand{\dRt}[1][empty]{\def\ArgI{#1}\dRtRelayI}
    \newcommand{\dRtRelayI}[1][empty]{\def\ArgII{#1}\dRtRelayII}
    \newcommand{\dRtRelayII}
    {
        \ifthenelse{\equal{\ArgII}{empty}}
        {
            \ifthenelse{\equal{\ArgI}{empty}}
                {\mathcal{S}}
                {\mathcal{S}_{\ArgI}}
        }
        {\mathcal{S}_{\ArgI}^{\ArgII}}
    }
 \newcommand{\Tt}[1][empty]{\def\ArgI{#1}\TtRelayI}
    \newcommand{\TtRelayI}[1][empty]{\def\ArgII{#1}\TtRelayII}
    \newcommand{\TtRelayII}
    {
        \ifthenelse{\equal{\ArgII}{empty}}
        {
            \ifthenelse{\equal{\ArgI}{empty}}
                {\mathcal{T}}
                {\mathcal{T}_{\ArgI}}
        }
        {\mathcal{T}_{\ArgI}^{\ArgII}}
    }
 \newcommand{\Ds}{\mathscr{D}}
 \newcommand{\Es}{\mathscr{E}}
 \newcommand{\Ss}{\mathscr{S}}
 \newcommand{\RDs}{\mathscr{S}}
 \newcommand{\Cs}{\mathscr{C}}
 \newcommand{\Vs}{\mathscr{V}}
 \renewcommand{\Re}{\textnormal{Re}}
 
 \newcommand{\Ad}{\textnormal{Ad}}
 \newcommand{\ad}{\textnormal{ad}}
 \newcommand{\tr}{\textnormal{tr}}
 \newcommand{\supp}{\textnormal{supp}}
 \newcommand{\ch}{\textnormal{conv}}
 \renewcommand{\1}{\mathbf{1}}
 \renewcommand{\H}{\mathcal{H}}
 \newcommand{\Ind}{\textnormal{Ind}}
 \newcommand{\Hom}{\textnormal{Hom}}
 \newcommand{\End}{\textnormal{End}}
 
 \newcommand{\un}{\textnormal{un}}
 \newcommand{\ltpb}[1]{l^{*}_{#1}}
 \newcommand{\pr}{\textnormal{pr}}
 \newcommand{\leftsuperscript}[2]{{\vphantom{#2}}^{#1}{#2}}

 \addto\captionsenglish{}
 \addto\captionsenglish{}

 \title{Radon transformation on reductive symmetric spaces: support theorems}
 \author{Job\,\,J.\,Kuit}
 \date{}
\begin{document}
 \maketitle
 \begin{abstract}
We introduce a class of Radon transforms for reductive symmetric spaces, including the horospherical transforms, and derive support theorems for these transforms.

A reductive symmetric space is a homogeneous space $G/H$ for a reductive Lie group $G$ of the Harish-Chandra class, where $H$ is an open subgroup of the fixed-point subgroup for an involution $\sigma$ on $G$. Let $P$ be a parabolic subgroup such that $\sigma(P)$ is opposite to $P$ and let $N_{P}$ be the unipotent radical of $P$. For a compactly supported smooth function $\phi$ on $G/H$, we define $\Rt[P](\phi)(g)$ to be the integral of $N_{P}\ni n\mapsto\phi(gn\cdot H)$ over $N_{P}$. The Radon transform $\Rt[P]$ thus obtained can be extended to a large class of distributions containing the rapidly decreasing smooth functions and the compactly supported distributions.

For these transforms we derive support theorems in which the support of $\phi$ is (partially) characterized in terms of the support of $\Rt_{P}\phi$. The proof is based on the relation between the Radon transform and the Fourier transform on $G/H$, and a Paley-Wiener-shift type argument. Our results generalize the support theorem of Helgason for the Radon transform on a Riemannian symmetric space.
\end{abstract}

 \tableofcontents
\section*{Introduction}
\addcontentsline{toc}{section}{\protect\numberline{}Introduction}
Let $G$ be a noncompact connected semisimple Lie group with finite center,
let $G=KAN$ be an Iwasawa decomposition of $G$ and let $M$ be the
centralizer of $A$ in $K$. A horosphere in $X = G/K$ is a submanifold of the form $gN\cdot
x_{0}$, $x_{0}=e\cdot K$. The set of horospheres is isomorphic (as a $G$-space) to $\Xi=G/MN$ via the map
$E:g\cdot MN\mapsto g\cdot\xi_{0}$, $\xi_{0}=N\cdot x_{0}$.
The Radon transform on $X$ is the $G$-equivariant map
$\Rt:C_{c}^{\infty}(X)\to C^{\infty}(\Xi)$ given by
$$
\Rt\phi(g\cdot \xi_{0})=\int_{N}\phi(gn\cdot x_{0})\,dn.
$$
In \cite[Lemma
8.1]{Helgason_TheSurjectivityOfInvariantDifferentialOperatorsOnSymmetricSpaces} Helgason proved the following support theorem for this
transform.
\medbreak
\medbreak
\noindent
{\em Let $\phi$ be a
compactly supported smooth function on $X$ and let $V$ be a closed ball
in $X$. Assume that $\Rt\phi(\xi)=0$ whenever $E(\xi)\cap V=\emptyset$.
Then
$\phi(x)=0$ for $x\notin V$.
}
\medbreak
\noindent Note that this theorem implies that $\Rt$ is injective on the space of compactly supported smooth
functions.

In this article Helgason's result is generalized to a support theorem
for a class of Radon transforms (including the horospherical
transforms) on a reductive symmetric space $X=G/H$ with $G$ a
real reductive Lie group of  the Harish-Chandra class and $H$ an
essentially connected open subgroup of the fixed-point subgroup
$G^{\sigma}$ of an involution $\sigma$ on $G$.

Let $\theta$ be a Cartan involution of $G$ commuting with $\sigma$. Let $P$ be a $\sigma\circ\theta$-stable parabolic subgroup. We write $N_{P}$ for the unipotent radical of $P$ and $A_{P}$ for the $\theta$-stable split component of $P$. We consider the Radon
transform $\Rt[P]$ mapping a function $\phi$ on $X$ to the
function on the homogeneous space $\Xi_{P}=G/\Cen_{H}(A_{P})N_{P}$ given by
$$
\Rt[P]\phi(g\cdot\xi_{P})
=\int_{N_{P}}\phi(gn\cdot x_{0})\,dn.
$$
Here $\xi_{P}=e\cdot\Cen_{H}(A_{P})N_{P}$ and $x_{0}=e\cdot H$.  This Radon transform, which
is initially defined for compactly supported smooth functions, can
be extended to a large class of distributions on $X$. For a minimal $\sigma\circ\theta$-stable parabolic subgroup $P_{0}$ the transform $\Rt[P_{0}]$ is called the horospherical transform related to $P_{0}$.

If $P_{0}$ is a minimal $\sigma\circ\theta$-stable parabolic
subgroup of $G$ contained in $P$, then $A_{P}\subseteq A_{P_{0}}$. The Lie algebra
$\g[a]_{P_{0}}$ of $A_{P_{0}}$ is $\sigma$-stable and decomposes
as the direct sum of the $+1$ and $-1$ eigenspaces for $\sigma$.
The latter space is denoted by $\g[a]_{\g[q]}$.

The maps
$$K\times \g[a]_{\g[q]}\to X;\qquad(k,Y)\mapsto k\exp(Y)\cdot x_{0}$$
and
$$K\times \g[a]_{\g[q]}\to \Xi_{P};\qquad(k,Y)\mapsto k\exp(Y)\cdot \xi_{P}$$
are surjective.
For a subset $B$ of $\g[a]_{\g[q]}$, we define
$$X(B)=K\exp(B)\cdot x_{0}\quad\text{and}\quad\Xi_{P}(B)=K\exp(B)\cdot\xi_{P}.$$
\medbreak
If $\phi$ is a function on $X$ with compact support, then the support of $\Rt_{P}\phi$ need not be compact in general. In fact, if $\supp(\phi)\subseteq X(B)$ for some compact convex subset $B$ of $\g[a]_{\g[q]}$ that is invariant under the action of the normalizer $\Nor_{K\cap H}(\g[a]_{\g[q]})$ of $\g[a]_{\g[q]}$ in $K\cap H$,
then $\supp(\Rt[P]\phi)\subseteq\Xi_{P}(B+\Gamma_{P})$, where $\Gamma_{P}$ is the cone in
$\g[a]_{\g[q]}$ spanned by the root vectors corresponding to roots that are positive with respect to $P$. The support theorem (Theorem  \ref{SuppThm thm Support theorem for distributions and non-minimal P}) that we prove
in this article is a partial converse to this statement for $\mu$ in a
suitable class of distributions, containing the compactly supported ones.
\medbreak
\medbreak
\noindent
{\bf Support Theorem.}
{\em
    Let $B$ be a convex compact
    subset of $\g[a]_{\g[q]}$ that is invariant under the action of $\Nor_{M_{P}\cap K\cap H}(\g[a]_{\g[q]})$ on  $\g[a]_{\g[q]}$. If
    $$\supp(\Rt[P]\mu)\subseteq \Xi_{P}(B+\Gamma_{P}),$$
    then
    $$\supp(\mu)\subseteq X\Big(\bigcap_{k\in\Nor_{K\cap H}(\g[a]_{\g[q]})}\Ad(k)(B+\Gamma_{P})\Big).$$
    }
\medbreak
\medbreak
If $\sigma=\theta$ and $P$ is a minimal parabolic subgroup of $G$, then $\bigcap_{k\in\Nor_{K\cap H}(\g[a]_{\g[q]})}\Ad(k)(B+\Gamma_{P})=B$. Our theorem reduces then to Helgason's support theorem.
Just as in the Riemannian case, the support theorem implies injectivity of the Radon transform.

After some preliminaries in Section \ref{section Notation and Preliminaries}, we introduce the
transforms under consideration and establish some of their
properties in Section \ref{section Radon transformation on a
reductive symmetric space}.
In Section \ref{section convex geometry} we derive some results in convex geometry that we need in the following  sections. Then, in Section \ref{section Support of a transformed function}, we give a description of the support of a transformed function in terms of the support of that function. The support theorem is a partial converse of this result. To prove the support theorem for the horospherical transform, we first establish a relation between the horospherical transform $\Rt[P_{0}]$ (related to a minimal parabolic subgroup $P_{0}$) and the Fourier transform on $X$. This allows us to derive a Paley-Wiener type estimate for the Fourier transform of a function $\phi$ in terms of $\supp(\Rt[P_{0}]\phi)$. Together with the inversion formula for the Fourier transform this Paley-Wiener estimate then leads to the support theorem (Theorem \ref{SuppThm thm Support Theorem for functions and minimal P_0}) for $\Rt[P_{0}]$. All this is done in Section \ref{section Support theorem for the horospherical transform}.
In Section \ref{section Support theorems} we generalize this (in Theorem
\ref{SuppThm thm Support theorem for distributions and non-minimal
P}) to a support theorem for the Radon transform $\Rt[P]$ on distributions for $P$ an arbitrary
$\sigma\circ\theta$-stable parabolic subgroup.

\subsection*{Acknowledgements}
This article grew out of my PhD-research project supervised by
E.P.\,van den Ban. The project was financially supported by the
Van Beuningen--Peterich Foundation.

I wish to express my deep gratitude to Erik van den Ban. I
benefited tremendously from his many critical questions, the long
discussions and his suggestions for improvements.

\section{Notation and preliminaries}
\label{section Notation and Preliminaries}

In this section we recall some basic theory of Radon transformation, reductive symmetric spaces and parabolic subgroups, and we fix the notation that we use throughout the article.
The theory discussed in this section can be found in \cite[part 1, Chapter 3 \& 6]{vdBanSchlichtkrullDelorme_LieTheory}, \cite[Chapter II]{Helgason_IntegralGeometryAndRadonTransforms}, \cite{Knapp_LieGroups} and \cite[II.1 \& II.6]{Varadarajan_HarmonicAnalysisOnRealReductiveGroups}.

\subsection{Double fibrations}
\label{subsection Introduction to double fibrations}
In the theory of Radon transformation, as introduced by Helgason in \cite{Helgason_ADualityInIntegralGeometryOnSymmetricSpaces},
one considers a double fibration of homogeneous spaces
\begin{equation}\label{RT eq double fibration}
\xymatrix@C+1pc@R-1pc{
&\ar[dl]_{\Pi_{X}}Z=G/(S\cap T)\ar[dr]^{\Pi_{\Xi}}&\\
G/S=X && \Xi=G/T}
\end{equation}
where $G$ is a Lie group, $S$ and $T$ are two closed subgroups of
$G$ and $\Pi_{X}$ and $\Pi_{\Xi}$ are the canonical projections.

A double fibration defines an incidence relation: $x\in\ X$ and
$\xi\in\Xi$ are said to be incident if $\Pi_{X}^{-1}(\{x\})\cap\Pi_{\Xi}^{-1}(\{\xi\})\neq\emptyset$. Note
that $x$ and $\xi$ are incident if and only if
$\xi\in\Pi_{\Xi}(\Pi_{X}^{-1}(x))$, or equivalently,
$x\in\Pi_{X}(\Pi_{\Xi}^{-1}(\xi))$. If the set-valued maps $X\ni x\mapsto\Pi_{\Xi}(\Pi_{X}^{-1}(x))$ and $\Xi\ni\xi\mapsto\Pi_{X}(\Pi_{\Xi}^{-1}(\xi))$
are both injective, then following Helgason, we say that $S$ and
$T$ are transversal.

We make the following assumptions.
\begin{itemize}
\item[(A)]
    There exist non-zero Radon measures $d_{S\cap T}s$ and $d_{S\cap T}t$ on
    $S/(S\cap T)$ and $T/(S\cap T)$ invariant under $S$ and $T$,
    respectively.
\item[(B)]
    There exist non-zero $G$-invariant Radon measures $dx$ and $d\xi$ on $X$ and
    $\Xi$.
\item[(C)]
    $ST$ is a closed subset of $G$
\end{itemize}

\subsection{Radon transforms}
\label{subsection Introduction to Radon transforms}
Following Schwartz, we denote spaces of compactly supported smooth functions and spaces of smooth functions by $\Ds$ and $\Es$, respectively. Spaces of distributions and spaces of compactly supported distributions we denote by $\Ds'$ and $\Es'$, respectively.

The Radon transforms for the double
fibration (\ref{RT eq double fibration}) are defined to be the $G$-equivariant continuous operators
$\Rt:\Ds(X)\to\Es(\Xi)$ and $\dRt:\Ds(\Xi)\to\Es(X)$
given by
$$
\Rt\phi(g\cdot T) =\int_{T/(S\cap T)}\phi(gt\cdot S)\,d_{S\cap T}t
\quad\text{and}\quad
\dRt\psi(g\cdot S) =\int_{S/(S\cap T)}\psi(gs\cdot T)\,d_{S\cap T}s,
$$
respectively. (See \cite[Section II.2]{Helgason_IntegralGeometryAndRadonTransforms}.)

The transforms $\dRt$ and $\Rt$ are dual to each other in the sense that for
all $\phi\in \Ds(X)$ and $\psi\in \Ds(\Xi)$
\begin{equation}\label{RT eq dualiteit R R*}
\int_{\Xi}\Rt \phi(\xi)\,\psi(\xi)\,d\xi
    =\int_{X} \phi(x)\,\dRt\psi(x)\,dx,
\end{equation}
assuming the measures are suitably normalized. This duality allows
to extend the Radon transform to the space of compactly supported
distributions: for $\mu\in\Es'(X)$ and $\nu\in\Es'(\Xi)$, we
define $\Rt\mu\in\Ds'(\Xi)$ and $\dRt\nu\in\Ds'(X)$ to be the
distributions
\begin{equation}\label{RT eq def distributions}
\Rt\mu:\Ds(\Xi)\ni\psi\mapsto\mu(\dRt\psi) \quad\text{and}\quad
\dRt\nu:\Ds(X)\ni\phi\mapsto\nu(\Rt\phi).
\end{equation}

\subsection{Reductive Symmetric spaces}
\label{subsection reductive symmetric spaces}
Let $G$ be a reductive Lie group of the Harish-Chandra class, let $\sigma$ be an involution of $G$ and let $X$ be the symmetric space $G/H$, where $H$ is an open subgroup of $G^{\sigma}$. We denote the corresponding involution on $\g$
by $\sigma$ as well.

Let $\theta$ be a Cartan involution commuting with $\sigma$ and let $K$ be the fixed point subgroup $G^{\theta}$. The corresponding involution of the Lie algebra $\g$ we denote by $\theta$ too.
The Lie algebra of $\g$ decomposes as a direct sum of vector spaces $\g=\g[k]\oplus\g[p]=\g[h]\oplus\g[q]$
of eigenspaces for $\theta$ and $\sigma$ respectively. Here the first component is the $+1$ and the second
the $-1$ eigenspace. Note that $\g[k]$ is the Lie algebra of $K$ and $\g[h]$ is the Lie algebra of $H$.
Since $\sigma$ and $\theta$ commute we also have
$$\g
=(\g[k]\cap\g[h])\oplus(\g[k]\cap\g[q])\oplus(\g[p]\cap\g[h])\oplus(\g[p]\cap\g[q]).$$
We fix a maximal abelian subspace $\g[a]_{\g[q]}$ of
$\g[p]\cap\g[q]$. The subgroup $H$ is called essentially connected
if $H=\Cen_{K\cap H}(\g[a]_{\g[q]})H^{0}$,
where $H^0$ is the identity component of $H$ and $\Cen_{K\cap
H}(\g[a]_{\g[q]})$ the centralizer of $\g[a]_{\g[q]}$ in $K\cap H$.
(See \cite[p.
24]{vdBan_ConvexityTheoremForSemisimpleSymmetricSpaces}.) If $H$ is essentially connected, then $X=G/H$ is
called a reductive symmetric space of the Harish-Chandra class, or
for short a reductive symmetric space.

Examples of reductive symmetric spaces are spheres, Euclidean
spaces, pseudo-Riemannian hyperbolic spaces and the De Sitter and the anti De
Sitter space. Also a Lie group $G$ of the Harish-Chandra
class may be viewed as a reductive symmetric space. In fact
$$(G\times G)/\diag(G)\to G;\qquad (g_{1},g_{2})\cdot\diag(G)\mapsto g_{1}g_{2}^{-1}$$
is a diffeomorphism and $\diag(G)$ is the fixed point subgroup of
the involution $(g_{1},g_{2})\mapsto(g_{2},g_{1})$ of $G$.
Therefore $(G\times G)/\diag(G)$ is a symmetric space. Since
$\Cen_{K}(\g[a])$ meets every connected component of $G$ (see
\cite[7.33]{Knapp_LieGroups}), the subgroup $\diag(G)$ of $G\times
G$ is essentially connected, hence $(G\times G)/\diag(G)$ is of
the Harish-Chandra class. Another important class of symmetric
spaces is the class of Riemannian symmetric spaces of non-compact type. These symmetric
spaces are obtained by taking $G$ to be a non-compact connected semi-simple Lie group and  $\sigma$
to be a Cartan involution of $G$.

From now on we will always assume that $X=G/H$ is a reductive symmetric space.

\subsection{Parabolic subgroups}
If $P$ is a parabolic subgroup of $G$, we write
$P=M_{P}A_{P}N_{P}$ for its Langlands decomposition such that $A_{P}$ is $\theta$-stable and we write $L_{P}$
for the $\theta$-stable reductive component $M_{P}A_{P}$ of $P$.
The corresponding decompositions of the  Lie algebra of $P$ are
denoted by $\g[m]_{P}\oplus\g[a]_{P}\oplus\g[n]_{P}$ and
$\g[l]_{P}\oplus\g[n]_{P}$, respectively.

We recall the following well known results.

\begin{lemma}\
\begin{enumerate}[(i)]
 \item If $P$ is a parabolic subgroup of $G$, then $L_{P}$ is a reductive Lie group of the Harish-Chandra class.
 \item Let $P$ and $Q$ be two parabolic subgroups of $G$. If $P\subseteq Q$, then
 $L_{P}\subseteq L_{Q}$, $A_{Q}\subseteq A_{P}$ and $N_{Q}\subseteq N_{P}$.
\end{enumerate}
\end{lemma}

Assume that $P$ and $Q$ are parabolic subgroups of $G$ and
$P\subseteq Q$.  We write $N_{P}^{Q}$ for $N_{P}\cap L_{Q}$. The Lie algebra of $N_{P}^{Q}$ is
denoted by $\g[n]_{P}^{Q}$. Note that $N_{Q}$ is a normal subgroup of $N_{P}$. In fact
$N_{P}= N_{Q}\rtimes N_{P}^{Q}$
and hence the multiplication map $N_{Q}\times N_{P}^{Q}\to N_{P}$ is a diffeomorphism.

\subsection{The class of $\sigma\circ\theta$-stable parabolic subgroups}
\label{subsection sigma circ theta-stable parabolic subgroups}
In this section we describe the $\sigma\circ\theta$-stable parabolic subgroups, i.e., the parabolic subgroups $P$ with the property that $\sigma\circ\theta(P)=P$.

\medbreak

Fix a maximal abelian subspace $\g[a]_{\g[q]}$ of
$\g[p]\cap\g[q]$. Let $\g[a]$ be a maximal abelian subspace of $\g[p]$ containing $\g[a]_{\g[q]}$. Then $\g[a]=(\g[a]\cap\g[h])\oplus\g[a]_{\g[q]}$. Denote by $A_{\g[q]}$ respectively $A$ the corresponding analytic subgroups of $G$.

Let $\Sigma(\g,\g[a]_{\g[q]})$ be the set of roots of $\g$ in $\g[a]_{\g[q]}^{*}$. For each $\alpha\in\Sigma(\g,\g[a]_{\g[q]})$ we define the root space
$$
\g_{\alpha}=\{Y\in\g: \ad(Z)Y=\alpha(Z)Y\text{ for all }Z\in\g[a]_{\g[q]}\}.
$$
Let $\Sigma^{+}(\g,\g[a]_{\g[q]})$ be a choice of a positive system for $\Sigma(\g,\g[a]_{\g[q]})$. Let $\g[n]_{0}=\bigoplus_{\alpha\in\Sigma^{+}(\g,\g[a]_{\g[q]})}\g_{\alpha}$ and let $N_{0}=\exp(\g[n]_{0})$. Then $\Cen_{G}(\g[a]_{\g[q]})$ normalizes $N_{0}$ and $P_{0}=\Cen_{G}(\g[a]_{\g[q]})N_{0}$ is a closed subgroup of $G$. In fact it is a minimal $\sigma\circ\theta$-stable parabolic subgroup. Every minimal $\sigma\circ\theta$-stable parabolic subgroup of $G$ is conjugate to $P_{0}$ via an element of $K$.

If $P_{0}$ is a minimal $\sigma\circ\theta$-stable parabolic subgroup, then $\g[a]_{P_{0}}\cap\g[q]$ is a maximal abelian subspace of $\g[p]\cap \g[q]$.
Let $\Sigma(\g,\g[a]_{P_{0}}\cap\g[q];P_{0})=\{\alpha\in\Sigma(\g,\g[a]_{P_{0}}\cap\g[q]):\g_{\alpha}\in\g[n]_{P_{0}}\}$. Then $\g[n]_{P_{0}}=\bigoplus_{\alpha\in\Sigma(\g,\g[a]_{P_{0}}\cap\g[q];P_{0})}\g_{\alpha}$. Furthermore, $\Sigma(\g,\g[a]_{P_{0}}\cap\g[q];P_{0})$ is a positive system for $\Sigma(\g,\g[a]_{P_{0}}\cap\g[q])$.

The following results are well known. (See for example \cite[first part, Chapter 6]{vdBanSchlichtkrullDelorme_LieTheory}.)
\begin{lemma}
Let $P$ be a $\sigma\circ\theta$-stable parabolic subgroup of $G$.
\begin{enumerate}[(i)]
 \item $L_{P}$ equals the centralizer of  $\g[a]_{P}\cap \g[q]$ in $G$.
 \item $A_{\g[q]}\subseteq P$ if and only if $\g[a]_{P}\cap\g[q]\subseteq\g[a]_{\g[q]}$.
 \item Assume $P_{0}$ is a minimal $\sigma\circ\theta$-stable parabolic subgroup contained in $P$ and $A_{\g[q]}\subseteq P_{0}$. Then
 $$
 \g[n]_{P}
 =\bigoplus_{\genfrac{}{}{0pt}{}{\alpha\in\Sigma(\g,\g[a]_{\g[q]};P_{0})}
    {\alpha|_{\g[a]_{P}\cap \g[a]_{\g[q]}}\neq 0}}
    \g_{\alpha}.
 $$
\end{enumerate}
\end{lemma}

We write $\Psg(\g[a]_{\g[q]})$ for
the collection of $\sigma\circ \theta$-stable parabolic subgroups
$P$ of $G$ with $\g[a]_{P}\cap\g[a]_{\g[q]}=\g[a]_{\g[q]}$. Note
that $\Psg(\g[a]_{\g[q]})$ consist of the minimal
$\sigma\circ\theta$-stable parabolic subgroups containing
$A_{\g[q]}$.

\begin{lemma}\label{PSubgrp lemma n_MAN^(P_0)=n cap h}
Let $P_{m}$ be a minimal parabolic subgroup containing $A$ and let $P_{0}$ be a minimal
$\sigma\circ\theta$-stable parabolic subgroup containing $P_{m}$.
Then
$$\g[n]_{P_{m}}^{P_{0}}=\g[n]_{P_{m}}\cap \g[h]\quad \textnormal{and}\quad N_{P_{m}}^{P_{0}}=N_{P_{m}}\cap H.$$
\end{lemma}

\begin{proof}
We first note that
$$
\g[n]_{P_{m}} \cap \g[h] \subseteq
\g[n]_{P_{m}} \cap \sigma(\g[n]_{P_{m}}) =
\g[n]_{P_{m}}\cap \sigma(\g[n]_{P_{0}}\oplus\g[n]_{P_{m}}^{P_{0}}).
$$
As $\sigma(\g[n]_{P_{0}})=\theta \g[n]_{P_{0}}$ and $\sigma(\g[n]_{P_{m}}^{P_{0}}) \subseteq \sigma(\g[l]_{P_{0}}) = \g[l]_{P_{0}}$,
it follows that $\g[n]_{P_{m}}\cap\g[h] \subseteq\g[n]_{P_{m}}^{P_{0}}$.

Since $\g[n]_{P_{m}}^{P_{0}}$ is a direct sum of root spaces for roots in $\g[a]$ that vanish on $\g[a]_{\g[q]}$, $\g[n]_{P_{m}}^{P_{0}}$ is $\sigma$-stable. Moreover, it follows that $\g[n]_{P_{m}}^{P_{0}}\subseteq\g[h]$ if and only if $\big(\g[n]_{P_{m}}^{P_{0}}\oplus\theta\g[n]_{P_{m}}^{P_{0}}\big)\cap \g[p]\subseteq\g[h]$.
From the maximality of $\g[a]_{\g[q]}$ it follows that $\g[m]_{P_{0}}\cap\g[p]\subseteq\g[h]$. Therefore
$\big(\g[n]_{P_{m}}^{P_{0}}\oplus\theta\g[n]_{P_{m}}^{P_{0}}\big)\cap\g[p]\subseteq\g[m]_{P_{0}}\cap\g[p] \subseteq\g[h]$. This proves the first identity.

Since $\exp:\g[n]_{P_{m}}\to N_{P_{m}}$ is a diffeomorphism, it follows that $N_{P_{m}} \cap H \subseteq N_{P_{m}}\cap G^{\sigma}=\exp(\g[n]_{P_{m}}\cap \g[h])$, hence $N_{P_{m}}\cap H=\exp(\g[n]_{P_{m}}\cap\g[h])$. The second identity now follows from the first.
\end{proof}

If $P$ is a $\sigma\circ\theta$-stable parabolic subgroup containing $A_{\g[q]}$, then
$$L_{P}=(M_{P}\cap K)A_{\g[q]}(L_{P}\cap H).$$
Furthermore, if $Y,Y_{0}\in \g[a]_{\g[q]}$ and $\exp(Y)\in(M_{P}\cap K)\exp(Y_{0})(L_{P}\cap H)$, then $Y=\Ad(k)Y_{0}$ for some $k\in\Nor_{M_{P}\cap K\cap H}(\g[a]_{\g[q]})$. This is the polar decomposition of $L_{P}$. Note that in particular $G$ is a $\sigma\circ\theta$-stable parabolic subgroup and the polar decomposition of $G$ is $G=KA_{\g[q]}H$.

To conclude this section, we describe a generalization of the Iwasawa decomposition. Let $P_{0}$ be a minimal $\sigma\circ\theta$-stable parabolic subgroup. Then the double coset space $P_{0}\setminus G/H$ is finite. Furthermore, the sets $P_{0}wH$ with $w\in\Nor_{K}(\g[a]_{\g[q]})$ are open subsets of $G$. In fact, these are all the open $P_{0}\times H$-orbits in $G$. The union of these open orbits is dense in $G$. Finally the map
$$
N_{P_{0}}\times A_{\g[q]}\times(M_{P_{0}}\cap K)\times_{M_{P_{0}}\cap K\cap H} H\to P_{0}H
$$
is a diffeomorphism onto an open subset of $G$.

\subsection{Notation}
\label{subsection Notation}
We recall that $\Sigma(\g,\g[a]_{\g[q]})$ is a possibly unreduced root system.
We write $\Sigma_{\pm}(\g,\g[a]_{\g[q]};P)$
for the set of positive roots $\alpha\in\Sigma(\g,\g[a]_{\g[q]};P)$
for which the $\pm 1$-eigenspace of $\sigma\circ\theta$ in the
root space $\g_{\alpha}$ is non-trivial.
We denote by $\WG$ the Weyl group of
the root system $\Sigma(\g,\g[a]_{\g[q]})$. Note that there is a
natural isomorphism $\WG \simeq\Nor_{K}(\g[a]_{\g[q]})/\Cen_{K}(\g[a]_{\g[q]})$.
If $S$ is a subgroup of $G$, then we define $\WG[S]$ to be the
subgroup consisting of elements that can be realized as ${\rm
Ad}(s)|_{\g[a]_{\g[q]}}$ for $s \in \Nor_{S}(\g[a]_{\g[q]})$. We write $\WGs$ for a set of representatives
in $\Nor_K(\g[a]_{\g[q]})$ of $\WG/\WG[K\cap H]$.
The Weyl group of
$\Sigma(\g[l]_{P},\g[a]_{\g[q]})$ equals $\WG[M_{P}\cap K]$. We write $\WGs[M_{P}]$
for a set of representatives in $\Nor_{M_{P}\cap K}(\g[a]_{\g[q]})$ for $\WG[M_{P}\cap K]/\WG[M_{P}\cap K\cap H]$.

If $P$ is a $\sigma\circ\theta$-stable parabolic subgroup containing $A_{\g[q]}$, then by $\g[a]^{+}_{\g[q]}(P)$ we denote the set
consisting of elements $H\in\g[a]_{\g[q]}$ such that $\alpha(H)>0$
for all $\alpha\in\Sigma(\g,\g[a]_{\g[q]};P)$. Furthermore, for $R\in \R$, we define
$$\g[a]_{\g[q]}^{*}(P,R)=
\{\lambda\in\g[a]_{\g[q],\C}^{*}:\Re\,\lambda(H_{\alpha})<R\text{ for all }\alpha\in\Sigma(\g,\g[a]_{\g[q]};P)\}.$$

Let $B$ be an $\Ad(G)$-invariant and $\theta$-invariant bilinear form on $\g$ such that $B$ agrees with the Killing form on $[\g,\g]$ and $-B(\cdot,\theta\cdot)$ is an inner product on $\g$. Thus $\g[k]\perp\g[p]$. We further assume that $\Cen_{\g}\perp[\g,\g]$.

For a root $\alpha\in\Sigma(\g,\g[a]_{\g[q]})$, we define
$H_{\alpha}\in\g[a]_{\g[q]}$ to be the element given by
\begin{equation}\label{PSubgrp eq alpha(Y)=B(H_alpha,Y)}
\alpha(Y)=B(H_{\alpha},Y)\qquad(Y\in\g[a]_{\g[q]}).
\end{equation}

For $g\in G$ and $P$ a parabolic subgroup of $G$ we denote the parabolic subgroup $g^{-1}Pg$ by $P^{g}$.

Finally if $V$ is a Fr\'echet space and $(\pi,V)$ a continuous
representation of $G$ on $V$, then we denote the space of smooth
vectors for $\pi$ by $V^{\infty}$.

\section{Radon transformation on a reductive symmetric space}
\label{section Radon transformation on a reductive symmetric space}
In Sections \ref{subsection Horospheres} -- \ref{subsection Radon transforms} we first introduce for each $\sigma\circ\theta$-stable parabolic subgroup $P$ a homogeneous space $\Xi_{P}$. For a pair of
$\sigma\circ\theta$-stable parabolic subgroups $P$ and $Q$, with
$P\subseteq Q$, we then introduce a class of closed submanifolds (related to $P$) of $\Xi_{Q}$ and
describe the Radon transforms that are obtained by integrating over these submanifolds.
Then in Section \ref{subsection Extensions of the Radon transforms} some basic
estimates are derived. It follows from these estimates that the
Radon transforms, which are initially defined on the space of
compactly supported smooth functions, extend to larger spaces of
functions and distributions that are defined in Section \ref{subsection
Spaces of functions and distributions}. In Section \ref{subsection
Relations between Radon transforms} we describe some relations
between Radon transforms related to different parabolic subgroups.
\subsection{Horospheres}
\label{subsection Horospheres}
For a $\sigma\circ\theta$-stable parabolic subgroup $P$, we define
$$\Xi_{P}=G/(L_{P}\cap H)N_{P}.$$

Since $(L_{P}\cap H)N_{P}$ is closed in $P$, hence in $G$, it follows that $\Xi_{P}$ is a smooth homogeneous space for $G$.
Note that $G$  is a $\sigma\circ\theta$-stable parabolic subgroup of $G$ and $\Xi_{G}=X$.
The cosets $e\cdot H$ and $e\cdot (L_{P}\cap H)N_{P}$ are denoted
by $x_{0}$ and $\xi_{P}$, respectively.

\medbreak

Let $P_{0}$ be a
minimal $\sigma\circ\theta$-stable parabolic subgroup of $G$.
A horosphere in $X$ is an orbit of a subgroup of $G$ conjugate to $N_{P_{0}}$ in $X$ of maximal dimension, i.e., a submanifold of $X$ (see Proposition \ref{RadTrans prop N_P cdot xi_Q closed submanifold}) of the form
$g_{1}N_{P_{0}}g_{2}\cdot x_{0}$ with dimension equal to the dimension of
$N_{P_{0}}$. The set of all horospheres in $X$ is denoted by
$\Hor(X)$.

According to \cite[Theorem
13]{Rossmann_StructureOfSemisimpleSymmetricSpaces} and
\cite{Matsuki_TheOrbitsOfAffineSymmetricSpacesUnderTheActionOfMinimalParabolicSubgroups}
$G$ equals the union of subsets $P_{0}kH$, where $k$ runs over a
finite subset of $K$. This implies that $\Hor(X)$ is the union of
finitely many $G$-orbits. The dimension of these orbits need not
be all equal. It follows from the same theorem in
\cite{Rossmann_StructureOfSemisimpleSymmetricSpaces} that the set
of orbits of maximal dimension is in bijection with $\WGs$ via the
map $w\mapsto G\cdot(N_{P_{0}}w\cdot
x_{0})=G\cdot(w^{-1}N_{P_{0}^{w}}\cdot x_{0})$. Here the
superscript ${}^{w}$ denotes conjugation with $w^{-1}$.

The stabilizer in $G$ of $\xi_{P_{0}}$ equals $(L_{P_{0}}\cap
H)N_{P_{0}}$. (See Proposition \ref{Transv prop horosphere
stabilizer} in Appendix B.) Therefore the $G$-orbits in $\Hor(X)$ of maximal
dimension, i.e., the sets $G\cdot\xi_{P_{0}^{w}}$ for $w\in\WGs$, are
parametrized by the homogeneous spaces $\Xi_{P_{0}^{w}}$.

The double fibration
\begin{equation}\label{Horospheres eq double fibration}
\xymatrix{
    &G/(L_{P_{0}}\cap H)\ar[ld]_{\Pi_{X}}\ar[rd]^{\Pi_{\Xi_{P_{0}}}}  &\\
**[l]X  &               &**[r]\Xi_{P_{0}}
    }
\end{equation}
describes the incidence relation between points in $X$ and
horospheres in $X$ which are parametrized by $\Xi_{P_{0}}$: a point $x$ is
contained in a horosphere parametrized by $\xi\in\Xi_{P_{0}}$ if
and only if $x\in\Pi_{X}\big(\Pi_{\Xi_{P_{0}}}^{-1}(\xi)\big)$.

If $\sigma=\theta$, then the double fibration (\ref{Horospheres eq double fibration}) and the corresponding Radon transform reduce to the double fibration and the Radon transform considered by Helgason.

\subsection{Double fibration}
\label{subsection Double fibration}
In this and the following sections we will assume that $P$ and $Q$ are $\sigma\circ\theta$-stable parabolic subgroups
such that $A \subseteq P\subseteq Q$.
Since $L_{P}\cap H$ is a closed subgroup of $L_{Q}$, it follows that $(L_{P} \cap H) N_{Q}$ is a closed subgroup of $G$.  We recall that $\Xi_{P} = G/(L_{P} \cap H)N_{P}$
and $\Xi_{Q} = G/(L_{Q}\cap H)N_{Q}$ and consider the following generalization of (\ref{Horospheres eq double fibration}).

\begin{equation}\label{RadTrans eq double fibration}
\xymatrix{
    &G/(L_{P}\cap H)N_{Q}\ar[ld]_{\Pi_{\Xi_{Q}}}\ar[rd]^{\Pi_{\Xi_{P}}}  &\\
**[l]\Xi_{Q}  &               &**[r]\Xi_{P}
    }
\end{equation}
Note that for $Q=G$ and $P=P_{0}\in\Psg(\g[a]_{\g[q]})$, (\ref{RadTrans eq double fibration}) reduces to (\ref{Horospheres eq double fibration}).

In view of  the following proposition, (\ref{RadTrans eq double fibration}) is a double fibration of the form (\ref{RT eq double fibration}).

\begin{prop}
Let $P, Q$ be $\sigma\circ\theta$-stable parabolic subgroups with $P\subseteq Q$. Then
$$
(L_{Q}\cap H)N_{Q}\cap(L_{P}\cap H)N_{P}=(L_{P}\cap H)N_{Q}.
$$
\end{prop}

\begin{proof}
As $L_{P}\subseteq L_{Q}$ and $N_{Q}\subseteq N_{P}$, it follows that the set on the right-hand side is contained
in the one on the left-hand side. We turn to the converse inclusion. Assume that $g$ belongs to the intersection on the
left-hand side. Using that $N_{P}=N_{P}^{Q}N_{Q}$, that $N_{P}^{Q}\subseteq L_{Q}$ and that $L_{P} \subseteq L_{Q}$ we see that $g=ln$
for certain elements $n \in N_{Q}$ and $l \in (L_{Q} \cap H) \cap (L_{P} \cap H)N_{P}^{Q}$.
Since $P$ and $Q$ are $\sigma\circ \theta$-stable, $\sigma(N_{P}^{Q}) = \theta N_{P}^{Q}$, so that $N_{P}^{Q}\cap H=\{e\}$. Therefore, $l \in L_{P} \cap H$.
\end{proof}

The double fibration (\ref{RadTrans eq double fibration}) describes the incidence relations between points in $\Xi_{Q}$ and
subsets of $\Xi_{Q}$ of the form $g N_{P}\cdot \xi_{Q}$.
For $\xi\in\Xi_{P}$ we define $E_{P}^{Q}(\xi)$ to be the subset of $\Xi_{Q}$ given by
$$
E_{P}^{Q}(\xi)=\Pi_{\Xi_{Q}}\big(\Pi_{\Xi_{P}}^{-1}(\{\xi\})\big).
$$
Moreover, we agree to write $E_{P}(\xi)$ for $E_{P}^{G}(\xi)$.
Note that for $g \in G$, we have $E_{P}^{Q}(g\cdot\xi_{P})=g\cdot E_{P}^{Q}(\xi_{P})=gN_{P}\cdot\xi_{Q}=gN_{P}^{Q}\cdot\xi_{Q}$.

According to Corollary \ref{Transv cor (L_Q cap H)N_Q and (L_P cap
H)N_P transversal} of Appendix B, the groups $(L_{P}\cap
H)N_{P}$ and $(L_{Q}\cap H)N_{Q}$ are transversal. We recall from
Section \ref{subsection Introduction to double fibrations} that this means in
particular that the map $\xi\mapsto E_{P}^{Q}(\xi)$ is injective.

\begin{prop}\label{RadTrans prop N_P cdot xi_Q closed submanifold}
Let $g \in  G$. Then we have the following.
\begin{enumerate}[(i)]
\item
$E_{P}^{Q}(g\cdot\xi_{P})$ is a closed submanifold of $\Xi_{Q}$.
\item
The map $n \mapsto gn \cdot \xi_{Q}$ is a diffeomorphism from $N_{P}^{Q}$ onto $E_{P}^{Q}(g\cdot\xi_{P})$.
\end{enumerate}
\end{prop}

\begin{proof}
Without loss of generality we may assume that $g =e$.
The map
\begin{equation}\label{RadTrans eq K times_(M_Q cap K) L_Q/(L_Q cap H) to Xi_Q}
K\times_{(K\cap M_{Q})}L_{Q}/(L_{Q}\cap H)\to\Xi_{Q};
\quad\big(k,l\cdot(L_{Q}\cap H)\big)\mapsto kl\cdot\xi_{Q}
\end{equation}
is a diffeomorphism. Hence, $l \mapsto l \cdot \xi_{Q}$ is a diffeomorphism from $L_{Q}/(L_{Q}\cap H)$ onto the closed
submanifold $L_{Q}\cdot\xi_{Q}$ of $\Xi_{Q}$.
Let now $P_{0}$ be a minimal $\sigma\circ\theta$-stable parabolic subgroup contained in $P$.
From \cite[Lemma 3.4]{vdBan_ConvexityTheoremForSemisimpleSymmetricSpaces} applied to $L_{Q}$ and the minimal $\sigma\circ\theta$-stable parabolic subgroup $P_{0}\cap L_{Q}$ of $L_{Q}$, it follows that
the multiplication map
\begin{equation}\label{RadTrans eq N_P_0^Q x L_P_0 x (L_Q cap H) to L_Q}
N_{P_{0}}^{Q}\times L_{P_{0}}\times_{(L_{P_{0}}\cap H)}(L_{Q}\cap H)\to L_{Q}
\end{equation}
is a diffeomorphism onto the open subset $(P_{0}\cap L_{Q})(L_{Q}\cap H)$ of $L_{Q}$.
Therefore
$$
N_{P_{0}}^{Q}\times L_{P_{0}}/(L_{P_{0}}\cap H)\to L_{Q}\cdot\xi_{Q}
$$
is a diffeomorphism onto the open subset $(P_{0}\cap L_{Q})\cdot\xi_{Q}$ of the submanifold $L_{Q}\cdot\xi_{Q}$ of $\Xi_{Q}$. As the multiplication map $N_{P_{0}}^{P}\times N_{P}^{Q}\to N_{P_{0}}^{Q}$
is a diffeomorphism, the map $N_{P}^{Q}\ni n\mapsto n\cdot \xi_{Q}$ is a diffeomorphism onto $N_{P}^{Q}\cdot\xi_{Q}$ and this set is a submanifold of $L_{Q}\cdot\xi_{Q}$ and therefore a submanifold of $\Xi_{Q}$. Furthermore, it follows that $N_{P}^{Q}\cdot\xi_{Q}$ is closed in $(P_{0}\cap L_{Q})\cdot\xi_{Q}$ equipped with the subspace topology. It remains to be shown that $N_{P}^{Q}\cdot\xi_{Q}$ is closed in $\Xi_{Q}$.

Let $(n_{j})_{j\in \N}$, $(a_{j})_{j\in\N}$ and $(m_{j})_{j\in\N}$
be sequences in $N_{P_{0}}^{Q}$, $A_{\g[q]}$ and $M_{P_{0}}$,
respectively, such that $(n_{j}a_{j}m_{j}\cdot\xi_{Q})_{j\in\N}$
converges to a point in the boundary of $(P_{0}\cap
L_{Q})\cdot\xi_{Q}$. Then, in view of \cite[Lemma
3.4]{vdBan_ConvexityTheoremForSemisimpleSymmetricSpaces}, the set
$\{a_{j}:j\in\N\}$ is not relatively compact in $A_{\g[q]}$. It
follows that if $(n_{j}\cdot\xi_{Q})_{j\in\N}$ is a sequence in
$N_{P}^{Q}\cdot\xi_{Q}$ converging to $\xi$ in $\Xi_{Q}$, then
$\xi$ cannot be an element of the boundary of $(P_{0}\cap
L_{Q})\cdot\xi_{Q}$, hence $\xi\in (P_{0}\cap L_{Q})\cdot\xi_{Q}$.
Since $N_{P}^{Q}\cdot\xi_{Q}$ is closed in $(P_{0}\cap
L_{Q})\cdot\xi_{Q}$, we conclude that $\xi\in
N_{P}^{Q}\cdot\xi_{Q}$. This proves that $N_{P}^{Q}\cdot\xi_{Q}$
is closed in $\Xi_{Q}$.
\end{proof}

\begin{cor}\label{RadTrans cor N_P(L_Q cap H) closed}
$(L_{Q}\cap H)N_{Q}(L_{P}\cap H)N_{P}=(L_{Q}\cap H)N_{P}$ is a closed submanifold of $G$.
\end{cor}

\begin{proof}
Since $L_{P}\cap H$ normalizes $N_{Q}$,
$$
(L_{Q}\cap H)N_{Q}(L_{P}\cap H)N_{P}
=(L_{Q}\cap H)(L_{P}\cap H)N_{Q}N_{P}
=(L_{Q}\cap H)N_{P}.
$$
The set
$N_{P}(L_{Q}\cap H)$ equals the pre-image of $E_{P}^{Q}(\xi_{P})$ under the projection $G\to\Xi_{Q}$. Since $E_{P}^{Q}(\xi_{P})$ is a closed submanifold by Proposition \ref{RadTrans prop N_P cdot xi_Q closed submanifold} and $G$ is a fiber bundle over $\Xi_{Q}$, it follows that $N_{P}(L_{Q}\cap H)$ is a closed submanifold of $G$. The same holds for its image under the map $g \mapsto g^{-1}$, which is
$(L_{Q}\cap H)N_{P}$.
\end{proof}

As a consequence of Corollary \ref{RadTrans cor N_P(L_Q cap H)
closed}, the double fibration (\ref{RadTrans eq double fibration}) satisfies condition (C) of Section \ref{subsection Introduction to Radon transforms}.

\subsection{Invariant measures}
\label{subsection Invariant measures}
In this section we show that properties (A) and (B) of Section \ref{subsection Introduction to
Radon transforms} are satisfied for the double fibration (\ref{RadTrans eq double fibration}).
We retain the notation of the previous sections. In particular, $P$ and $Q$ are $\sigma\circ\theta$-stable parabolic subgroups $A\subseteq P \subseteq Q$.

\begin{prop}\label{Measures prop H, L_P cap H, (L_P cap H)N_Q unimodular}
The group $(L_{P}\cap H)N_{Q}$ is unimodular.
\end{prop}

\begin{proof}
Let $\Delta$ be the modular function of $(L_{P}\cap H)N_{Q}$.
Note that
$$(L_{P}\cap H)N_{Q}=(L_{P}\cap H\cap K)(L_{P}\cap H)^{0}N_{Q},$$
where $(L_{P}\cap H)^{0}$ denotes the identity component of
$(L_{P}\cap H)$. Let $k\in (L_{P}\cap H\cap K)$, $l\in(L_{P}\cap H)^{0}$ and $n\in N_{Q}$. Then
$\Delta(kln)=\Delta(l)$ because $(L_{P}\cap H\cap K)$ is compact and $\ad(\log n)$ acts upper triangular with respect to the usual basis of $(\g[l]_{P}\cap\g[h])\oplus\g[n]_{Q}$.
The group $(L_{P}\cap H)^{0}$ is reductive and normalizes both $\g[l]_{P}\cap \g[h]$ and
$\g[n]_{Q}$.  This implies that
$\Delta(l)=\big|\det\big(\Ad(l)|_{\g[n]_{Q}}\big)\big|$.

Let $m\in M_{Q}\cap H$ and $a\in A_{Q}\cap H$ be such that $l=ma$. Then $\big|\det\big(\Ad(ma)|_{\g[n]_{Q}}\big)\big|=a^{2\rho_{Q}}$, where $\rho_{Q}$ is
the element of $\g[a]_{Q}^{*}$ given by
\begin{equation}\label{Measures eq def rho_Q}
\rho_{Q}(Y)=\frac{1}{2}\tr(\ad(Y)\big|_{\g[n]_{Q}}), \qquad (Y\in\g[a]_{Q}).
\end{equation}
Since $Q$ is $\sigma\circ\theta$-stable, $\rho_{Q}|_{\g[a]_{Q}\cap\g[h]} = 0$, so that $a^{2\rho_{Q}} = 1$.
Therefore, $\Delta(l)=1$ and we conclude that $(L_{P}\cap H)N_{Q}$
is unimodular.
\end{proof}

\begin{cor}\label{Measures prop existence of inv measures}\
 \begin{enumerate}[(i)]
 \item
 There exists a non-zero $G$-invariant Radon measure on $G/(L_{P}\cap H)N_{Q}$. In particular there exists a non-zero $G$-invariant Radon measure on $\Xi_{P}$.
\item
There exists a non-zero $(L_{Q}\cap H)N_{Q}$-invariant Radon measure on $(L_{Q}\cap H)N_{Q}/(L_{P}\cap H)N_{Q}$.
\item
There exists a non-zero $(L_{P}\cap H)N_{P}$-invariant Radon measure on $(L_{P}\cap H)N_{P}/(L_{P}\cap H)N_{Q}$.
\end{enumerate}
\end{cor}

\begin{proof}
By Proposition \ref{Measures prop H, L_P cap H, (L_P cap H)N_Q unimodular}, also applied with $P = Q$, all occurring groups are unimodular.
\end{proof}

As a consequence of Corollary \ref{Measures prop existence of inv measures}, the double fibration (\ref{RadTrans eq double fibration}) satisfies properties (A) and (B) of  Section  \ref{subsection Introduction to Radon transforms}.

The groups $(L_{P}\cap H)N_{Q}$, $N_{P}^{Q}$ and $(L_{P}\cap H)N_{P}$ are unimodular (see
Proposition \ref{Measures prop H, L_P cap H, (L_P cap H)N_Q
unimodular}) and the multiplication map $(L_{P}\cap H)N_{Q}\times N_{P}^{Q}\to(L_{P}\cap H)N_{P}$
is a diffeomorphism. Therefore,
$$
\int_{(L_{P}\cap H)N_{P}/(L_{P}\cap H)N_{Q}}\psi(ln\cdot(L_{P}\cap H)N_{Q})\,d_{(L_{P}\cap H)N_{Q}}(ln)
 =\int_{N_{P}^{Q}}\psi(n\cdot(L_{P}\cap H)N_{Q})\,dn
$$
for every $\psi\in L^{1}\big((L_{P}\cap H)N_{P}/(L_{P}\cap H)N_{Q}\big)$ if the measures are suitably normalized.
Similarly,
\begin{align*}
&\int_{(L_{Q}\cap H)N_{Q}/(L_{P}\cap H)N_{Q}}\phi(ln\cdot(L_{P}\cap H)N_{Q})\,d_{(L_{P}\cap H)N_{Q}}(ln)\\
&\qquad=\int_{(L_{Q}\cap H)/(L_{P}\cap H)}\phi(l\cdot(L_{P}\cap H)N_{Q})\,dl
\end{align*}
for every $\phi\in L^{1}\big((L_{Q}\cap H)N_{Q}/(L_{P}\cap H)N_{Q}\big)$ if  the measures are suitably normalized.

If $N$ is a simply connected subgroup of $G$ with nilpotent Lie algebra $\g[n]$, then the Haar measure on $N$
is related to the Lebesgue measure on $\g[n]$ by
$$
\int_{N}\phi(n)\,dn=c\int_{\g[n]}\phi(\exp(Y))\,dY\qquad\big(\phi\in L^{1}(N)\big)
$$
for some constant $c>0$. Here $dY$ denotes the unit Lebesgue measure of $\g[n]$ relative to the inner product $-B(\cdot,\theta\cdot)$.
We will choose the normalization of measures on the groups $A_{\g[q]}$ and $N_{P}^{Q}$ always such that $c=1$.
If $P$, $Q$ and $S$ are parabolic subgroups with $P\subseteq Q\subseteq S$, then because of this choice for the normalization of the measures
\begin{equation}\label{Measures eq int_(N_P^S)=int_(N_P^Q) int_(N_Q^S)}
\int_{N_{P}^{S}}\phi(n)\,dn=\int_{N_{P}^{Q}}\int_{N_{Q}^{S}}\phi(nn')\,dn'\,dn\qquad\big(\phi\in L^{1}(N_{P}^{S})\big).
\end{equation}
Furthermore, we normalize the Haar measure on any compact group such that it is a probability measure.

\subsection{Radon transforms}
\label{subsection Radon transforms}
The Radon transforms $\Rt[P]$ and $\dRt[P]$
for the double fibration (\ref{RadTrans eq double fibration})  are  given by
\begin{align}
\label{RadTrans eq def R_P^Q}&\Rt[P][Q]\phi(g\cdot\xi_{P})=\int_{N_{P}^{Q}}\phi(gn\cdot\xi_{Q})\,dn & (\phi\in \Ds(\Xi_{Q}), g\in G),\\
\label{RadTrans eq def S_P^Q}&\dRt[P][Q]\psi(g\cdot\xi_{Q})=\int_{(L_{Q}\cap H)/(L_{P}\cap H)}\psi(gh\cdot\xi_{P})\,d_{L_{P}\cap H}h & (\psi\in \Ds(\Xi_{P}), g\in G).
\end{align}

\noindent
We write $\Rt[P]$ for $\Rt[P][G]$ and $\dRt[P]$ for $\dRt[P][G]$. These Radon transforms are given by
\begin{align*}
&\Rt[P]\phi(g\cdot\xi_{P})=\int_{N_{P}}\phi(gn\cdot x_{0})\,dn & (\phi\in \Ds(X), g\in G),\\
&\dRt[P]\psi(g\cdot x_{0})=\int_{H/(L_{P}\cap H)}\psi(gh\cdot\xi_{P})\,d_{L_{P}\cap H}h & (\psi\in \Ds(\Xi_{P}), g\in G).
\end{align*}
If $P=P_{0}$ is a minimal $\sigma\circ\theta$-stable parabolic subgroup, then $\Rt[P_{0}]$ is called the horospherical transform associated to $P_{0}$.

\begin{rem}\label{RadTrans rem Kroetz}
In \cite[Section 2]{Krotz_HorosphericalTransformOnRealSymmetricVarieties:KernelAndCokernel} it is claimed that the set $\Hor(X)$ of horospheres in $X$ can be given the structure of a connected real analytic manifold. However, the real analytic atlas for $\Hor(X)$ given there is not an atlas in the proper sense.
In fact, each of the finitely many $G$-orbits in $\Hor(X)$ serves as the domain for a chart, but it is easily seen that not all of these orbits need to have maximal dimension.

In \cite[Remark
3.3]{Krotz_HorosphericalTransformOnRealSymmetricVarieties:KernelAndCokernel}
it is furthermore claimed that the horospherical transforms
$\Rt[P^{w}]$ with $w\in\WGs$ induce a transform
from the space of real analytic vectors for the left regular
representation of $G$ on $L^{1}(X)$ to the space of real analytic functions on
$\Hor(X)$. Even if $\Hor(X)$ could be equipped with a canonical structure of a real analytic manifold, then it is not clear to us why this should
be true.
\end{rem}

\subsection{Spaces of functions and distributions}
\label{subsection Spaces of functions and distributions}
Let $\pi_{\g[q]}$ be the orthogonal projection $\g\to\g[q]$ with respect to the decomposition $\g=\g[h]\oplus\g[q]$. We define the map $\Acomp_{P}:G\to\g[a]_{P}\cap\g[a]_{\g[q]}$ to be given by
\begin{equation}\label{RadTrans eq A_P def}
\Acomp_{P}(kman)=\pi_{\g[q]}(\log a)\qquad(k\in K, m\in M_{P}, a\in A_{P}, n\in N_{P}).
\end{equation}
Note that $\Acomp_{P}$ is real analytic and is left $K$ and right $(L_{P}\cap H)N_{P}$-invariant.

The following lemma describes the relation between $\Acomp_{P}$ and the Iwasawa decomposition.

\begin{lemma}\label{RadTrans lemma A_P=pi_(a_P cap q)A_KAN}
Let $\pi_{\g[a]_{P}\cap \g[q]}$ be the orthogonal projection
(with respect to $-B(\cdot,\theta\cdot)$) onto $\g[a]_{P}\cap
\g[q]$. Let $P_{m}$ be a minimal parabolic subgroup such that $A\subseteq P_{m}\subseteq P$. Define $\Acomp_{KAN_{m}}:G\to\g[a]$ to be the map that for $g\in G$ is given by $g\in K\exp\circ\Acomp_{KAN_{P_{m}}}(g)N_{P_{m}}$. Then
$$
\pi_{\g[a]_{P}\cap \g[q]}\circ\Acomp_{KAN_{P_{m}}}
=\Acomp_{P}.
$$
Moreover, for all $g_{1},g_{2}\in G$ we have $\Acomp_{P}(g_{1}g_{2})\in\Acomp_{P}(g_{1}K)+\Acomp_{P}(g_{2})$.
\end{lemma}

\begin{proof}
The first part of the lemma follows immediately from the equality $L_{P}=(K\cap M_{P})AN_{P_{m}}^{P}$. The second part is a direct consequence of the first part as
$$
\Acomp_{KAN_{P_{m}}}(gkan)
=\Acomp_{KAN_{P_{m}}}(gk)+\log(a)\qquad(g\in G, k\in K,a\in A,n\in N_{P_{m}}).
$$
\end{proof}

We define the function $J_{P}:\Xi_{P}\to\R_{>0}$ to be given by $J_{P}(g\cdot\xi_{P})=e^{-2\rho_{P}\circ\Acomp_{P}(g)}$,
where
$\rho_{P}$ is defined as in (\ref{Measures eq def rho_Q}) with $P$ in place of $Q$.
Note that $J_{G}$ equals the constant function $\1$ on $X$.
Let
$$
L^{1}(\Xi_{P},J_{P})=\{\phi:\Xi_{P}\to\C: \phi J_{P}\in L^{1}(\Xi_{P})\}.
$$
Endowed with the norm
\begin{equation}\label{FuncSpaces eq def norm L^1(Xi_P,J_P)}
\phi\mapsto\int_{\Xi_{P}}|\phi(\xi)|J_{P}(\xi)\,d\xi,
\end{equation}
$L^{1}(\Xi_{P},J_{P})$ is a Banach space.

\begin{lemma}\label{FuncSpaces lemma c^-1g^*J<J<c g^*J}
For every compact subset $C$ of $G$, there exists a constant $c > 0$ such that for every $g\in C$ and $\xi\in\Xi_{P}$
$$
c^{-1} J_{P}(g\cdot\xi)\leq J_{P}(\xi)\leq c J_{P}(g\cdot\xi).
$$
\end{lemma}

\begin{proof}
Let $C$ be a compact subset of $G$, let $g\in C$ and $\xi\in\Xi_{P}$. By Lemma \ref{RadTrans lemma A_P=pi_(a_P cap q)A_KAN} we have $J_{P}(g\cdot\xi)\in e^{-2\rho_{P}\Acomp_{P}(gK)}J_{P}(\xi)$.
The set $CK$ is compact and hence $\Acomp_{P}(CK)$ is also compact. Let $c>0$ be such that
$$
c>
\max_{\pm}\max_{Y\in\Acomp_{P}(CK)}e^{\pm2\rho_{P}(Y)}.
$$
Then the desired inequalities hold.
\end{proof}

It follows from Lemma  \ref{FuncSpaces lemma c^-1g^*J<J<c g^*J} that the space $L^{1}(\Xi_{P},J_{P})$ is invariant under left translation
by elements of $G$. Accordingly, we define the representation  $\pi$ of $G$ in this space by
\begin{equation}\label{FuncSpaces eq def rep pi on L^1(Xi_P,J_P)}
[\pi(g) \phi] (\xi) = \phi(g^{-1} \cdot \xi), \qquad (\phi \in L^{1}(\Xi_{P},J_{P}), \; g \in G).
\end{equation}

\begin{prop}\label{FuncSpaces prop L^1(Xi_P,J_P) Banach rep}
The representation $\pi$ is a continuous Banach-representation.
\end{prop}

\begin{proof}
Put $V=L^{1}(\Xi_{P}, J_{P})$, and write $\|\cdot\|_{V}$ for the norm
given in (\ref{FuncSpaces eq def norm L^1(Xi_P,J_P)}). For a
compact subset $C\subseteq G$, let $c>0$ be the constant of Lemma
\ref{FuncSpaces lemma c^-1g^*J<J<c g^*J}. Then for all $\phi\in V$ and $g\in C$ we have
$$
\|\pi(g)\phi\|_{V}
=\int_{\Xi_{P}}|\phi(g^{-1}\cdot\xi) J_{P}(\xi)|\,d\xi
=\int_{\Xi_{P}}|\phi(\xi) J(g\cdot \xi)|\,d\xi \leq c \|\phi\|_{V}.
$$
This shows that each map $\pi(g): V \to V$ is bounded, and that
the family $\{\pi(g):g\in C\}$ is equicontinuous.

We will now show that $\lim_{g \to e} \pi(g)\phi = \phi$ for each
$\phi \in V$. By the above mentioned equicontinuity, it suffices
to do this for a dense subspace of $V$. We take the dense subspace
$V_{0}= C_{c}(\Xi_{P})$. Then for each $\phi \in V_{0}$ we have by
the principle of uniform continuity that $\pi(g)\phi  \to \phi $
uniformly and with supports in a compact set as $g \to e$. This in
turn implies that $\pi(g)\phi  \to \phi$ in $V$.

It now follows by application of the principle of uniform
boundedness that $\pi$ is a continuous representation.
\end{proof}

\begin{lemma}\label{FuncSpaces lemma int_Xi phi J= int_K int_L/(L cap H)phi}
Let $P$ be a  $\sigma\circ\theta$-stable parabolic subgroup
containing $A$. Then with suitably normalized measures, we have
$$
\int_{\Xi_{P}}\phi(\xi)J_{P}(\xi)\,d\xi
=\int_{K}\int_{L_{P}/(L_{P}\cap H)}\phi(kl\cdot\xi_{P})\,d_{L_{P}\cap H}l\,dk
$$
for every $\phi\in L^{1}(\Xi_{P}, J_{P})$.
\end{lemma}

See Appendix A for the normalization of measures.

\begin{proof}
We leave it to the reader to check that all measures that appear
in this proof may be normalized such that the equalities hold.

If $\chi\in L^{1}(G)$, then
\begin{align*}
\int_{G}\chi(g)\,dg
&=\int_{K}\int_{P}\chi(kp)\Delta_{P}(p)\,dp\,dk\\
&=\int_{K}\int_{L_{P}/(L_{P}\cap H)}\int_{(L_{P}\cap H)N_{P}}\chi(kls)\Delta_{P}(ls)\,ds\,d_{(L_{P}\cap H)}l\,dk.
\end{align*}
Here $dp$ denotes the left invariant measure, and $\Delta_{P}$  the
modular function of $P$. $\Delta_{P}$ is right
$(L_{P}\cap H)N_{P}$-invariant, hence
$$\int_{K}\int_{L_{P}/(L_{P}\cap H)}\eta(kl\cdot\xi_{P})\Delta_{P}(l)\,d_{(L_{P}\cap H)}l\,dk\qquad\big(\eta\in L^{1}(\Xi_{P})\big)$$
defines a $G$-invariant Radon measure on $\Xi_{P}$. Note that $G$-invariant Radon measures on $\Xi_{P}$ are unique up to multiplication by a  constant.
Therefore, under the assumption that the measures are suitably normalized,
$$
\int_{\Xi_{P}}\phi(\xi)J_{P}(\xi)\,d\xi
=\int_{K}\int_{L_{P}/(L_{P}\cap H)}\Delta_{P}(l)\phi(kl\cdot\xi_{P})J_{P}(kl\cdot\xi_{P})\,d_{L_{P}\cap H}l\,dk.
$$
for every $\phi\in L^{1}(\Xi_{P},J_{P})$.
The pullback of $J_{P}$ along 
$$
K\times L_{P}/(L_{P}\cap H)\to\Xi_{P};\qquad \big(k,l\cdot(L_ {P}\cap H)\big)\mapsto kl\cdot\xi_{P}
$$
equals
$$
K\times L_{P}/(L_{P}\cap H)\to\R;\quad \big(k,l\cdot(L_{P}\cap H)\big)\mapsto\frac{1}{\Delta_{P}(l)}.
$$
This proves the proposition.
\end{proof}

We define $\Es^{1}(\Xi_{P},J_{P})$ to be the subspace of
$\Es(\Xi_{P})$ consisting of functions that represent a smooth
vector for the $G$-representation $\big(\pi,L^{1}(\Xi_{P},J_{P})\big)$ defined in (\ref{FuncSpaces eq def rep pi on L^1(Xi_P,J_P)}). We endow this space with the Fr\'{e}chet
topology induced by the natural bijection $\Es^{1}(\Xi_{P},J_{P})\to L^{1}(\Xi_{P},J_{P})^{\infty}$.
The space $\Es^{1}(X,J_{G})$ is denoted by $\Es^{1}(X)$. (Recall
that $J_{G}=\1_{X}$.)

\begin{prop}\label{FuncSpaces prop Es^1(Xi_P,J_P) subset C_0(Xi_P)}
If $\phi\in \Es^{1}(\Xi_{P}, J_{P})$, then $\phi$ vanishes at
infinity.
\end{prop}

\begin{proof}
Let $\diag(K)$ be the diagonal in $K\times K$ and let
$S=\big((K\times K)/\diag(K)\big)\times \big(L_{P}/(L_{P}\cap
H)\big)$. The map
$$\Phi:S\to \Xi_{P};\qquad
\big((k_{1},k_{2})\cdot\diag(K),l\cdot(L_{P}\cap H)\big)\mapsto k_{1}k_{2}^{-1}l\cdot\xi_{P}$$
is a surjective, smooth submersion. Since we take the measure of
$K$ to be normalized, pullback along $\Phi$ defines an isometric
embedding $\Phi^{*}:L^{1}(\Xi_{P},J_{P})\to L^{1}(S)$ by Lemma \ref{FuncSpaces lemma int_Xi phi J= int_K int_L/(L cap H)phi}.

Let $n\in\N$ and let $v\in\U(\g)$ be of degree smaller than or
equal to $n$. If $k\in K$, then $\Ad(k)v$ can be written as a
finite sum $\Ad(k)v=\sum_{j}c_{j}(k)u_{j}$,
where the $c_{j}$ are continuous functions from $K$ to $\C$ and
the $u_{j}$ form a basis for the subspace of $\U(\g)$ consisting
of the elements of order at most $n$. Since $K$ is compact, the
functions $c_{j}$ are bounded. Therefore, if $\phi\in
\Es^{1}(\Xi_{P},J_{P})$ and $u,v_{1}\in\U(\g[k])$ and
$v_{2}\in\U(\g[l]_{P})$, then the $L^{1}(S)$-norm of $(u\otimes
v_{1}\otimes v_{2})\Phi^{*}\phi$ can be estimated by a constant
times
$$
\sum_{j}\int_{\Xi_{P}}|uu_{j}\phi(\xi)|J_{P}(\xi)\,d\xi.
$$

This proves that pullback along $\Phi$ maps
$\Es^{1}(\Xi_{P},J_{P})$ to the space $\Es^{1}(S)$ of smooth
representatives for elements in $L^{1}(S)^{\infty}$. Note that $S$
is a symmetric space of the Harish-Chandra class. According to
\cite[Theorem
3.1]{Krotz&Schlichtkrull_OnFunctionSpacesOnSymmetricSpaces} every
function $\psi\in\Es^{1}(S)$ vanishes at infinity. Since pullback
along $\Phi$ maps $\Es^{1}(\Xi_{P},J_{P})$ to $\Es^{1}(S)$ and
$\Phi$ is a continuous surjection, it follows that every function
$\phi\in\Es^{1}(\Xi_{P},J_{P})$ vanishes at infinity.
\end{proof}

Define
$$L^{\infty}(\Xi_{P},J_{P}^{-1})
=\left\{\phi:\frac{\phi}{J_{P}}\in L^{\infty}(\Xi_{P})\right\}.$$
We endow this space with the norm $\phi\mapsto\|J_{P}^{-1}\phi\|_{L^{\infty}(\Xi_{P})}$.
Since $(\Xi_{P},d\xi)$ is a $\sigma$-finite measure space, the pairing
$$
(\psi,\phi)\mapsto\int_{\Xi_{P}}\psi\phi\,d\xi
$$
induces an isometric isomorphism $L^{\infty}(\Xi_{P},J_{P}^{-1})\to
L^{1}(\Xi_{P},J_{P})'$. (See for example \cite[Theorem 4.14.6]{Friedman_FoundationsOfModernAnalysis}.)

Let $\Es_{b}(\Xi_{P},J_{P}^{-1})$ be the space of
$\phi\in\Es(\Xi_{P})$ such that $(u\phi)/J_{P}$ is bounded for every $u\in\U(\g)$.

\begin{prop}
The left regular representation of $G$ on the space $\Es_{b}(\Xi_{P},J_{P}^{-1})$
is a smooth Fr\'echet representation.
\end{prop}

\begin{proof}
We denote the left regular representation of $G$ on
$\Es_{b}(\Xi_{P},J_{P}^{-1})$ by $\pi$. Let $u\in\U(\g)$ be of order
$n$. Let $\{u_{j}:1\leq j\leq m\}$ be a basis for the subspace of
$\U(\g)$ consisting of elements of order at most $n$. Then there
exist continuous functions $c_{j}:G\to\C$ such that
$\Ad(g)u=\sum_{j=1}^{m}c_{j}(g)u_{j}$. Let $C$ be a compact subset
of $G$ and $\phi\in\Es_{b}(\Xi_{P},J_{P}^{-1})$. By Lemma
\ref{FuncSpaces lemma c^-1g^*J<J<c g^*J} there exists a constant
$c>0$ such that for every $g\in C$
\begin{align*}
\sup_{\xi\in\Xi_{P}}\frac{|u(\pi(g)\phi)(\xi)|}{J_{P}(\xi)}
&=\sup_{\xi\in\Xi_{P}}\sum_{j=1}^{m}\frac{c_{j}(g) u_{j}\phi(g\cdot\xi)}{J_{P}(\xi)}
\leq\sup_{\xi\in\Xi_{P}}c\sum_{j=1}^{m}\frac{|u_{j}\phi(g\cdot\xi)|}{J_{P}(g\cdot\xi)}\\
&=\sup_{\xi\in\Xi_{P}}c\sum_{j=1}^{m}\frac{|u_{j}\phi(\xi)|}{J_{P}(\xi)}.
\end{align*}
This implies in particular that $\pi(g)$ is continuous for every
$g\in G$ and every seminorm of $\pi(g)\phi$ is locally uniformly
bounded  in $g$.  Furthermore, if $g=\exp(Y)$ for some element
$Y\in \g$, then again by Lemma \ref{FuncSpaces lemma c^-1g^*J<J<c
g^*J} there exists a constant $c$ such that
\begin{align*}
\sup_{\Xi_{P}}\frac{|\pi(g)\phi-\phi|}{J_{P}}
&\leq\sup_{\Xi_{P}}\frac{1}{J_{P}}\int_{0}^{1}\left|\frac{d}{dt}\pi\big(\exp(tY)\big)\phi\right|\,dt
\leq\int_{0}^{1}\sup_{\Xi_{P}}\frac{|\pi\big(\exp(tY)\big)(Y\phi)|}{J_{P}}\,dt\\
&\leq c\sup_{\Xi_{P}}\frac{|Y\phi|}{J_{P}}.
\end{align*}
Therefore $\lim_{g\to e}\pi(g)\phi=\phi$.
We conclude that the assumptions in \cite[Proposition 4.1.1.1]{Warner_HarmonicAnalysisOnSemiSimpleLieGroups} are satisfied and hence that $\pi$ is a continuous representation.

In order to prove that the representation is smooth, it suffices to show that for every $\phi\in\Es_{b}(\Xi_{P},J_{P}^{-1})$ and $Y\in\g$ we have
\begin{equation}\label{Funcspaces eq limit for derivative}
\lim_{t\to0}\,\sup_{\Xi_{P}}\frac{1}{J_{P}}\left|\frac{\pi\big(\exp(tY)\big)\phi-\phi}{t}-Y\phi\right|=0.
\end{equation}
Since
$$\int_{s=0}^{t}(t-s)\frac{d^{2}}{ds^{2}}\pi\big(\exp(sY)\big)\phi\,ds
    =\pi\big(\exp(tY)\big)\phi-\phi-tY\phi,
$$
it follows that
\begin{align*}
\sup_{\Xi_{P}}\frac{1}{J_{P}}\left|\frac{\pi\big(\exp(tY)\big)\phi-\phi}{t}-Y\phi\right|
&\leq\sup_{\Xi_{P}}\frac{1}{J_{P}}\int_{s=0}^{t}\frac{t-s}{t}\,\Big|\frac{d^{2}}{ds^{2}}\pi\big(\exp(sY)\big)\phi\Big|\,ds\\
&\leq\int_{s=0}^{t}\frac{t-s}{t}\,\sup_{\Xi_{P}}\frac{\big|\frac{d^{2}}{ds^{2}}\pi\big(\exp(sY)\big)\phi\big|}{J_{P}}\,ds.
\end{align*}
By Lemma \ref{FuncSpaces lemma
c^-1g^*J<J<c g^*J} there exists a constant $c$ such that
the latter is smaller than or equal to
$$c\,\sup_{\Xi_{P}}\frac{|Y^{2}\phi|}{J_{P}}\int_{s=0}^{t}\frac{t-s}{t}ds=
c\,\sup_{\Xi_{P}}\frac{|Y^{2}\phi|}{J_{P}}\frac{t}{2}.$$
This implies (\ref{Funcspaces eq limit for derivative}).
\end{proof}

\subsection{Extensions of the Radon transforms}
\label{subsection Extensions of the Radon transforms}
Recall that for a parabolic subgroup $P$ of $G$ and $g\in G$ the conjugate parabolic subgroup $g^{-1}Pg$ is denoted by
$P^{g}$. Furthermore, recall that $\WGs[M_{P}]$ denotes a set of
representatives for the quotient of Weyl groups $\WG[M_{P}\cap
K]/\WG[M_{P}\cap K\cap H]$. (See Section \ref{subsection Notation}.)

\begin{lemma}\label{RadTransExt lemma int_Xi phi J=sum int_K int_A int_N phi}
Let $P_{0}\in\Psg(\g[a]_{\g[q]})$  and let $P$ be a $\sigma\circ\theta$-stable parabolic subgroup containing $P_{0}$.
Then with a suitable normalization of the measure $d\xi$ on $\Xi_{P}$,
$$
\int_{\Xi_{P}}\phi(\xi)J_{P}(\xi)\,d\xi
=\sum_{w\in\WGs[M_{P}]}\int_{K}\int_{A_{\g[q]}}\int_{N_{P_{0}^{w}}^{P}}
    \phi(kan\cdot\xi_{P})\,dn\,da\,dk
$$
for all $\phi\in L^{1}(\Xi_{P}, J_{P})$.
\end{lemma}

\begin{proof}
The lemma follows directly from Lemma \ref{FuncSpaces lemma int_Xi phi J= int_K int_L/(L cap H)phi} and \cite[Theorem 1.2 \& Lemma 7.2]{Olafsson_FourierAndPoissonTransformationAssociatedToASemisimpleSymmetricSpace} applied to $L_{P}/(L_{P}\cap H)$ and the minimal $\sigma\circ\theta$-stable parabolic subgroup
$P_{0}\cap L_{P}=M_{P_{0}}A_{P_{0}}N_{P_{0}}^{P}$ of $L_{P}$.
\end{proof}

From now on we assume that the measure on $\Xi_{P}$ is normalized such that the identity in Lemma \ref{RadTransExt lemma int_Xi phi J=sum int_K int_A int_N phi} holds. (See Appendix A for the normalization of measures.)

As before, let $P$ and $Q$ be $\sigma\circ\theta$-stable parabolic subgroups of $G$ with $A \subseteq P \subseteq Q$.
\begin{lemma}\label{RadTransExt lemma int_Xi int_N|phi| leq int_X|phi|}
If $\phi\in L^{1}(\Xi_{Q},J_{Q})$, then
$$
\int_{\Xi_{P}}\Big(\int_{N_{P}^{Q}}|\phi(gn\cdot\xi_{Q})|\,dn\Big) J_{P}(g\cdot\xi_{P})\,d_{(L_{P}\cap H)N_{P}}g
\leq \int_{\Xi_{Q}}|\phi(\xi)|J_{Q}(\xi)\,d\xi.
$$
\end{lemma}

\begin{proof}
Since $A_{\g[q]}\subseteq L_{P}\subseteq L_{Q}$, we can choose the
sets of representatives $\WGs[M_{P}]$ and $\WGs[M_{Q}]$ such that
$\WGs[M_{P}]\subseteq\WGs[M_{Q}]$. Let
$P_{0}\in\Psg(\g[a]_{\g[q]})$ be a minimal
$\sigma\circ\theta$-stable parabolic subgroup with $P_{0}
\subseteq P$. By Lemma \ref{RadTransExt lemma int_Xi phi J=sum
int_K int_A int_N phi},
\begin{align*}
&\int_{\Xi_{P}}\Big(\int_{N_{P}^{Q}}|\phi(gn\cdot \xi_{Q})|\,dn\Big)J_{P}(g\cdot\xi_{P})\,d_{(L_{P}\cap H)N_{P}}g
=\\&\qquad \sum_{w\in\WGs[M_{P}]}\int_{K}\int_{A_{\g[q]}}\int_{N_{P_{0}^{w}}^{P}}\int_{N_{P}^{Q}}
    |\phi(kann_{P}\cdot\xi_{Q})|\,dn_{P}\,dn\,da\,dk
=\\&\qquad \sum_{w\in\WGs[M_{P}]}\int_{K}\int_{A_{\g[q]}}\int_{N_{P_{0}^{w}}^{Q}}
    |\phi(kan\cdot\xi_{Q})|\,dn\,da\,dk.
\end{align*}
Here we used (\ref{Measures eq int_(N_P^S)=int_(N_P^Q) int_(N_Q^S)}).  Now we use that $\WGs[M_{P}]\subseteq\WGs[M_{Q}]$ and apply Lemma \ref{RadTransExt lemma int_Xi phi J=sum int_K int_A int_N phi} once more, to obtain
\begin{align*}
&\sum_{w\in\WGs[M_{P}]}\int_{K}\int_{A_{\g[q]}}\int_{N_{P_{0}^{w}}^{Q}}
    |\phi(kan\cdot\xi_{Q})|\,dn\,da\,dk
\leq\\&\qquad\sum_{w\in\WGs[M_{Q}]}\int_{K}\int_{A_{\g[q]}}\int_{N_{P_{0}^{w}}^{Q}}
    |\phi(kan\cdot\xi_{Q})|\,dn\,da\,dk
=\int_{\Xi_{Q}}|\phi(\xi)|J_{Q}(\xi)\,d\xi.
\end{align*}
\end{proof}

Recall the definition of the map $\Rt[P][Q]: \Ds(\Xi_{Q})\to\Es(\Xi_{P})$ from (\ref{RadTrans eq def R_P^Q}).

\begin{prop}\label{RadTransExt prop Rt_P^Q:Ds(Xi_Q) to Es^1(Xi_P,J_P) continuously}
The transform $\Rt[P][Q]$ defines a continuous map from $\Ds(\Xi_{Q})$ to
$\Es^{1}(\Xi_{P},J_{P})$.
\end{prop}

\begin{proof}
$\Rt[P][Q]$ defines a continuous map from $\Ds(\Xi_{Q})$ to $L^{1}(\Xi_{P},J_{P})$ by Lemma \ref{RadTransExt lemma int_Xi int_N|phi| leq int_X|phi|}. Due to continuity and equivariance, it is a continuous transform between the spaces of smooth vectors for the left-regular representation of $G$, i.e., it is a continuous map from $\Ds(\Xi_{Q})$ to $L^{1}(\Xi_{P},J_{P})^{\infty}$. The proposition now follows from the fact that $\Rt[P][Q]$ maps elements in $\Ds(X)$ to smooth functions.
\end{proof}

The image of the injection $\Ds(\Xi_{Q})\hookrightarrow L^{1}(\Xi_{Q},J_{Q})$
is dense. Hence, by Lemma \ref{RadTransExt lemma int_Xi int_N|phi| leq
int_X|phi|} and Proposition \ref{RadTransExt prop Rt_P^Q:Ds(Xi_Q)
to Es^1(Xi_P,J_P) continuously} there exists a unique continuous
transform $\Tt[P][Q]:L^{1}(\Xi_{Q},J_{Q})\to L^{1}(\Xi_{P},J_{P})$
such that
$$
\xymatrix@C+1pc@R-1pc{
\ar@{^{(}->}[dd]\Ds(\Xi_{Q})\ar[r]^{\Rt[P][Q]} & \Es^{1}(\Xi_{P},J_{P})\ar@{^{(}->}[dd]\\
\\
L^{1}(\Xi_{Q},J_{Q})\ar[r]^{\Tt[P][Q]} & L^{1}(\Xi_{P},J_{P})
}
$$
is a commuting diagram. Note that $\Tt[P][Q]$ is equivariant and if $\phi\in L^{1}(\Xi_{Q},J_{Q})$, then
$$
\Tt[P][Q]\phi(g\cdot\xi_{P})=\int_{N_{P}^{Q}}\phi(gn\cdot\xi_{Q})\,dn
$$
for almost every $g\cdot\xi_{P}\in \Xi_{P}$. Since $\Tt[P][Q]$ is equivariant and continuous, it maps $L^{1}(\Xi_{Q},J_{Q})^{\infty}$ continuously to $L^{1}(\Xi_{P},J_{P})^{\infty}$.
It follows that $\Rt[P][Q]$ extends to a continuous $G$-map $\Rt[P][Q]: \Es^{1}(\Xi_{Q},J_{Q})\to\Es^{1}(\Xi_{P},J_{P})$.

\begin{prop}\label{RadTransExt prop Rt_P^Q: Es^1 to Es^1}
The map $\Rt[P][Q]: \Ds(\Xi_{Q})\to\Es(\Xi_{P})$ has a unique extension to a continuous linear $G$-map
$\Rt[P][Q]: \Es^{1}(\Xi_{Q},J_{Q})\to \Es^{1}(\Xi_{P},J_{P}).$
The extension is given by the absolute convergent integral
\begin{equation}\label{RadTransExt eq Rt_P^Q phi=int_(N_P^Q) phi}
\Rt[P][Q]\phi(g\cdot\xi_{P})= \int_{N_{P}^{Q}}\phi(gn\cdot\xi_{Q})\,dn
\qquad\big(\phi\in \Es^{1}(\Xi_{Q},J_{Q}), g\in G\big).
\end{equation}
\end{prop}

\begin{proof}
It is clear that the extension is unique and is given by (\ref{RadTransExt eq Rt_P^Q
phi=int_(N_P^Q) phi}) for almost every $g\in G$. It remains to
be shown that for all $g\in G$ the extension is given by (\ref{RadTransExt eq Rt_P^Q
phi=int_(N_P^Q) phi}).

Let $\chi\in\Ds(G)$ and $\varphi\in\Es^{1}(\Xi_{Q},J_{Q})$. Since $\Tt[P][Q]\varphi\in L^{1}(\Xi_{P}, J_{P})^{\infty}$ and since $J_{P}$ is smooth and non-vanishing, the integral
$$
\psi(g\cdot\xi_{P})
=\int_{G}\chi(\gamma)\Tt[P][Q]\varphi(\gamma^{-1}g\cdot\xi_{P})\,d\gamma
=\int_{G}\chi(g\gamma)\Tt[P][Q]\varphi(\gamma^{-1}\cdot\xi_{P})\,d\gamma
$$
is absolutely convergent for every $g\in G$. Furthermore, by Fubini's theorem,
$$\psi(g\cdot\xi_{P})=\int_{N_{P}^{Q}}(\chi*\varphi)(gn\cdot\xi_{Q})\,dn,$$
where $\chi*\varphi$ denotes the convolution of $\varphi$ with $\chi$, i.e.,
$$\chi*\varphi(g\cdot\xi_{Q})
=\int_{G}\chi(g\gamma)\varphi(\gamma^{-1}\cdot\xi_{Q})\,d\gamma.$$
By application of Lebesgue's dominated convergence theorem one
sees that $\psi$ is a smooth function on $\Xi_{P}$. The function
$\chi*\varphi$ is an element of $\Es^{1}(\Xi_{Q},J_{Q})$, hence
$\psi$ is a smooth representative for the smooth vector
$\Tt[P][Q](\chi*\varphi)\in L^{1}(\Xi_{P},J_{P})^{\infty}$. This
implies that $\psi=\Rt[P][Q](\chi*\varphi)$. We conclude that
$\Rt[P][Q]\phi$ is given by (\ref{RadTransExt eq Rt_P^Q phi=int_(N_P^Q) phi}), with absolutely convergent integrals, for functions $\phi=\chi*\varphi$ with $\chi\in\Ds(G)$ and
$\varphi\in\Es^{1}(\Xi_{P},J_{P})$.

By Proposition \ref{FuncSpaces prop L^1(Xi_P,J_P) Banach rep} the
left regular representation of $G$ on $L^{1}(\Xi_{Q}, J_{Q})$ is a
Banach representation. By \cite[Th\'eor\`eme
3.3]{Dixmier&Malliavin_FactorisationsDeFunctionsEtDeVecteursIndefinimentDifferentiables}
the space of smooth vectors for this representation is spanned by
$\{\chi*\phi:\chi\in\Ds(G),\phi\in\Es^{1}(\Xi_{Q},J_{Q})\}$.
This proves the proposition.
\end{proof}

\begin{lemma}\label{RadTransExt lemma int_(Xi_Q)int_(L_Q cap H)/(L_P cap H)|psi(gl cdot xi_P) phi(g cdot xi_Q)|leq|psi|_(infty,J_P)|phi|_(1,J_Q)}
There exists a normalization of the invariant measure on $(L_{Q}\cap H)/(L_{P}\cap H)$ such that for every $\phi\in L^{1}(\Xi_{Q},J_{Q})$ and $\psi\in L^{\infty}(\Xi_{P},J_{P}^{-1})$
\begin{align*}
&\int_{\Xi_{Q}}\int_{(L_{Q}\cap H)/(L_{P}\cap H)}\psi(gl\cdot\xi_{P})\phi(g\cdot\xi_{Q})
    \,d_{(L_{P}\cap H)}l\,d_{(L_{Q}\cap H)N_{Q}}g\\
&\qquad=\int_{\Xi_{P}}\int_{N_{P}^{Q}}\psi(g\cdot\xi_{P})\phi(gn\cdot\xi_{Q})\,dn\,d_{(L_{P}\cap H)N_{P}}g
\end{align*}
with absolutely convergent integrals.
\end{lemma}

\begin{proof}
Since $\psi\in
L^{\infty}(\Xi_{P},J_{P}^{-1})$ the function $|\frac{\psi}{J_{P}}|$ is essentially
bounded on $\Xi_{P}$. Hence, by Lemma \ref{RadTransExt lemma
int_Xi int_N|phi| leq int_X|phi|}
\begin{align*}
&\int_{\Xi_{P}}\int_{N_{P}^{Q}}\psi(g\cdot\xi_{P})\phi(gn\cdot\xi_{Q})\,dn\,d_{(L_{P}\cap H)N_{P}}g\\
&\qquad=\int_{\Xi_{P}}\int_{N_{P}^{Q}}
    \frac{\psi(g\cdot\xi_{P})}{J_{P}(g\cdot\xi_{P})}\phi(gn\cdot\xi_{Q})J_{P}(g\cdot\xi_{P})\,dn\,d_{(L_{P}\cap H)N_{P}}g
\end{align*}
is absolutely convergent.

By the uniqueness of $G$-invariant Radon measures on $G/(L_{P}\cap H)N_{Q}$, we have (up to a positive normalizing constant)
\begin{align*}
&\int_{\Xi_{Q}}\int_{(L_{Q}\cap H)/(L_{P}\cap H)}\theta\big(gl\cdot(L_{P}\cap H)N_{Q}\big)\,d_{(L_{P}\cap H)}l\,d_{(L_{Q}\cap H)N_{Q}}g\\
&\qquad=\int_{\Xi_{P}}\int_{N_{P}^{Q}}\theta\big(gn\cdot(L_{P}\cap H)N_{Q}\big)
    \,dn\,d_{(L_{P}\cap H)N_{P}}g
        \qquad\Big(\theta\in \Ds(G/(L_{P}\cap H)N_{Q})\Big).
\end{align*}
By the first part of the proof and Fubini's theorem,
$$\int_{\Xi_{Q}}\int_{(L_{Q}\cap H)/(L_{P}\cap H)}\psi(gl\cdot\xi_{P})\phi(g\cdot\xi_{Q})
    \,d_{(L_{P}\cap H)}l\,d_{(L_{Q}\cap H)N_{Q}}g
$$
is absolutely convergent as well and the claimed equality holds.
\end{proof}

From now on we assume that the $L_{Q}\cap H$-invariant measure on the homogeneous space $(L_{Q}\cap H)/(L_{P}\cap H)$ is normalized such that the equality in Lemma \ref{RadTransExt lemma int_(Xi_Q)int_(L_Q cap H)/(L_P cap H)|psi(gl cdot xi_P) phi(g cdot xi_Q)|leq|psi|_(infty,J_P)|phi|_(1,J_Q)} holds. (See Appendix A for the normalization of measures.)

\begin{prop}\label{RadTransExt prop Rt_P^Q*: Es_b to Es_b}
If $\psi\in \Es_{b}(\Xi_{P},J_{P}^{-1})$, then for every $g\in G$ the integral
\begin{equation}\label{RadTransExt eq Rt_P^Q* def}
\dRt[P][Q] \psi(g\cdot\xi_{Q})
= \int_{(L_{Q}\cap H)/(L_{P}\cap H)}\psi(gl\cdot\xi_{P})\,d_{(L_{P}\cap H)}l
\end{equation}
is absolutely convergent and the associated function $\dRt[P][Q] \psi: \Xi_{Q} \to \C$
belongs to the space $\Es_{b}(\Xi_{Q},J_{Q}^{-1})$.
Furthermore, the transform
\begin{equation}\label{RadTransExt eq Rt_P^Q* :Es_b to Es_b}
 \dRt[P][Q]:\Es_{b}(\Xi_{P},J_{P}^{-1})\to \Es_{b}(\Xi_{Q},J_{Q}^{-1})\
\end{equation}
thus obtained, is continuous.
\end{prop}

\begin{proof}
Let $\eta\in\Es_{b}(\Xi_{P}, J_{P}^{-1})$ and $\chi\in\Ds(G)$. It
follows from Lemma \ref{RadTransExt lemma int_(Xi_Q)int_(L_Q cap
H)/(L_P cap H)|psi(gl cdot xi_P) phi(g cdot
xi_Q)|leq|psi|_(infty,J_P)|phi|_(1,J_Q)} that the integral
(\ref{RadTransExt eq Rt_P^Q* def}) is absolutely convergent for
almost every $g\in G$ and the associated almost everywhere defined
function $g\mapsto\dRt[P][Q]\eta(g\cdot\xi_{Q})$ on $G$ is locally
integrable. Therefore, for every $g\in G$, the integral
$$\int_{G}\chi(g\gamma)\int_{(L_{Q}\cap H)/(L_{P}\cap H)}\eta(\gamma^{-1}l\cdot\xi_{P})\,d_{(L_{P}\cap H)}l\,d\gamma$$
is absolutely convergent. Furthermore, the integral depends
smoothly on $g$ and by Fubini's theorem it is equal to
$\dRt[P][Q](\chi*\eta)(l\cdot\xi_{Q})$. Here $\chi*\eta$ denotes
the convolution product between $\chi$ and $\eta$, i.e.,
$\chi*\eta$ is the function on $\Xi_{P}$ given by
$$\chi*\eta(\xi)=\int_{G}\chi(g)\eta(g^{-1}\cdot\xi)\,dg\qquad(\xi\in\Xi_{P}).$$
This proves that for every $\chi\in\Ds(G)$ and $\eta\in\Es_{b}(\Xi_{P},J_{P}^{-1})$ the function $\dRt[P][Q](\chi*\eta)$ is defined by absolutely convergent integrals (\ref{RadTransExt eq Rt_P^Q* def}) and is smooth.

Let $\psi\in\Es_{b}(\Xi_{P}, J_{P}^{-1})$. The left regular representation of $G$ on $\Es_{b}(\Xi_{P},J_{P}^{-1})$ is a smooth Fr\'echet representation. Therefore, it follows from \cite[Th\'eor\`eme 3.3]{Dixmier&Malliavin_FactorisationsDeFunctionsEtDeVecteursIndefinimentDifferentiables}
that $\psi$ is equal to a finite sum of convolutions
$\chi*\eta$, with $\chi\in\Ds(G)$ and $\eta\in\Es_{b}(\Xi_{P},
J_{P}^{-1})$. We conclude from the above argument that $\dRt[P][Q]
\psi$ is a smooth function on $\Xi_{Q}$ defined by absolutely
convergent integrals (\ref{RadTransExt eq Rt_P^Q* def}).

From Lemma \ref{RadTransExt lemma int_Xi int_N|phi| leq int_X|phi|} and Lemma \ref{RadTransExt lemma int_(Xi_Q)int_(L_Q cap H)/(L_P cap H)|psi(gl cdot xi_P) phi(g cdot xi_Q)|leq|psi|_(infty,J_P)|phi|_(1,J_Q)} it follows that for every $\phi\in L^{1}(\Xi_{Q},J_{Q})$
$$
\int_{\Xi_{Q}}|\dRt[P][Q]\psi(\xi)\phi(\xi)|\,d\xi
\leq
\sup_{\Xi_{P}}\frac{|\psi|}{J_{P}}\int_{\Xi_{Q}}|\phi(\xi)|J_{Q}(\xi)\,d\xi.
$$
Therefore $\dRt[P][Q]$ defines a continuous map from
$\Es_{b}(\Xi_{P},J_{P}^{-1})$ to the dual space of
$L^{1}(\Xi_{Q},J_{Q})$, which is $L^{\infty}(\Xi_{Q},J_{Q}^{-1})$.
Since $\dRt[P][Q]$ is equivariant and the left regular
representation of $G$ on $\Es_{b}(\Xi_{P},J_{P}^{-1})$ is smooth, it
follows that $\dRt[P][Q]$ is a continuous map
$\Es_{b}(\Xi_{P},J_{P}^{-1})\to\Es_{b}(\Xi_{Q},J_{Q}^{-1})$.
\end{proof}

Note that the transform (\ref{RadTransExt eq Rt_P^Q* :Es_b to Es_b}) is an extension of the earlier defined transform (\ref{RadTrans eq def S_P^Q}) for compactly supported smooth functions. Thus, the notation is unambiguous.

Lemma \ref{RadTransExt lemma int_(Xi_Q)int_(L_Q cap H)/(L_P cap H)|psi(gl cdot xi_P) phi(g cdot xi_Q)|leq|psi|_(infty,J_P)|phi|_(1,J_Q)} has the following
corollary.

\begin{cor}\label{RadTransExt Cor duality}
If $\phi\in \Es^{1}(\Xi_{Q},J_{Q})$ and
$\psi\in\Es_{b}(\Xi_{P},J_{P}^{-1})$, then
$$
\int_{\Xi_{P}}\Rt[P][Q]\phi(\xi)\psi(\xi)\,d\xi
=\int_{\Xi_{Q}}\phi(\zeta)\dRt[P][Q]\psi(\zeta)\,d\zeta.
$$
\end{cor}

Let $\dRt[P][Qt]$ be the adjoint transform of
$\dRt[P][Q]$. By Corollary \ref{RadTransExt Cor duality}, the following diagram commutes.
$$
\xymatrix@C+1pc@R-1pc{
\ar@{^{(}->}[dd]\Es^{1}(\Xi_{Q},J_{Q})\ar[r]^{\Rt[P][Q]}& \Es^{1}(\Xi_{P},J_{P})\ar@{^{(}->}[dd]\\
\\
\Es'_{b}(\Xi_{Q},J_{Q}^{-1})\ar[r]^{\dRt[P][Qt]}&\Es_{b}'(\Xi_{P},J_{P}^{-1})
}
$$
This allows to extend the definition
(\ref{RT eq def distributions}) of the Radon transform
$\Rt[P][Q]\mu$ of a compactly supported distribution $\mu\in\Es'(\Xi_{Q})$ to the Radon transform $\Rt[P][Q]\mu$ of a
distribution $\mu\in \Es_{b}'(\Xi_{Q},J_{Q}^{-1})$. Accordingly, from now
on we will write $\Rt[P][Q]$ for $\dRt[P][Qt]$. If $\mu\in\Es_{b}'(\Xi_{Q},J_{Q}^{-1})$, then $\Rt[P][Q]\mu\in\Es_{b}'(\Xi_{P},J_{P}^{-1})$ is the distribution given by
$$\Rt[P][Q]\mu(\psi)=\mu\big(\dRt[P][Q]\psi\big)\qquad\big(\psi\in\Es_{b}(\Xi_{P},J_{P}^{-1})\big).$$

\subsection{Relations between Radon transforms}
\label{subsection Relations between Radon transforms}
Let $P$, $Q$ and $S$  be three $\sigma\circ\theta$-stable parabolic
subgroups such that $P\subseteq Q\subseteq S$.

We now consider the following diagram.
$$\xymatrix@C=15pt@R=20pt{
        &                               &G/(L_{P}\cap H)N_{S}\ar[ld]\ar[rd] &                               &\\
        &G/(L_{Q}\cap H)N_{S}\ar[ld]\ar[rd] &\circlearrowleft                   &G/(L_{P}\cap H)N_{Q}\ar[ld]\ar[rd] &\\
\Xi_{S} &                               &\Xi_{Q}                            &                               &\Xi_{P}
}
$$
Here every map is a canonical projection. This diagram describes four double fibrations of the type
considered in Section \ref{subsection Introduction to Radon
transforms}. Only three of these are relevant for our
purposes:
\medbreak\noindent
$
\xymatrix{
    &G/(L_{Q}\cap H)N_{S}\ar[ld]\ar[rd]  &\\
**[l]\Xi_{S}  &               &**[r]\Xi_{Q}
    }$
\hspace{-2.2pt}
$\xymatrix{
    &G/(L_{P}\cap H)N_{Q}\ar[ld]\ar[rd]  &\\
**[l]\Xi_{Q}  &               &**[r]\Xi_{P}
    }
$
\medbreak
\medbreak\noindent
and
$$
\xymatrix{
    &G/(L_{P}\cap H)N_{S}\ar[ld]\ar[rd]  &\\
**[l]\Xi_{S}  &               &**[r]\Xi_{P}
    }
$$

The Radon transforms for these double fibrations are related to each other in the following way.

\pagebreak[3]
\begin{prop}\label{RadTransRel prop Rt_P=Rt_P^Q circ Rt_Q}\
\begin{enumerate}[(i)]
\item
The Radon transforms of functions $\Rt[P][Q]$ and $\Rt[Q][S]$ compose as
$$\Rt[P][Q]\circ\Rt[Q][S] = \Rt[P][S]: \Es^{1}(\Xi_{S},J_{S})\to\Es^{1}(\Xi_{P},J_{P}).$$

\item
The dual Radon transforms of functions $\dRt[P][Q]$ and $\dRt[Q][S]$ compose as
$$\dRt[Q][S]\circ\dRt[P][Q] = \dRt[P][S]:\Es_{b}(\Xi_{P},J_{P}^{-1})\to\Es_{b}(\Xi_{S},J_{S}^{-1}).$$

\item
The Radon transforms of distributions $\Rt[P][Q]$ and $\Rt[Q][S]$ compose as
$$\Rt[P][Q]\circ\Rt[Q][S] = \Rt[P][S]: \Es_{b}'(\Xi_{S},J_{S}^{-1})\to\Es_{b}'(\Xi_{P},J_{P}^{-1}).$$
\end{enumerate}
\end{prop}

\begin{proof}
\ \\
{\em (i):}
The multiplication map $N_{P}^{Q}\times N_{Q}^{S}\to N_{P}^{S}$
is a diffeomorphism with Jacobian equal to $1$. Therefore the identity follows from the definitions and
by application of Fubini's theorem.
\ \\
{\em (ii):}
The continuous linear functional on $\Ds\big((L_{S}\cap H)/(L_{P}\cap H)\big)$ mapping a function $\psi$ to
$$
\int_{(L_{S}\cap H)/(L_{Q}\cap H)}\int_{(L_{Q}\cap H)/(L_{P}\cap H)}\psi\big(l_{S}l_{Q}\cdot(L_{P}\cap H)\big)
    \,d_{(L_{P}\cap H)}l_{Q}\,d_{(L_{Q}\cap H)}l_{S}$$
defines a $L_{S}\cap H$ invariant measure on $(L_{S}\cap H)/(L_{P}\cap H)$. If the measures are normalized such that the equality in Lemma \ref{RadTransExt lemma int_(Xi_Q)int_(L_Q cap H)/(L_P cap H)|psi(gl cdot xi_P) phi(g cdot xi_Q)|leq|psi|_(infty,J_P)|phi|_(1,J_Q)} holds, as we assumed, then this measure and the invariant measure on $(L_{S}\cap H)/(L_{P}\cap H)$ are equal. This proves the claim.
\ \\
{\em (iii):}
This is a direct corollary of {\em (ii)}.
\end{proof}

\section{Convex geometry}
\label{section convex geometry}
In this section we prove some results in convex geometry that are needed in the next sections.

Let $V$ be a finite dimensional vector space. If $B$ is a subset of $V$,
we denote the convex hull of $B$ by $\ch(B)$, i.e., $\ch(B)$
is the smallest convex set containing $B$. We call a subset of  $V$ a cone if it
is closed under the action of the multiplicative group $\R_{>0}$.

\subsection{Support functions}
\label{subsection support functions}
Let $S$ be a subset of $V$. The function $H_{S}:V^{*}\to\R\cup\{\pm\infty\}$; $\lambda\mapsto\sup_{x\in S}\lambda(x)$
is called the support function of $S$. The image of $H_{S}$ contains $-\infty$ if and only if $S=\emptyset$.
Define
\begin{equation}\label{ConvexGeom eq def C_S}
 \Cone_{S}=\{\lambda\in V^{*}:H_{S}(\lambda)<\infty\}.
\end{equation}
Note that $\Cone_{S}$ is a convex cone.
It is well known that $x\in \overline{\ch(S)}$ if and only if $\lambda(x)\leq H_{S}(\lambda)$ for all $\lambda\in \Cone_{S}$. (See for example \cite[Theorem 13.1]{Rockafellar_ConvexAnalysis}.)

\begin{lemma}\label{ConvexGeom lemma properties H_S}
Let $B\subseteq V$ be non-empty and let $\Gamma$ be a cone in
$V$ containing $0$. Then the following statements hold.
\begin{enumerate}[(i)]
\item
    $H_{\Gamma}-H_{-B}\leq H_{B+\Gamma}\leq H_{\Gamma}+H_{B}$.

\item
    $\Cone_{B+\Gamma}=\Cone_{\Gamma}\cap\Cone_{B}=\{\lambda\in\Cone_{B}:\lambda(x)\leq 0 \text{ for every }x\in\Gamma\}$. In particular $\Cone_{\Gamma}=\{\lambda\in V^{*}:H_{\Gamma}(\lambda)\leq 0\}$.

\item
    $H_{B+\Gamma}\big|_{\Cone_{B+\Gamma}}=H_{B}\big|_{\Cone_{B+\Gamma}}$. In particular, $H_{\Gamma}\big|_{\Cone_{\Gamma}}=0$.
\end{enumerate}
\end{lemma}

\begin{proof}\ \\
{\em (i):}
Let $\lambda\in V^{*}$.
Then  $-H_{-B}(\lambda)=\inf_{B}\lambda$.
Moreover, for every $x\in B$ and $y\in \Gamma$
$$\lambda(y)+\inf_{B}\lambda
\leq \lambda(x+y)
\leq \lambda(y)+\sup_{B}\lambda
$$
because $B\neq \emptyset$.
The required estimate at the point $\lambda$ now follows by taking suprema over $x\in B$ and $y\in\Gamma$.
\\
{\em (ii):}
By (i) we have $H_{B+\Gamma}\leq H_{\Gamma}+H_{B}$, hence $\Cone_{B}\cap\Cone_{\Gamma}\subseteq\Cone_{B+\Gamma}$.

To prove the converse inclusion, let $\lambda\in\Cone_{B+\Gamma}$. From  $B\subseteq B+\Gamma$ we see that
$H_{B}(\lambda)\leq H_{B+\Gamma}(\lambda)$, hence $\lambda\in\Cone_{B}$. If $\lambda\notin\Cone_{\Gamma}$, then there exists an $x\in \Gamma$ such that $\lambda(x)>0$, hence, because $\Gamma$ is a cone,
$$H_{B+\Gamma}(\lambda)\geq \sup_{b\in B,r\in\R_{>0}}\lambda(b+rx)=\infty.$$
This contradicts the assumption $\lambda\in\Cone_{B+\Gamma}$, and we see that $\lambda\in\Cone_{\Gamma}$.
We have now established the first equality of (ii).
The second equality is obtained by taking $B=\{0\}$.
\\
{\em (iii):}
Let $\lambda\in\Cone_{B+\Gamma}$. Then by {\em (ii)} we have $\lambda(x)\leq 0$ for every $x\in\Gamma$.
Since $\Gamma$ is a cone, it follows that
$H_{\Gamma}\big|_{\Cone_{B+\Gamma}}=0$. Using subsequently {\em (i)} and the fact that $B\subseteq B+\Gamma$, we find
$$H_{B+\Gamma}(\lambda)
\leq(H_{B}+H_{\Gamma})(\lambda)
=H_{B}(\lambda)
\leq H_{B+\Gamma}(\lambda).
$$
This establishes the equality of the restrictions. The final assertion now follows by taking $B =\{0\}$.
\end{proof}

\subsection{Some lemmas}
\label{subsection finitely generated cones}
\begin{lemma}\label{ConvexGeom lemma W cdot B identity}
Let $B$ be a convex compact subset of $V$ and let $\Gamma$ be a closed convex cone in $V$. If $\mathcal{B}$ is a dense subset of $\Cone_{\Gamma}$, then
$$
B+\Gamma
=\big\{x\in V:
    \lambda(x)\leq H_{B}(\lambda)
\textnormal{ for all }\lambda\in\mathcal{B}\big\}.
$$
\end{lemma}

\begin{proof}
The set $B+\Gamma$ is convex and closed. Therefore 
$$
B+\Gamma
=\{x\in V:\lambda(x)\leq H_{B+\Gamma}(\lambda)\textnormal{ for all }\lambda\in\Cone_{B+\Gamma}\}.
$$
By Lemma \ref{ConvexGeom lemma properties H_S} (ii) we have
$\Cone_{B+\Gamma}=\Cone_{\Gamma}=\overline{\mathcal{B}}$.
Furthermore, by Lemma \ref{ConvexGeom lemma properties H_S} (iii)
the functions $H_{B+\Gamma}$ and
$H_{B}$ coincide on that set.
Since $B$ is compact, $H_{B}$ is continuous. This proves the lemma.
\end{proof}

\begin{lemma}\label{ConvexGeom Lemma S=cap(S+Gamma)}
Let $\mathscr{C}$ be a collection of cones $\Gamma$ in $V$ with $0\in\Gamma$ and let $S$ be a closed convex subset of $V$.
If $\Cone_{S}\subseteq\bigcup_{\Gamma\in \mathscr{C}}\Cone_{\Gamma}$ then
$$
S=\bigcap_{\Gamma\in \mathscr{C}}(S+\Gamma).
$$
\end{lemma}

\begin{proof}
If $\Gamma\in\mathscr{C}$ then $0 \in \Gamma$ and it follows that $S \subseteq \cap_{\Gamma\in\mathscr{C}}(S+\Gamma)$.
Conversely, assume that $x\in\cap_{\Gamma\in \mathscr{C}}(S+\Gamma)$. Let $\lambda\in\Cone_{S}$. By assumption there exists a $\Gamma\in \mathscr{C}$ such that $\lambda\in\Cone_{\Gamma}$. According to Lemma \ref{ConvexGeom lemma properties H_S} {\em (ii)}, $\lambda\in\Cone_{S+\Gamma}$. Since $x\in S+\Gamma$, we have $\lambda(x)\leq H_{S+\Gamma}(\lambda)$.
By Lemma \ref{ConvexGeom lemma properties H_S} {\em (iii)}, $H_{S+\Gamma}(\lambda)=H_{S}(\lambda)$, hence
$\lambda(x)\leq H_{S}(\lambda)$.
As $S$ is closed and convex, this implies that $x\in S$.
\end{proof}

We say that a cone $\Gamma$ in a finite dimensional real vector space $V$ is finitely generated if there exists
a finite set $\{\omega_{k}\in V:1\leq k\leq n\}$ such that $\Gamma=\sum_{k=1}^{n}\R_{\geq0}\omega_{k}$.
A cone is said to be polyhedral if it equals the intersection of finitely many closed halfspaces.
According to \cite[Theorem 19.1]{Rockafellar_ConvexAnalysis} every finitely generated
cone is polyhedral and, vice versa, every polyhedral cone is finitely generated.
If $\Gamma$ is a finitely generated cone, generated by $\{\omega_{k}\in V:1\leq k\leq n\}$, then $\Cone_{\Gamma}$ equals the polyhedral cone
$$
\Cone_{\Gamma}=\{\lambda\in V^{*}:\lambda(\omega_{k})\leq 0 \text{ for all }1\leq k\leq n\}.
$$
Therefore $\Cone_{\Gamma}$ is finitely generated as well.

Note that every finitely generated cone  is closed and convex.

\begin{lemma}\label{FuncSpaces lemma cap_k (B+Gamma_k) subseteq B'+cap_k Gamma_k}
Let $V$ be a finite dimensional real vector space and let $n\in\N$. For $k\in\N$ with $k\leq n$,
let $\Gamma_{k}$ be a finitely generated cone in $V$ and let $B_{k}$ be a compact subset of $V$.
Then there exists a compact subset $B$ of $V$ such that
$$\bigcap_{k=1}^{n}\big(B_{k}+\Gamma_{k}\big)\subseteq B+\bigcap_{k=1}^{n}\Gamma_{k}.$$
\end{lemma}

\begin{proof}
Fix an inner product on $V$.
This inner product induces a dual inner product on $V^{*}$ in the usual manner.
We write $\Gamma_{0}=\bigcap_{k=1}^{n}\Gamma_{k}$.

Since the cones $\Gamma_{k}$ are closed and convex,
$$
\Gamma_{0} = \{ x\in V:\,\lambda(x) \leq 0\;\; \text{for all}\;\;
              \lambda \in \sum_{k=1}^{n}\Cone_{\Gamma_{k}}\}.
$$
Hence, $\Cone_{\Gamma_{0}}$ equals the closure of $\sum_{k=1}^{n}\Cone_{\Gamma_{k}}$.
As the cones $\Cone_{\Gamma_{k}}$ are finitely generated, so is their sum.
In particular the sum is closed and we conclude that $\Cone_{\Gamma_{0}}=\sum_{k=1}^{n}\Cone_{\Gamma_{k}}$.

We define the continuous functions
$$s:\Cone_{\Gamma_{1}}\times\dots\times\Cone_{\Gamma_{n}}\to\Cone_{\Gamma_{0}};
    \quad (\lambda_{k})_{k=1}^{n}\mapsto \sum_{k=1}^{n}\lambda_{k}$$
and
$$\nu:(V^{*})^{n}\to\R;\quad(\lambda_{k})_{k=1}^{n}\mapsto \sum_{k=1}^{n}\|\lambda_{k}\|.$$
Note that $s$ is a surjection. Since $\nu$ is proper and non-negative,  $\nu$ has a minimum on $s^{-1}(\{\lambda\})$. We define the function
$$H_{0}:\Cone_{\Gamma_{0}}\to\R;\quad\lambda\mapsto\min_{s^{-1}(\{\lambda\})}\nu.$$

If $\kappa, \lambda\in\Cone_{\Gamma_{0}}$,
then $s^{-1}(\{\kappa\})+s^{-1}(\{\lambda\})\subseteq s^{-1}(\{\kappa+\lambda\})$.
Since $\nu$ is subadditive, we deduce that
$$H_{0}(\kappa+\lambda)
\leq\min_{s^{-1}(\{\kappa\})+s^{-1}(\{\lambda\})}\nu
\leq\min_{s^{-1}(\{\kappa\})}\nu+\min_{s^{-1}(\{\lambda\})}\nu
=H_{0}(\kappa)+H_{0}(\lambda).$$
Hence $H_{0}$ is subadditive. Furthermore, if $r>0$ and
 $\lambda\in\Cone_{\Gamma_{0}}$, then
$s^{-1}(\{r\lambda\})=r s^{-1}(\{\lambda\})$ and thus
it follows that $H_{0}$ is positively homogeneous of degree $1$.
This implies in particular that $H_{0}$ is a convex function
on $\Cone_{\Gamma_{0}}$ .

Since  $\Cone_{\Gamma_{0}}$ is finitely generated,
the intersection of $\Cone_{\Gamma_{0}}$ with any finitely generated cone is again finitely generated. If we fix an orthonormal basis for $V$, then this is in particular the case for the intersection of $\Cone_{\Gamma_{0}}$ with any orthant (hyperoctant). Such an intersection is a proper cone in the sense that if $\lambda$ is a non-zero element of the cone, then $-\lambda$ is not. We can thus conclude that there exists a finite collection $\{\Cone_{j}:1\leq j\leq m\}$ of proper finitely generated cones $\Cone_{j}$ such that $\Cone_{\Gamma_{0}}=\bigcup_{j=1}^{m}\Cone_{j}$.
For $1\leq j\leq m$,
let $\{\omega^{j}_{k}\in \Cone_{j}:1\leq k\leq n_{j}\}$
be a finite set such that $\Cone_{j}=\sum_{k=1}^{n_{j}}\R_{\geq0}\omega^{j}_{k}$.
Since $\Cone_{j}$ is a proper closed cone,  $\Cone_{j}\setminus \{0\}$
is contained in an open halfspace. This implies that
$$\prod_{k=1}^{n_{j}}\R_{\geq0}\omega^{j}_{k}\to\Cone_{j};
\quad (r_{k}\omega^{j}_{k})_{k=1}^{n_{j}}
\mapsto \sum_{k=1}^{n_{j}}r_{k}\omega^{j}_{k}
$$
is a proper map. Therefore, there exist
$r_{k} >0$, for $1\leq k\leq n_{j}$, such that the intersection of the unit sphere with $\Cone_{j}$ is contained in $\ch\big(\bigcup_{k=1}^{n_{j}}[0,r_{k}]\omega^{j}_{k}\big)$.
Since $H_{0}$ is convex, the supremum of $H_{0}$
over the intersection of the unit sphere with $\Cone_{j}$
is smaller than or equal to the supremum of $H_{0}$
over the sets $[0,r_{k}]\omega^{j}_{k}$.
The latter is finite because $H_{0}$ is
homogeneous and $H_{0}(\omega^{j}_{k})$ is
finite for every $1\leq k\leq n_{j}$.
Therefore, there exists an $R_{j}>0$ such that $H_{0}(\lambda)\leq R_{j}\|\lambda\|$ for every $\lambda\in \Cone_{j}$. Let $R$ be the maximum of the $R_{j}$. Then $H_{0}(\lambda)\leq R\|\lambda\|$ for every $\lambda\in\Cone_{\Gamma_{0}}$.

Let now  $x\in \bigcap_{k}(B_{k}+\Gamma_{k})$.
We will use that by compactness of each $B_{k}$, for
$1 \leq k\leq n$,
we have $\Cone_{\Gamma_{k}} = \Cone_{B_{k} + \Gamma_{k}}$ and
$H_{B_{k} + \Gamma_{k}} = H_{B_{k}}$ on the latter set (see Lemma \ref{ConvexGeom lemma properties H_S}).
Let $\lambda\in\Cone_{\Gamma_{0}};$
then we may write $\lambda=\sum_{k=1}^{n}\lambda_{k}=s\big((\lambda_{k})_{k=1}^{n}\big)$
with $\lambda_{k}\in\Cone_{\Gamma_{k}}$, and we see that
$$
\lambda(x)
=s\big((\lambda_{k})_{k=1}^{n}\big)(x)
\leq\sum_{k=1}^{n}H_{B_{k}}(\lambda_{k}).
$$
Again, by compactness of the sets $B_{k}$ ,
there exists an $r>0$ such that $H_{B_{k}}\leq r\|\cdot\|$
and we finally see that $\lambda(x)\leq r\sum_{k=1}^{n}\|\lambda_{k}\|$.
This inequality holds for every
$n$-tuple $(\lambda_{k})_{k=1}^{n}\in\Cone_{\Gamma_{1}}
\times\dots\times\Cone_{\Gamma_{n}}$ such that $s\big((\lambda_{k})_{k=1}^{n}\big)=\lambda$. Therefore, $\lambda(x)\leq rH_{0}(X)\leq rR\|\lambda\|$.

Let $B(0,rR)$ be the closed ball centered at the
origin with radius $rR$. From Lemma
\ref{ConvexGeom lemma properties H_S} it follows that
$\Cone_{B(0,rR)+\Gamma_{0}}=\Cone_{\Gamma_{0}}$
and the restriction of $H_{B(0,rR)+\Gamma_{0}}$ to
$\Cone_{\Gamma_{0}}$ equals $H_{B(0,rR)}$.
The latter support function is given by $H_{B(0,rR)}(\lambda)=rR\|\lambda\|$ for $\lambda\in V^{*}$.

We conclude that $\lambda(x)\leq H_{B(0,rR)}(\lambda)$ for every $\lambda\in\Cone_{\Gamma_{0}}$, and therefore $x\in B(0,rR)+\Gamma_{0}$.
This establishes the desired inclusion with $B = B(0, rR)$.
\end{proof}

\begin{rem}\label{FuncSpaces rem opmerking Bart}
The lemma does not hold true if ``finitely generated'' is replaced
by ``closed and convex''. Bart van den Dries showed us the
following counterexample. Let
\begin{align*}
&\Gamma_{1}=\{(x,0,z)\in\R^{3}:0\leq x\leq z\},\\
&\Gamma_{2}=\{(x,y,z)\in\R^{3}:0\leq x\leq z,\frac{x^{2}}{z}\leq y\leq x\}.
\end{align*}
If $B=\{(0,y,0):-\frac{1}{2}\leq y\leq \frac{1}{2}\}$, then for every $t>1$,
$$(t,\frac{1}{2},t^2)=
\begin{cases}
(t,0,t^{2})+(0,\frac{1}{2},0)\\
(t,1,t^{2})+(0,-\frac{1}{2},0)
\end{cases}$$
is contained in the intersection of $B+\Gamma_{1}$ and $B+\Gamma_{2}$, but there exists no compact subset $B'$ of $\R^{3}$ such that
$$\{(t,\frac{1}{2},t^{2}):t>1\}\subseteq B'+\{(0,0,z):z\geq0\}=B'+(\Gamma_{1}\cap\Gamma_{2}).$$
When going through the proof for Lemma \ref{FuncSpaces lemma cap_k (B+Gamma_k) subseteq B'+cap_k Gamma_k} in this particular case,  the first serious obstruction encountered is that $\Cone_{\Gamma_{1}}+\Cone_{\Gamma_{2}}$ is not closed and therefore $\Cone_{\Gamma_{1}\cap\Gamma_{2}}\neq \Cone_{\Gamma_{1}}+\Cone_{\Gamma_{2}}$.
\end{rem}

For a subset $T$ of $\Sigma(\g,\g[a]_{\g[q]})$, we define $\Gamma(T)=\sum_{\alpha\in T}\R_{\geq0}H_{\alpha}$.
Here $H_{\alpha}$ is given by (\ref{PSubgrp eq
alpha(Y)=B(H_alpha,Y)}). For a $\sigma\circ\theta$-stable parabolic subgroup $P$ containing $A$, we define the cone $\Gamma_{P}=\Gamma(\Sigma(\g,\g[a]_{\g[q]};P))$.

We recall from Section \ref{subsection sigma circ theta-stable parabolic subgroups} that
$\Psg(\g[a]_{\g[q]})$ denotes the collection of minimal
$\sigma\circ \theta$-stable parabolic subgroups containing $A$.

\begin{lemma}\label{RadTransSupp Lemma B+Gamma_P=cap(B+Gamma_P_0)}
Let $P$ be a $\sigma\circ\theta$-stable parabolic subgroup containing $A$. Let $\mathscr{C}$ be the collection of $P_{0}\in
\Psg(\g[a]_{\g[q]})$ with $P_{0}\subseteq P$ and let $B$
be a closed and convex subset of $\g[a]_{\g[q]}$. Then
$$B+\Gamma_{P}=\bigcap_{P_{0}\in \mathscr{C}}(B+\Gamma_{P_{0}}).$$
\end{lemma}

\begin{proof}
Let $\mathscr{C}_{L_{P}}$ denote the set of $\sigma\circ\theta$-stable parabolic subgroups of $L_{P}$ containing $A$. The assignment $P_{0}\mapsto P_{0}\cap L_{P}$ sets up a bijection between the minimal $\sigma\circ\theta$-stable parabolic subgroups of $G$ contained in $P$ and the minimal $\sigma\circ\theta$-stable parabolic subgroups of $L_{P}$. In particular it sets up a bijection between $\mathscr{C}$ and $\mathscr{C}_{L_{P}}$. If $P_{0}\in \mathscr{C}$ and $R=L_{P}\cap P_{0}$, then $\Gamma_{P_{0}}=\Gamma_{P}+\Gamma_{P;R}$, where $\Gamma_{P;R}=\Gamma(\Sigma(\g[l]_{P},\g[a]_{\g[q]};R))$.

The cone $\Cone_{\Gamma_{P;R}}$ equals the closure of the dual Weyl chamber
corresponding to $\theta R$. Hence,
$$\bigcup_{R\in\mathscr{C}_{L_{P}}}\Cone_{\Gamma_{P;R}}=\g[a]_{\g[q]}^{*}.$$
Application of Lemma \ref{ConvexGeom Lemma S=cap(S+Gamma)} with $S=B+\Gamma_{P}$ now yields
$$
\bigcap_{P_{0}\in \mathscr{C}}(B+\Gamma_{P_{0}})
=\bigcap_{R\in\mathscr{C}_{L_{P}}}(B+\Gamma_{P}+\Gamma_{P;R})
=B+\Gamma_{P}.
$$
\end{proof}

\subsection{Polar decompositions}
\label{subsection polar decompositions}
Let $P$ and $Q$ be $\sigma\circ\theta$-stable parabolic subgroups of $G$ with $A \subseteq P\subseteq Q$.

Since $G=KL_{P}N_{P}$ and $L_{P}=(L_{P}\cap K)A_{\g[q]}(L_{P}\cap H)$, the map $K\times A_{\g[q]}\to \Xi_{P}$; $(k,a)\mapsto ka\cdot\xi_{P}$
is surjective. The decomposition $\Xi_{P}=KA_{\g[q]}\cdot x_{0}$ is called the polar decomposition of $\Xi_{P}$. For $P=G$, we obtain the polar decomposition $X=KA_{\g[q]}\cdot x_{0}$ of $X$. If $B$ is a subset of $\g[a]_{\g[q]}$, then we define
$$
X(B)=K\exp(B)\cdot x_{0}
\quad {\rm and} \quad
\Xi_{P}(B)=K\exp(B)\cdot\xi_{P}.
$$
Note that $X(B)=\Xi_{G}(B)$.

We recall that the definition of $\Acomp_{P}:G\to\g[a]_{P}\cap \g[q]$ is given by (\ref{RadTrans eq A_P def}). We further recall that $\WG[M_{P}\cap K\cap H]$ denotes the subgroup of
$\WG$ consisting of all elements that can be realized in
$M_{P}\cap K\cap H$ and that $\Sigma_{-}(\g,\g[a]_{\g[q]};P)$ is the subset of $\Sigma(\g,\g[a]_{\g[q]};P)$ consisting of roots $\alpha$ such that $-1$-eigenspace for $\sigma\circ\theta$ in $\g_{\alpha}$ is non-trivial.

For convenience we state here the convexity theorem of Van den Ban.

\begin{thm}[{Van den Ban's convexity Theorem \cite[Theorem 1.1]{vdBan_ConvexityTheoremForSemisimpleSymmetricSpaces}}]
\label{RadTransSupp Thm Convexity theorem}
Let $P_{0}\in\Psg(\g[a]_{\g[q]})$ and $a\in A_{\g[q]}$. Then
$$
\Acomp_{P_{0}}(aH)=\ch(\WG[K\cap H]\cdot\log a)+\Gamma\big(\Sigma_{-}(\g,\g[a]_{\g[q]};P)\big)
$$
\end{thm}

The convexity theorem allows us to give the following characterization of the sets $\Xi_{P}(B+\Gamma_{P})$ with $B$ a closed $\WG[M_{P}\cap K\cap H]$-invariant convex subset of $\g[a]_{\g[q]}$.

\begin{lemma}\label{RadTransSupp lemma relation Xi_(P,B), Acomp}
Let $B$ be a closed $\WG[M_{P}\cap K\cap H]$-invariant convex subset of
$\g[a]_{\g[q]}$ and $g\in G$. Then the following two assertions
are equivalent.
\begin{enumerate}[(i)]
\item
$g\cdot\xi_{P}\in\Xi_{P}(B+\Gamma_{P})$
\item
$\Acomp_{P_{0}}\big(g(L_{P}\cap H)\big)\subseteq B+\Gamma_{P_{0}}$ for all $P_{0}\in\Psg(\g[a]_{\g[q]})$ such that $P_{0} \subseteq P$.
\end{enumerate}
\end{lemma}

\begin{proof}
Let $g \in G$ and let $Y\in\g[a]_{\g[q]}$ be such that $g\in K \exp Y (L_{P}\cap H)N_{P}$. Let $P_{0}\in \Psg(\g[a]_{\g[q]})$ with $P_{0}\subseteq P$. Then
$$
\Acomp_{P_{0}}\big(g(L_{P}\cap H)\big)
=\Acomp_{P_{0}}\big(\exp Y (L_{P}\cap H)\big).
$$
Note that $\exp Y (L_{P}\cap H)\subseteq L_{P}$. Now we apply Theorem \ref{RadTransSupp Thm Convexity theorem} to
$L_{P}/(L_{P}\cap H)$ and the minimal $\sigma\circ\theta$-stable parabolic subgroup $P_{0}\cap L_{P}$ of $L_{P}$ and
we thus obtain that
\begin{equation}\label{RadTransSupp eq Acomp(gL_PH)=ch(W Acomp(g))+Gamma}
\Acomp_{P_{0}}\big(g(L_{P}\cap H)\big)
=\ch\big(\WG[M_{P}\cap K\cap H]\cdot Y\big)
    +\Gamma\big(\Sigma_{-}(\g[l]_{P},\g[a]_{\g[q]};P_{0}\cap L_{P})\big).
\end{equation}
Note that $\WG[M_{P}\cap K\cap H]\cdot\Gamma_{P}=\Gamma_{P}$.

Now assume that (i) holds, then  $Y\in B+\Gamma_{P}$. The latter set is $\WG[M_{P}\cap K \cap H]$-invariant, so if $P_{0}$ is as in (ii),
then it follows from (\ref{RadTransSupp eq Acomp(gL_PH)=ch(W Acomp(g))+Gamma}) that
$$
\Acomp_{P_{0}}\big(g(L_{P}\cap H)\big)
\subseteq
B+\Gamma_{P}+\Gamma\big(\Sigma_{-}(\g[l]_{P},\g[a]_{\g[q]};P_{0}\cap
L_{P})\big) \subseteq B+\Gamma_{P_{0}},
$$
and (ii) follows.

Conversely, assume that (ii) holds. Then it follows from (\ref{RadTransSupp eq Acomp(gL_PH)=ch(W Acomp(g))+Gamma}) that $Y \in B + \Gamma_{P_{0}}$ for all $P_{0}\in\Psg(\g[a]_{\g[q]})$ with $P_{0}\subseteq P$.
In view of Lemma \ref{RadTransSupp Lemma B+Gamma_P=cap(B+Gamma_P_0)}, this implies that $Y\in B+\Gamma_{P}$, so that (i) follows.
\end{proof}

\section{Support of a transformed function}
\label{section Support of a transformed function}
Throughout this section, let $P$ and $Q$ be $\sigma\circ\theta$-stable parabolic subgroups of $G$ with $A \subseteq P\subseteq Q$.

The support of $\Rt[P][Q]\phi$ for $\phi\in\Ds(\Xi_{Q})$ need not be compact in general.
The aim of the present section is to give a description of $\supp(\Rt[P][Q]\phi)$ in terms of $\supp(\phi)$. We start with the general case in Section \ref{subsection the general case}. For the horospherical transform some stronger statements can be obtained. We deal with this in Section \ref{subsection support for the horospherical transform}.

\subsection{The general case}
\label{subsection the general case}
\begin{prop}\label{RadTransSupp prop relation Xi_(Q,B+Gamma^+),Xi_(P,B+Gamma^+)}
Let $B$ be a $\WG[M_{Q}\cap K\cap H]$-invariant closed convex subset of
$\g[a]_{\g[q]}$ and let $g\in G$.
If $g\cdot\xi_{Q}\in\Xi_{Q}(B+\Gamma_{Q})$, then $g\cdot\xi_{P}\in\Xi_{P}(B+\Gamma_{P})$.
\end{prop}

\begin{proof}
Let $g\cdot\xi_{Q}\in\Xi_{Q}(B+\Gamma_{Q})$. Then
(ii) of Lemma \ref{RadTransSupp lemma relation Xi_(P,B), Acomp} holds with $Q$ in place of $P$.
Since $P\subseteq Q$, every minimal $\sigma\circ\theta$-stable parabolic
subgroup $P_{0}$ that is contained in $P$ is also contained in
$Q$ so that (ii) of Lemma \ref{RadTransSupp lemma relation
Xi_(P,B), Acomp} also holds for $P$.
By the mentioned lemma it then follows that $g\cdot\xi_{P}\in\Xi_{P}(B+\Gamma_{P})$.

\end{proof}

Proposition \ref{RadTransSupp prop relation
Xi_(Q,B+Gamma^+),Xi_(P,B+Gamma^+)} has the following corollary.

\begin{cor}\label{RadTransSupp cor supp(Rt_Q phi)subseteq Xi_Q,B => supp(Rt_P phi)subseteq Xi_P,B}
Assume that $B$ is a $\WG[M_{Q}\cap K\cap H]$-invariant compact convex subset of
$\g[a]_{\g[q]}$ and that $\phi\in\Es^{1}(\Xi_{Q},J_{Q})$. If
$$
\supp(\phi)\subseteq\Xi_{Q}(B+\Gamma_{Q})
$$
then
$$
\supp(\Rt[P][Q]\phi)\subseteq\Xi_{P}(B+\Gamma_{P}).
$$
In particular, if $B\subseteq \g[a]_{\g[q]}$ is compact, convex and $\WG[K\cap H]$-invariant and $\phi\in\Es^{1}(X)$, then $\supp(\phi)\subseteq X(B)$ implies $\supp(\Rt[P]\phi)\subseteq\Xi_{P}(B+\Gamma_{P})$.
\end{cor}

\begin{proof}
Let $\phi\in\Es^{1}(\Xi_{Q},J_{Q})$ and assume that $\supp(\phi)\subseteq\Xi_{Q}(B+\Gamma_{Q})$.
Let $g\in G$ be such that $\Rt[P][Q]\phi(g\cdot\xi_{P})\neq0$. Then there exists an $n\in N_{P}^{Q}$ such that $\phi(gn\cdot\xi_{Q})\neq0$. By assumption
$gn\cdot\xi_{Q}\in\Xi_{Q}(B+\Gamma_{Q})$, hence by
Proposition \ref{RadTransSupp prop relation
Xi_(Q,B+Gamma^+),Xi_(P,B+Gamma^+)},  $g\cdot\xi_{P}=gn\cdot\xi_{P}\in\Xi_{P}(B+\Gamma_{P})$. As $B$ is compact, it follows that $\Xi_{P}(B+\Gamma_{P})$ is closed and hence $\supp(\Rt[P][Q]\phi)\subseteq\Xi_{P}(B+\Gamma_{P})$.

The second statement is obtained from the first by taking $Q$ equal to $G$.
\end{proof}

The results stated in Corollary \ref{RadTransSupp cor supp(Rt_Q phi)subseteq Xi_Q,B => supp(Rt_P phi)subseteq Xi_P,B} are sufficient for our purposes. The following proposition provides a more precise statement that is however less explicit and for a smaller class of functions. We recall that the map $E_{P}^{Q}$ from $\Xi_{P}$ to the power set of
$\Xi_{Q}$ that for $\xi\in\Xi_{P}$ is given by $E_{P}^{Q}(g\cdot\xi_{P})=gN_{P}^{Q}\cdot\xi_{Q}$.

\begin{prop}\label{RadTransSupp prop E(xi) cap supp(mu) = empty <-> xi notin supp(R mu)}
\begin{enumerate}[(i)]
\item
If $\phi\in\Ds(\Xi_{Q})$, then $\supp(\Rt[P][Q]\phi)\subseteq
\{\xi\in\Xi_{P}:E_{P}^{Q}(\xi)\cap\supp(\phi)\neq\emptyset\}$.
\item
If
$\phi\in\Ds(\Xi_{Q})$ is non-negative, then $\supp(\Rt[P][Q]\phi)
=\{\xi\in\Xi_{P}:E_{P}^{Q}(\xi)\cap\supp(\phi)\neq\emptyset\}$.
\end{enumerate}
\end{prop}

\begin{proof}
Let $\phi\in\Ds(\Xi_{Q})$.
Let $g\in G$ and assume $g\cdot\xi_{P}\in\supp(\Rt[P][Q]\phi)$.
Define $p:G\to \Xi_{P}$, $g\mapsto g\cdot\xi_{P}$. Then $\supp(p^{*}\Rt[P][Q]\phi)=p^{-1}\big(\supp(\Rt[P][Q]\phi)\big)$, hence $g\in
\supp(p^{*}\Rt[P][Q]\phi)$. It now follows that
there exists a sequence $(g_{j})_{j\in\N}$ in $G$ such that
$g_{j}\to g $ if $j\to\infty$ and
for every $j\in \N$
$$
0 \neq \Rt[P_{0}]\phi(g_{j}\cdot\xi_{P_{0}})=\int_{N_{P_{0}}}\phi(g_{j}n\cdot\xi_{P_{0}})\,dn.
$$
In particular, there exists a sequence $(n_{j})_{j\in\N}$ in
$N_{P}^{Q}$ such that $\phi(g_{j}n_{j}\cdot\xi_{Q})\neq0$.
Note that $(g_{j}n_{j}\cdot\xi_{Q})_{j\in\N}$ is a sequence in $\supp(\phi)$. As $\supp(\phi)$ is
compact, there exists a convergent subsequence. Without loss of
generality we may therefore assume that $g_{j}n_{j}\cdot\xi_{Q}$
converges to a point $g_{0}\cdot\xi_{Q}\in\supp(\phi)$ if $j\to\infty$.
Now $\lim_{j\to\infty}n_{j}\cdot\xi_{Q}=\lim_{j\to\infty}g_{j}^{-1}g_{j}n_{j}\cdot\xi_{Q} =g^{-1}g_{0}\cdot \xi_{Q}$.
By Proposition \ref{RadTrans prop N_P cdot xi_Q closed submanifold}, $E_{P}^{Q}(\xi_{P})$ is a closed submanifold of $\Xi_{Q}$. Therefore there exists an $n\in N_{P}^{Q}$ such that $g^{-1}g_{0}\cdot\xi_{Q}=n\cdot\xi_{Q}$. By the same proposition, the map $N_{P}^{Q}\to E_{P}^{Q}(\xi_{P})$; $n'\mapsto n'\cdot\xi_{Q}$ is a diffeomorphism. Therefore $n_{j}\to n$ for $j\to\infty$. Since
$(g_{j} n_{j}\cdot\xi_{Q})_{j\in\N}$ is a sequence in the compact set $\supp(\phi)$
it follows that the limit $gn\cdot\xi_{Q}$ is contained in $\supp(\phi)$
as well. On the other hand, $gn\cdot\xi_{Q}\in gN_{P}^{Q}\cdot\xi_{Q}=E_{P}^{Q}(g\cdot\xi_{P})$,
and thus we see that $g\cdot\xi_{P}\in\{\xi\in\Xi_{P}:E_{P}^{Q}(\xi)\cap\supp(\phi)\neq\emptyset\}$. This proves the first assertion.

We now turn to the proof of the second assertion. One inclusion is given by the first assertion of the proposition. To prove the other, assume that
$\phi\in\Ds(\Xi_{Q})$ is non-negative. Let $g \in G$ and assume that $E_{P}(g\cdot\xi_{P})\cap\supp(\phi)\neq\emptyset$.
Then there exists an $n\in N_{P}$ such that $gn\cdot\xi_{Q}\in \supp(\phi)$.
Let $p:G\to \Xi_{Q}$, $g\mapsto g\cdot\xi_{Q}$. Then $\supp(p^{*}\phi)=p^{-1}(\supp\,\phi)$, hence $gn \in
\supp(p^{*}\phi)$. Since $p^{*}\phi\geq 0$ it now follows that
there exists a sequence $(g_{j})_{j\in\N}$ in $G$ such that
$\phi(g_{j}\cdot\xi_{Q})>0$ and $g_{j}\to gn $ if $j\to\infty$.
Note that $g_{j}\cdot\xi_{P}\to gn\cdot\xi_{P}=g\cdot\xi_{P}$
for $j\to\infty$. Since $\phi$ is continuous, there exists for every $j\in\N$ an open neighborhood $U_{j}$ of $g_{j}$ in $G$ such that $\phi(g\cdot\xi_{Q})>0$ if $g\in U_{j}$. This implies that
$\Rt[P][Q]\phi(g_{j}\cdot\xi_{P})>0$ and thus we conclude
that $g\cdot\xi_{P}\in \supp(\Rt[P][Q]\phi)$. This proves the proposition.
\end{proof}

\subsection{The horospherical transform}
\label{subsection support for the horospherical transform}
If we consider the special case where $Q=G$ and $P=P_{0}$ is a minimal $\sigma\circ\theta$-stable parabolic subgroup, then we can obtain from Proposition \ref{RadTransSupp prop E(xi) cap supp(mu) = empty <-> xi notin supp(R mu)} more explicit statements. For this we first prove the following lemma.

\begin{lemma}\label{RadTransSupp lemma E(xi) cap KBH neq empty <-> xi in KBexp(Gamma) xi_P}
Let $P_{0}\in\Psg(\g[a]_{\g[q]})$ and let $B\subseteq \g[a]_{\g[q]}$. Then
$$
\{\xi\in\Xi_{P_{0}}:E_{P_{0}}(\xi)\cap X(B)\neq\emptyset\}
=\bigcup_{b\in B}\Xi_{P_{0}}\Big(\ch(\WG[K\cap H]\cdot b)+\Gamma\big(\Sigma_{-}(\g,\g[a]_{\g[q]};P_{0})\big)\Big).
$$
\end{lemma}

\begin{proof}
The map
\begin{equation}\label{RadTransSupp eq K x A_q to Xi_(P_0) diffeo}
K/(K\cap M_{P_{0}}\cap H)\times A_{\g[q]}\to\Xi_{P_{0}}; \big(k\cdot (K\cap M_{P_{0}}\cap H),a\big)\mapsto ka\cdot\xi_{P_{0}}
\end{equation}
is a diffeomorphism.
Let $\xi\in\Xi_{P_{0}}$. Then we may write
$\xi=ka\cdot\xi_{P_{0}}$ with $k \in K$ and $a\in A_{\g[q]}$. Now
$E_{P_{0}}(\xi)\cap X(B)\neq\emptyset$ is equivalent to the
existence of an element $Y\in B$ such that $KaN_{P_{0}}\cap K
\exp(Y)H\neq\emptyset$.
By Theorem \ref{RadTransSupp Thm Convexity theorem} the latter assertion is equivalent to the
existence of an element $Y \in B$ such that
$\log a\in\ch(\WG[K\cap H] Y) + \Gamma(\Sigma_{-}(\g,\g[a]_{\g[q]}; P_{0}))$.
As (\ref{RadTransSupp eq K x A_q to Xi_(P_0) diffeo}) is a
diffeomorphism, this assertion is in turn equivalent to the
existence of a $Y \in B$ such that
$$
ka\cdot\xi_{P_{0}}
\in\Xi_{P_{0}}\Big(\ch(\WG[K\cap H] Y)+\Gamma\big(\Sigma_{-}(\g, \g[a]_{\g[q]}; P_{0})\big)\Big).
$$
\end{proof}

Proposition \ref{RadTransSupp prop E(xi) cap supp(mu) = empty <-> xi notin supp(R mu)} and Lemma \ref{RadTransSupp lemma E(xi) cap KBH neq empty <-> xi in KBexp(Gamma) xi_P} have the following direct corollary.

\begin{cor}\label{RadTransSupp cor support horospherical transform}
Let $P_{0}\in\Psg(\g[a]_{\g[q]})$ and let $B\subseteq \g[a]_{\g[q]}$.
If $\phi\in\Ds(X)$ and $\supp(\phi)\subseteq X(B)$, then
$$
\supp(\Rt[P_{0}]\phi)
\subseteq
\bigcup_{b\in B}\Xi_{P_{0}}\Big(\ch(\WG[K\cap H]\cdot b)+\Gamma\big(\Sigma_{-}(\g,\g[a]_{\g[q]};P_{0})\big)\Big).
$$
Moreover, we have equality if $\phi$ is non-negative and $\supp(\phi)=X(B)$.
\end{cor}

Note that if $X$ is a Riemannian symmetric space, then $\Sigma_{-}(\g,\g[a]_{\g[q]};P_{0})=\emptyset$. Therefore for those spaces $\Rt[P_{0}]\phi\in\Ds(\Xi_{P_{0}})$ for every $\phi\in\Ds(X)$.

\section{Support theorem for the horospherical transform}
\label{section Support theorem for the horospherical transform}
The aim of this section is to prove a support theorem for the
horospherical transform for functions.

In Section \ref{subsection The Euclidean Fourier transform and
Paley-Wiener estimates} we derive Paley-Wiener type estimates
for the Fourier transform on a Euclidean space for Schwartz
functions with a certain type of support. The horospherical
transform is related to the so-called unnormalized Fourier
transform on $X$. Given the support of the horospherical transform
of a function, the theory from Section \ref{subsection The Euclidean Fourier
transform and Paley-Wiener estimates} yields
a Paley-Wiener estimate for one component of the unnormalized
Fourier transform. This is described in Section \ref{subsection
The unnormalized Fourier transform}. In Section \ref{subsection
The tau-spherical Fourier transform} Paley-Wiener estimates for
one component of the normalized $\tau$-spherical Fourier transform
are deduced from the estimates for the unnormalized Fourier
transform. Then in Section \ref{subsection Function Spaces} we
introduce some subspaces of $\Es^{1}(X)$ that will be used in the
last two sections. For the normalized $\tau$-spherical Fourier
transform there exists an inversion formula due to Van den Ban and
Schlichtkrull, that we describe in Section \ref{subsection
Inversion Formula}. Finally, in Section \ref{subsection Support
Theorem for the Horospherical transform}, we use the inversion
formula and the Paley-Wiener estimates to obtain a support theorem
for the horospherical transform.

Throughout this section, $P_{0}$ denotes a \emph{minimal}
$\sigma\circ\theta$-stable parabolic subgroup containing $A$.

\subsection{The Euclidean Fourier transform and Paley-Wiener estimates}
\label{subsection The Euclidean Fourier transform and Paley-Wiener estimates}
Let $V$ be a finite dimensional real vector space equipped with a positive definite inner product.

Let $u\in \Es(V)$. For each functional $\nu\in V^{*}$ we define $u_{\nu}=e^{-\nu} u$ and we write
$$
\Cone(u)=\{\nu\in V^{*}:u_{\nu}\in\Ss(V)\}.
$$
This set is a convex cone in $V^{*}$. (See \cite[Section
7.4]{Hormander_TheAnalyisOfLinearPartialDifferentialOperators}.)
The Fourier transform of $u$ is defined to be the function $\Ft u$ on $\Cone(u)+iV^{*}$ given by
$$
\Ft u(\lambda)
=\int_{V}e^{-\lambda(x)}u(x)\,dx
=\int_{V}e^{-i\zeta(x)}u_{\nu}(x)\,dx
	\quad(\lambda=\nu+i\zeta\in\Cone(u)+iV^{*}).
$$

Let $P(V^{*})$ denote the ring of polynomial functions $V^{*}\to\C$.
For $p\in P(V^{*})$ we use the notation $p(\partial)$ for the linear partial
differential operator with constant coefficients on $V$ determined by
$p(\partial) e^\nu = p(\nu) e^\nu$ for $\nu\in V^{*}$.
In a similar fashion, we associate to each $p \in P(V)$ a differential operator $p(\partial)$ on $V^{*}$.

Since for every homogeneous polynomial $p_{h}\in P(V)$ of degree $1$ the function $\big(p_{h}(\partial)u\big)_{\nu}=p_{h}(\partial)u_{\nu}+p_{h}(\nu)u_{\nu}$
is Schwartz if $u_{\nu}$ is Schwartz,
we see that $\Cone(u)\subseteq\Cone(p(\partial)u)$
for every polynomial function $p : V^{*} \to \C$.

The function $\zeta\mapsto\Ft u(\nu+i\zeta)$
is a Schwartz function on $V^{*}$ for each $\nu\in\Cone(u)$.
Furthermore, $\Ft u$ is holomorphic on the interior of $\Cone(u)+iV^{*}$ (see \cite[Theorem 7.4.2]{Hormander_TheAnalyisOfLinearPartialDifferentialOperators}) and there
$$
p(\partial) \Ft u = \Ft( x \mapsto p(-x) u(x)),\quad\text{and}\quad \Ft(q(\partial)u)=q \Ft(u),
$$
for all $p\in P(V)$ and $q \in P(V^{*})$.

\begin{lemma}\label{EFT lemma FT holom and id}
Let $S$ be a closed convex subset of $V$ and let the cone
$-\Cone_{S} \subseteq V^{*}$ be defined as in (\ref{ConvexGeom eq
def C_S}). Let $u\in\Ss(V)$. If $\supp(u)\subseteq S$ then $-\Cone_{S}\subseteq\Cone(u)$.
\end{lemma}

\begin{proof}
Let $\nu\in-\Cone_{S}$. Then $(p(\partial)e^{-\nu})\1_{S}$ is
bounded for every polynomial $p \in P(V^*)$. By application of the Leibniz rule
we now see that $p(\partial) u_{\nu}(x) = \mathcal{O}(1 + \|x\|)^{-N}$ for all $p \in P(V^{*})$ and
$N \in \N$. Hence $u_{\nu}$ is Schwartz.
\end{proof}

\begin{prop}[Paley-Wiener estimate]\label{EFT prop Paley-Wiener est for S(V)}
Let $S$ be a closed, convex subset of $V$. If $u\in\Ss(V)$ with
$\supp(u)\subseteq S$, then for every $N\in\N$ and
$\lambda\in-\Cone_{S}+iV^{*}$
$$
|\Ft u(\lambda)|
\leq 2^{N}\|(1+\Delta)^{N}u\|_{L^{1}}(1+\|\lambda\|)^{-N}e^{H_{S}(-\Re(\lambda))}.
$$
\end{prop}

\begin{proof}
If $w\in\Ss(V)$ satisfies $\supp(w)\subseteq S$ and $\lambda\in-\Cone_{S}+iV^{*}$, then
$$|\Ft w(\lambda)|
\leq\int_{V}|e^{-\lambda(x)}|\ |w(x)|\,dx
\leq\|w\|_{L^{1}}\ e^{H_{S}(-\Re(\lambda))}
$$
Let $N\in\N$. Then
$$
(1+\|\lambda\|)^{N}\ |\Ft u(\lambda)|
\leq2^{N}(1+\|\lambda\|^{2})^{N}|\Ft u(\lambda)|
=2^{N}|\Ft((1+\Delta)^{N}u)(\lambda)|,$$
hence, by taking $w=(1+\Delta)^{N}u$, we obtain
$$
(1+\|\lambda\|)^{N}\ |\Ft u(\lambda)|
\leq2^{N}\|(1+\Delta)^{N}u\|_{L^{1}}\ e^{H_{S}(-\Re(\lambda))}.
$$
\end{proof}

\begin{lemma}\label{EFT lemma equivalence of shifted H_(B+Gamma)}
Let $B$ be a compact subset of $V$ and $\Gamma$ a cone in $V$
with $0\in\Gamma$. For every $\nu_{0}\in-\Cone_{B+\Gamma}$ and for every
$\lambda\in\nu_{0}-\Cone_{B+\Gamma}$
$$
H_{B}(-\lambda)-H_{B}(-\nu_{0})
\leq H_{B+\Gamma}(-\lambda+\nu_{0})
\leq H_{B}(-\lambda)+H_{B}(\nu_{0}).
$$
\end{lemma}

\begin{proof}
Recall that $H_{B+\Gamma}\big|_{\Cone_{B+\Gamma}}=H_{B}\big|_{\Cone_{B+\Gamma}}$ by Lemma \ref{ConvexGeom lemma properties H_S}(iii). Let $\nu_{0}\in-\Cone_{B+\Gamma}$ and $\lambda\in\nu_{0}-\Cone_{B+\Gamma}$. Then, as $\Cone_{B+\Gamma}+\Cone_{B+\Gamma}\subseteq\Cone_{B+\Gamma}$, we have
$H_{B}(-\lambda+\nu_{0})\leq H_{B}(-\lambda)+H_{B}(\nu_{0})$ and $H_{B}(-\lambda)\leq H_{B}(-\lambda+\nu_{0})+H_{B}(-\nu_{0})$. As $-\nu_{0}\in\Cone_{B+\Gamma}$, it follows that $H_{B}(-\nu_{0})<\infty$. Hence $H_{B}(-\lambda)-H_{B}(-\nu_{0})\leq H_{B}(-\lambda+\nu_{0})$.
\end{proof}

\begin{lemma}\label{EFT lemma equivalence of shifted norm}
For all $\lambda, \nu_{0}\in V^{*}$ we have
$$
\frac{1 + \|\lambda\|}{1 + \|\nu_{0}\|} \leq ( 1 + \|\lambda + \nu_{0}\|) \leq (1 + \|\lambda\|)(1 + \|\nu_{0}\|).
$$
\end{lemma}

\begin{proof}
The estimate on the right is a straightforward consequence of the
triangle inequality. It follows that
$$
(1 + \|\lambda\|) = (1 + \|\lambda + \nu_{0} + (-\nu_{0})\|) \leq (1 + \|\lambda + \nu_{0}\|)(1 + \|\nu_{0}\|).
$$
This implies the required estimate on the left.
\end{proof}

\begin{prop}[Paley-Wiener estimate]\label{EFT prop Paley-Wiener est for E(V)}
Let $B$ be a compact subset of $V$ and let $\Gamma\subseteq V$ be
a closed cone. Let $\nu_{0}\in-\Cone_{\Gamma}$. Then for every $N
\in \N$ there exists a constant $C_{\nu_{0}, N} > 0$ with the
following property.

If $u$ is a
smooth function on $V$ such that $u_{\nu_{0}}\in\Ss(V)$ and
$\supp(u)\subseteq B+\Gamma$, then for every
$\lambda\in\nu_{0}-\Cone_{\Gamma}+iV^{*}$,
$$|\Ft u(\lambda)|
\leq C_{\nu_{0}, N}\|(1+\Delta)^{N}u_{\nu_{0}}\|_{L^{1}}(1+\|\lambda\|)^{-N}e^{H_{B}(-\Re(\lambda))}.
$$
\end{prop}

\begin{proof}
Let $N\in\N$. Since $u_{\nu_{0}}\in\Ss(V)$ and
$\supp(u_{\nu_{0}})=\supp(u)\subseteq B+\Gamma$,  it follows by application of Proposition
\ref{EFT prop Paley-Wiener est for S(V)} that
\begin{equation}\label{EFT eq Paley-Wiener est for Ft u_(nu_0)}
|\Ft u(\lambda)|
=|\Ft u_{\nu_{0}}(\lambda-\nu_{0})|
\leq 2^{N}\|(1+\Delta)^{N}u_{\nu_{0}}\|_{L^{1}}(1+\|\lambda-\nu_{0}\|)^{-N}
        e^{H_{B+\Gamma}(-\Re(\lambda-\nu_{0}))}
\end{equation}
for all $\lambda\in\nu_{0}-\Cone_{B+\Gamma}+iV^{*}$. Since $\Re(\lambda)\in \nu_{0}-\Cone_{B+\Gamma}=\nu_{0}-\Cone_{\Gamma}$, it follows by application of  Lemma \ref{EFT lemma equivalence of shifted H_(B+Gamma)} that
\begin{equation}\label{EFT eq H_(B+Gamma)(-Re lambda+ nu_0) leq H_B(-Re lambda)+H_B(nu_0)}
H_{B+\Gamma}(-\Re(\lambda-\nu_{0}))\leq H_{B}(-\Re\lambda)+H_{B}(\nu_{0}).
\end{equation}
Finally, by application of Lemma \ref{EFT lemma equivalence of shifted norm} we see that
\begin{equation}\label{EFT eq (1+|lambda+nu_0|)^(-N) leq (1+|lambda|)^(-N)(1+|nu_0|)^(-N)}
(1 +\|\lambda-\nu_{0}\|)^{-N} \leq (1+\|\lambda\|)^{-N} (1+\|\nu_{0}\|)^N.
\end{equation}
Substituting the estimates (\ref{EFT eq H_(B+Gamma)(-Re lambda+ nu_0) leq H_B(-Re lambda)+H_B(nu_0)}) and (\ref{EFT eq (1+|lambda+nu_0|)^(-N) leq (1+|lambda|)^(-N)(1+|nu_0|)^(-N)}) in (\ref{EFT eq Paley-Wiener est for Ft u_(nu_0)}), we obtain the required estimate with
$C_{\nu_{0},N}=2^{N} ( 1 + \|\nu_{0}\|)^{N} e^{H_{B}(\nu_{0})}$.
\end{proof}

\begin{rem}
Proposition \ref{EFT prop Paley-Wiener est for E(V)} is part of
the following Paley-Wiener theorem which we state here for the sake
of completeness.

\medbreak
\medbreak
\noindent
{\em Let $B$ be a
compact subset of $V$ and let $\Gamma\subseteq V$ be a closed
cone. Assume that $\nu_{0}\in-\Cone_{\Gamma}$ and that $w$ is a
function $\nu_{0}-\Cone_{\Gamma}+iV^{*}\to\C$. Then the
following two statements are equivalent.
\begin{enumerate}[(I)]
\item $w$ equals the restriction to $\nu_{0}-\Cone_{\Gamma}+iV^{*}$
of the Fourier transform $\Ft u$ of a function $u\in\Es(V)$ such
that $u_{\nu_{0}}\in\Ss(V)$ and $\supp(u)\subseteq B+\Gamma$
\item
The function $w$ is continuous and its restriction to
$\nu_{0}+iV^{*}$ is Schwartz. For every $\nu\in-\Cone_{\Gamma}$
and $\lambda\in\nu_{0}-\Cone_{\Gamma}+iV^{*}$ the function
$$
z\mapsto w(z\nu+\lambda)
$$
is holomorphic on $\{z\in {\C}:\, \Re\,z >0\}$ and for every $N\in\N$ there exists a positive constant $C_{N}$
such that for all $\lambda\in\nu_{0}-\Cone_{\Gamma}+iV^{*}$
$$
|w(\lambda)|\leq C_{N}
(1+\|\lambda\|)^{-N}e^{H_{B}(-\Re(\lambda))}.
$$
\end{enumerate}
For every $N$ there exists a constant $C_{\nu_{0},N}$, depending on $\nu_{0}$ and $N$
only, such that if $(I)$ holds, then $(II)$ holds with $C_{N}$
smaller than or equal to
$C_{\nu_{0},N}\|(1+\Delta)^{N}u_{\nu_{0}}\|_{L^{1}}$.}
\medbreak
\medbreak
The proof for the case $\nu_{0}=0$ is similar to the usual proof
for the Paley-Wiener theorem for $\Ds(V)$. See for example
\cite[Theorem 7.22]{Rudin_FunctionalAnalyis}. For $\nu_{0}\neq 0$
the theorem then follows by application of Lemmas \ref{EFT lemma equivalence of
shifted H_(B+Gamma)} and \ref{EFT lemma equivalence of shifted
norm}.
\end{rem}

\subsection{The unnormalized Fourier transform}
\label{subsection The unnormalized Fourier transform}
We start by recalling several definitions and results from
\cite{vdBan_ThePrincipalSeriesForAReductiveSymmetricSpaceI}, and \cite{vdBan&Schlichtkrull_FourierTransformOnASemisimpleSymmetricSpace}.

Let $(\zeta,\H_{\zeta})$ be a unitary representation of $M_{P_{0}}$ in
a finite dimensional Hilbert space $\H_{\zeta}$ and let
$\lambda\in\g[a]^{*}_{\g[q]\C}$. The space $\Es(P_{0}:\zeta:\lambda)$
of smooth vectors for the (left) induced representation
$\Ind^{G}_{P_{0}}(\zeta\otimes e^{\lambda}\otimes 1)$ consists of the smooth functions
$f:G\to\H_{\zeta}$ satisfying
\begin{equation}\label{FT eq f(mang)=a^(lambda+rho)xi(m)f(g)}
f(mang)=a^{\lambda+\rho_{P_{0}}}\zeta(m)f(g)\qquad(m\in M_{P_{0}}, a\in A_{P_{0}}, n\in N_{P_{0}}, g\in G).
\end{equation}
Here $\rho_{P_{0}}$ is defined as in (\ref{Measures eq def rho_Q}) with $P_{0}$ in place of $Q$.

We define $V(\zeta)$
to be the formal direct sum of Hilbert spaces
$$
V(\zeta)
=\bigoplus_{w\in\WGs}V(\zeta,w),
\qquad
V(\zeta,w)
=\H_{\zeta}^{w(H\cap M_{P_{0}})w^{-1}},
$$
where $\H_{\zeta}^{w(H\cap M_{P_{0}})w^{-1}}$ is the subspace of $w(H\cap
M_{P_{0}})w^{-1}$-fixed vectors in $\H_{\zeta}$.

Let $\widehat{M_{P_{0}}}_{H}$ be the set of equivalence classes of
finite dimensional unitary representations $(\zeta,\H_{\zeta})$ of
$M_{P_{0}}$ such that $V(\zeta)\neq\{0\}$. The principal series of
representations for $X$ is the series of representations
$\Ind^{G}_{P_{0}}(\zeta\otimes e^{\lambda}\otimes 1)$ with $\lambda\in\g[a]^{*}_{\g[q]\C}$ and
$(\zeta,\H_{\zeta})\in\widehat{M_{P_{0}}}_{H}$.

Let $\Ind^{G}_{P_{0}}(\zeta\otimes e^{\lambda}\otimes 1)$ be a
principal series representation. The space of generalized
functions $G\to \H_{\zeta}$ satisfying (\ref{FT eq
f(mang)=a^(lambda+rho)xi(m)f(g)}) is denoted by
$C^{-\infty}(P_{0}:\zeta:\lambda)$.
Following \cite[Section 5]{vdBan_ThePrincipalSeriesForAReductiveSymmetricSpaceI} we define
$j(P_{0}:\zeta:\lambda):V(\zeta)\to C^{-\infty}(P_{0}:\zeta:\lambda)^{H}$ as follows.
The sets $P_{0}wH$, for $w \in \WGs$ are disjoint and open in $G$ and their union
$$
\Omega(P_{0}) = \bigcup_{w \in\WGs} P_{0} w H
$$
is dense in $G$. For
$\lambda\in\g[a]_{\g[q]}^{*}(P_{0},0)-\rho_{P_{0}}$
the function  is given by
\begin{equation}\label{FT eq def j}
j(P_{0}:\zeta:\lambda)(\eta)(x) =
\begin{cases}a^{\lambda+\rho_{P_{0}}}\zeta(m)\eta_{w}
    & \text{for }x=manwh\in\Omega(P_{0}) \text{ with}\\
    &\qquad m \in M_{P_{0}}, a\in  A_{P_{0}}, n \in N_{P_{0}},\\
    &\qquad w\in\WGs \text{ and } h\in H \\
    0 &\text{for }\;\; x \notin\Omega(P_{0}).

        \end{cases}
\end{equation}
It is known that for $\lambda\in\g[a]_{\g[q]}^{*}(P_{0}, 0) - \rho_{P_{0}}$ the function $j(P_{0}: \zeta: \lambda)(\eta)$ thus defined is continuous;
see \cite[Proposition 5.6]{vdBan_ThePrincipalSeriesForAReductiveSymmetricSpaceI}. For the remaining $\lambda\in\g[a]_{\g[q],\C}^{*}$ it is defined by
meromorphic continuation. For generic $\lambda\in\g[a]^{*}_{\g[q]\C}$,
the map $j(P_{0}:\zeta:\lambda)$ is known to be a bijection
$V(\zeta)\to C^{-\infty}(P_{0}:\zeta:\lambda)^{H}$. See \cite[Theorem 5.10]{vdBan_ThePrincipalSeriesForAReductiveSymmetricSpaceI}.

\begin{lemma}\label{FT lemma a^nu psi Schwartz}
Let $B$ be a compact subset of $\g[a]_{\g[q]}$ and let $\psi\in
\Es^{1}(\Xi_{P_{0}},J_{P_{0}})$ be such that $\supp(\psi)\subseteq\Xi_{P_{0}}(B+\Gamma_{P_{0}})$.
For $k\in K$, define $\psi_{k}:\g[a]_{\g[q]}\to\C$ by $\psi_{k}(Y)=
\psi\big(k\exp(Y)\cdot\xi_{P_{0}}\big)$.
If $\nu\in\g[a]^{*}_{\g[q]}(P_{0},0)$ then $e^{\nu}\psi_{k}$
is a Schwartz function for every $k \in  K$. The map
$$
K\to\Ss(\g[a]_{\g[q]});\quad k\mapsto e^{\nu}\psi_{k}
$$
thus defined is continuous.
\end{lemma}

\begin{proof}
Let $\psi$ be fixed as above. According to Proposition
\ref{FuncSpaces prop Es^1(Xi_P,J_P) subset C_0(Xi_P)}, the
function $u\psi$ vanishes at infinity for every $u\in\U(\g)$. In
particular, it follows that each of the functions $u \psi$ is
bounded and uniformly continuous on $\Xi_{P_{0}}$.

For every $u\in\U(\g)$ there exists a finite set
$F_u \subseteq \U(\g)$,  consisting of linearly independent elements,
such that $\Ad(k) u \in {\rm span}(F_u)$ for all $k \in K$.
Write $\Ad(k)u=\sum_{w \in F_u} c_{u,w}(k) w$, then the $c_{u,w}$
are continuous, hence bounded functions on $K$. It follows that
there exists a constant $C_u > 0$ such that
\begin{align*}
&\sup_{k\in K}\sup_{\g[a]_{\g[q]}}|u \psi_{k}|\\
&\quad= \sup_{k\in K}\sup_{Y\in\g[a]_{\g[q]}}|
                   \big(\Ad(k)u\big)\psi\big(k \exp(Y)\cdot\xi_{P_{0}}\big)|
\leq\sum_{w\in F_{u}} \sup_{k \in K}\sup_{Y \in \g[a]_{\g[q]}} |c_{u,w}(k)|\, |w\psi(k \exp Y)|
<  C_{u}.
\end{align*}
Since
$$
K/(M_{P_{0}}\cap K\cap H)\times\g[a]_{\g[q]}\to\Xi_{P_{0}};\quad
 (k\cdot(M_{P_{0}}\cap K\cap H),Y)\mapsto k\exp(Y)\cdot\xi_{P_{0}}$$
is a diffeomorphism and $\supp(\psi)$
is contained in $\Xi(B+\Gamma_{P_{0}})$, the support of $\psi_{k}$ is contained
in $B+\Gamma_{P_{0}}$. Let $\nu\in\g[a]_{\g[q]}(P_{0},0)$. Then $\Re(\nu)<0$ on
$\Gamma_{P_{0}}\setminus\{0\}$.
Let $p$ be a polynomial function
on $\g[a]_{\g[q]}$ and let $u\in\Sym(\g[a]_{\g[q]})$.
Then  by the Leibniz rule
 there exist finitely many elements $u_{j}\in\Sym(\g[a]_{\g[q]})$
(independent of $k$) such that
$$
\sup_{\g[a]_{\g[q]}}|p\,u (e^{\nu}\psi_{k})|
\leq \sum_{j}\sup_{\g[a]_{\g[q]}}|p\, e^{\nu} u_{j}\psi_{k}|
\leq \sup_{B+\Gamma_{P_{0}}}|p\,e^{\nu}|\sum_{j}C_{u_{j}}
<\infty.
$$
This proves that the functions $e^{\nu}\psi_{k}$ are Schwartz functions on
$\g[a]_{\g[q]}$.

If $k,k'\in  K$, then
$$
\sup_{\g[a]_{\g[q]}}|p\, u\big(e^{\nu}\psi_{k}-e^{\nu}\psi_{k'}\big)|
\leq\Big(\sup_{B+\Gamma_{P_{0}}}|p\,e^{\nu}|\Big)\Big(\sum_{j}\sup_{\g[a]_{\g[q]}}
|u_{j}\psi_{k}-u_{j}\psi_{k'}|\Big).
$$
Now
\begin{equation}\label{Ft eq estimate of |u_j psi_k-u_j psi_k'|}
\sup_{\g[a]_{\g[q]}}|u_{j}\psi_{k}-u_{j}\psi_{k'}|
\leq \sum_{w\in F_{u_{j}}} \sup_{\Xi_{P_{0}}}
    \big|c_{u_{j}, w}(k)\,(w\psi)_{k}-c_{u_{j}, w}(k')\,(w \psi)_{k'}\big|.
\end{equation}
The second statement of  the proposition now follows since the right-hand side of
 (\ref{Ft eq estimate of |u_j psi_k-u_j psi_k'|}) converges to $0$ if $k\to k'$
by the uniform continuity of the functions $w\psi$ and the continuity of the functions
$c_{u_{j}, w}$.
\end{proof}

Let $\Ft_{A_{\g[q]}}$ be the Fourier transform on
$A_{\g[q]}$ normalized by
$$
\Ft_{A_{\g[q]}}\psi(\lambda)=\int_{A_{\g[q]}}\psi(a)a^{-\lambda}\,da
$$
for $\psi\in L^{1}(A_{\g[q]})$ and $\lambda\in\g[a]_{\g[q],\C}^{*}$. Note that $\Ft_{A_{\g[q]}}\psi=\Ft_{\g[a]_{\g[q]}}(\psi\circ\exp)$.
For $g\in G$ and $a\in A_{\g[q]}$, we write $\Ht[P_{0}]\phi(g)(a) =a^{\rho_{P_{0}}}\Rt[P_{0}]\phi(ga\cdot \xi_{P_{0}})$.

\begin{prop}\label{FT prop FT_e=FT_A circ Rt}
Let $B$ be a compact subset of $\g[a]_{\g[q]}$ , let $\eta\in
V(\zeta,e)$ and let $\lambda\in\g[a]^{*}_{\g[q]}(\overline{P}_{0},0)+\rho_{P_{0}}$.
If $\phi\in \Es^{1}(X)$ satisfies $\supp(\Rt[P_{0}]\phi)\subseteq\Xi_{P_{0}}(B+\Gamma_{P_{0}})$, then
\begin{equation}\label{FT eq int_X phi j(P:zeta:-lambda)}
\int_{X}\phi(x)j(P_{0}:\zeta:-\lambda)(\eta)(g\cdot x)\,dx
=\int_{M_{P_{0}}\cap K}\Ft_{A_{\g[q]}}\Big(\Ht[P_{0}]\phi(g^{-1}m)\Big)(\lambda)\zeta(m)\eta\,dm,
\end{equation}
where the integrals converge absolutely for every $g\in G$.
\end{prop}

\begin{proof}
Let $\phi$ satisfy the above hypotheses. For $\lambda$ as stated,
the function $j(P_{0}: \zeta : -\lambda)(\eta)$ is continuous. We
will first prove the assertions under the assumption that $g =e$.

By Proposition \ref{RadTransExt prop Rt_P^Q: Es^1 to Es^1}, the function $\Rt[P_{0}]\phi$
is an element of $\Es^{1}(\Xi_{P_{0}},J_{P_{0}})$.
In view of the condition on the support of $\Rt[P_{0}]\phi$ it now follows
by application of Lemma  \ref{FT lemma a^nu psi Schwartz}
that
\begin{equation}\label{Ft eq a_q ni Y mapsto Rt phi(k exp(Y))}
\g[a]_{\g[q]}\ni Y\mapsto
e^{-\lambda(Y)+\rho_{P_{0}}(Y)}\Rt[P_{0}]\phi\big(k\exp(Y)\cdot\xi_{P_{0}}\big)
=e^{-\lambda(Y)}\Ht[P_{0}]\phi(k)(\exp(Y))
\end{equation}
is a continuous family (with family parameter $k\in K$) of
functions in the Schwartz space
$\Ss(\g[a]_{\g[q]})$.
Therefore, the integral
$$
\int_{A_{\g[q]}}
    a^{-\lambda}\Ht[P_{0}]\phi(k^{-1}m)(a)\,da
$$
is absolutely convergent for all $k\in K$ and $m\in M_{P_{0}}\cap K$ and depends continuously on both.
Since $M_{P_{0}}\cap K$ is compact, the integral
\begin{equation}\label{FT eq int_M int_A a^(-lambda+rho)Rt phi zeta(m)eta}
\int_{M_{P_{0}}\cap K}\int_{A_{\g[q]}}
    a^{-\lambda}\Ht[P_{0}]\phi(k^{-1}m)(a)\zeta(m)\eta\,da\,dm
\end{equation}
is absolutely convergent as well. Clearly, this integral equals the right-hand side of (\ref{FT eq int_X phi j(P:zeta:-lambda)}), with $g = k$.

We will proceed to show that the integral also equals the left-hand side of (\ref{FT eq int_X phi j(P:zeta:-lambda)}). Indeed, substituting the
definition of the Radon transform in (\ref{FT eq int_M int_A a^(-lambda+rho)Rt phi zeta(m)eta}), we obtain the absolutely convergent integral
\begin{align}\label{FT eq int_M int_A int_N a^(-lambda+rho)phi zeta(m)eta=int_M int_A int_N phi j(P:zeta:-lambda)}
&\int_{M_{P_{0}}\cap K}\int_{A_{\g[q]}}\int_{N_{P_{0}}}
    a^{-\lambda+\rho_{P_{0}}}\phi(k^{-1}man\cdot\xi_{P_{0}})
                    \zeta(m)\eta\,dn\,da\,dm\\
\nonumber&\qquad=\int_{M_{P_{0}}\cap K}\int_{A_{\g[q]}}\int_{N_{P_{0}}}
    \phi(k^{-1}man\cdot\xi_{P_{0}}) j(P_{0}:\zeta: - \lambda)(\eta)(man\cdot x_{0})
                    \,dn\,da\,dm.
\end{align}
For the last equality we have used that $-\lambda\in\g[a]_{\g[q]}^{*}(P_{0},0)-\rho_{P_{0}}$,
so that $j(P_{0}:\zeta:-\lambda)(\eta)$ is the continuous function given by (\ref{FT eq def j}).
As $\eta\in V(\zeta, e)$, this continuous function is supported by the closure of the
set $P_{0}H$. As before, we denote $w^{-1}P_{0}w$ by $P_{0}^w$.
We now recall \cite[Theorem 1.2]{Olafsson_FourierAndPoissonTransformationAssociatedToASemisimpleSymmetricSpace}.
According to this result,
\begin{equation}\label{FT eq int_(P_0 cdot x0)=sum_W int_LcapK int_A_q int_N_P}
\int_{P_{0}\cdot x_{0}}\psi(x)\,dx
= \int_{M_{P_{0}}\cap K}\int_{A_{\g[q]}}\int_{N_{P_{0}}}
    \psi(man\cdot x_{0})\,dn\,da\,dm,
\end{equation}
for every $\psi \in C_{c}(X)$. As the multiplication map
$$
(M_{P_{0}}\cap K)/(M_{P_{0}}\cap K\cap H)\times A_{\g[q]}\times N_{P_{0}} \to P_{0}\cdot x_{0}
$$
is a diffeomorphism onto the open subset $P_{0}\cdot x_{0}$ of
$X$, it follows from Fubini's theorem that (\ref{FT eq int_(P_0 cdot x0)=sum_W int_LcapK
int_A_q int_N_P}) is valid for any measurable function $\psi: X
\to \C$, provided the integral on either one of the two sides of
the equation is absolutely convergent, and in that case the other
integral is absolutely convergent as well. Applying this result to
(\ref{FT eq int_M int_A int_N a^(-lambda+rho)phi zeta(m)eta=int_M
int_A int_N phi j(P:zeta:-lambda)}) we obtain the integral
$$
\int_{P_{0}\cdot x_{0}} \phi(k^{-1}\cdot x) j(P_{0}:\zeta:-\lambda)(\eta)(x)\,dx
=\int_{k^{-1}P_{0}\cdot x_{0}} \phi(x) j(P_{0}:\zeta:-\lambda)(\eta)(k\cdot x)\,dx,
$$
which we now can conclude to be absolutely convergent since (\ref{FT eq int_M int_A int_N a^(-lambda+rho)phi zeta(m)eta=int_M int_A int_N phi j(P:zeta:-lambda)}) is absolutely convergent.
As $j(P_{0}:\zeta:-\lambda)(\eta)(kx)$ is supported on $k^{-1}P_{0}\cdot x_{0}$, the latter integral
equals the left-hand side of (\ref{FT eq int_X phi j(P:zeta:-lambda)}). This completes the proof for $g=k\in K$.

Now let $g\in G$. Write $g = m_g a_g n_g k_g$,
with $m_{g}\in M_{P_{0}}$, $a_{g}\in A_{P_{0}}$, $n_{g}\in N_{P_{0}}$ and $k_{g}\in K$.
Then by the transformation properties of $j$, the integral at the left-hand side of (\ref{FT eq int_X phi j(P:zeta:-lambda)}) equals
$$
a_{g}^{-\lambda+\rho_{P_{0}}}\zeta(m_{g})
    \int_{X}\phi(x)j(P_{0}:\zeta:-\lambda)(\eta)(k_{g}\cdot x)\,dx.
$$
By what we proved above, this expression equals
\begin{align*}
&a_{g}^{-\lambda+\rho_{P_{0}}}\zeta(m_g) \int_{M_{P_{0}}\cap K}
    \Ft_{A_{q}}\Big(\Ht[P_{0}]\phi(k_g^{-1}m)\Big)(\lambda)                                                            \zeta(m)\eta\,dm\\
&\qquad=a_{g}^{-\lambda+\rho_{P_{0}}}\zeta(m_g) \int_{M_{P_{0}}\cap K}
    \Ft_{A_{q}}\Big(\Ht[P_{0}]\phi(k_g^{-1}n_{g}^{-1}m)\Big)(\lambda)                                                            \zeta(m)\eta\,dm\\
&\qquad=\int_{M_{P_{0}}\cap K}\Ft_{A_{q}}\Big(
    \Ht[P_{0}]\phi(k_g^{-1}n_{g}^{-1}a_{g}^{-1}m)\Big)(\lambda)\zeta(m_g m)\eta\,dm\\
&\qquad=\int_{M_{P_{0}}\cap K}\Ft_{A_{q}}\Big(
    \Ht[P_{0}]\phi(k_g^{-1}n_g^{-1}a_g^{-1}m_g^{-1}m)\Big)(\lambda)\zeta(m)\eta\,dm.
\end{align*}
Here we subsequently used that $N_{P}$ stabilizes $\xi_{P_{0}}$ and is normalized by $M_{P_{0}}\cap K$, the invariance of the measure of $A_{\g[q]}$, and the invariance of the measure of $M_{P_{0}}\cap K$.
We finally observe that the last integral
equals the right-hand side of (\ref{FT eq int_X phi j(P:zeta:-lambda)}).
\end{proof}

Following \cite{vdBan&Schlichtkrull_FourierTransformOnASemisimpleSymmetricSpace}, we define
the unnormalized Fourier transform
$\Ft^{\un}_{P_{0}}\phi(\zeta:\lambda)$ of a function
$\phi\in\Ds(X)$ to be the element of
$\Hom(V(\zeta),C^{-\infty}(P_{0}:\zeta:\lambda))$ given by
\begin{equation}\label{FT eq def Ft^un}
\Ft^{\un}_{P_{0}}\phi(\zeta:\lambda)\eta:g\mapsto\int_{X}\phi(x)j(P_{0}:\zeta:-\lambda)(\eta)(g\cdot x)\,dx
\end{equation}
for $\eta\in V(\zeta)$. This Fourier transform depends
meromorphically on $\lambda\in\g[a]_{\g[q],\C}^{*}$. For $\lambda\in
\g[a]_{\g[q]}^{*}(\overline{P}_{0},0)+\rho_{P_{0}}$, the dependence is
holomorphic, and the integral in (\ref{FT eq def Ft^un}) is
absolutely convergent.

If $\phi\in \Es^{1}(X)$, satisfies $\supp(\Rt[P_{0}]\phi)\subseteq\Xi_{P_{0}}(B+\Gamma_{P_{0}})$
for some compact subset $B$ of $\g[a]_{\g[q]}$, then we define
(the first component of) the unnormalized Fourier transform
$\Ft^{\un}_{P_{0},e}\phi(\zeta:\lambda)$, for
$\lambda\in\g[a]^{*}_{\g[q]}(\overline{P}_{0},0)+\rho_{P_{0}}$, to
be the homomorphism $V(\zeta,e) \to \Es(P_{0}: \zeta : \lambda)$ that for $\eta\in V(\zeta,e)$ is
given by the absolutely convergent integral (\ref{FT eq def Ft^un}).
We note that by Corollary \ref{RadTransSupp cor supp(Rt_Q phi)subseteq Xi_Q,B => supp(Rt_P phi)subseteq
Xi_P,B} $\Rt[P_{0}]\phi$ has support
in a set of the mentioned form if $\phi \in \Ds(X)$. In that case,
$$
\Ft^{\un}_{P_{0},e}\phi(\zeta:\lambda)
=\Ft^{\un}_{P_{0}}\phi(\zeta:\lambda)\big|_{V(\zeta,e)},
$$
for all $\lambda\in\g[a]_{\g[q]}^{*}(\overline{P}_{0}, 0)+\rho_{P_{0}}$.

\begin{prop}\label{FT prop F^un PW-est}
Let $B$ be a compact subset of $\g[a]_{\g[q]}$ and let $\Gamma$ be
a cone in $\g[a]_{\g[q],\C}^{*}$ generated by a compact subset of
$\g[a]_{\g[q]}^{*}(\overline{P}_{0},0)$. For sufficiently large
$R>0$ there exist for every $N\in\N$ a constant $C_{N}>0$ and a
finite set $F_N\subset\U(\g)$, such that the following holds. For
all $\phi\in \Es^{1}(X)$ satisfying
$$\supp(\Rt_{P_{0}}\phi)\subseteq \Xi_{P_{0}}(B+\Gamma_{P_{0}}),$$
for all $k\in K$, $\eta\in V(\zeta,e)$ and
all $\lambda\in\Gamma$ with $\|\lambda\|>R$,
$$
\|\Ft^{\un}_{P_{0},e}\phi(\zeta:\lambda)\eta(k)\| \leq
C_{N}\sum_{u\in F_{N}}\|u\phi\|_{L^{1}(X)}(1+\|\lambda\|)^{-N}e^{H_{B}(-\Re(\lambda))}\|\eta\|.
$$
Let $\eta\in V(\zeta,e)$. The function
\begin{equation}\label{FT eq (k, lambda) mapsto Ft^(un)_(P,e)phi(lambda)(k)}
K\times\big(\g[a]^{*}_{\g[q]}(\overline{P}_{0},0)+\rho_{P_{0}}\big)\ni
 (k,\lambda) \mapsto\Ft^{\un}_{P_{0},e}\phi(\zeta:\lambda)\eta(k)
\end{equation}
is smooth and holomorphic in the second variable.
\end{prop}

\begin{proof}
Let $\nu_{0}\in\g[a]^{*}_{\g[q]}(\overline{P}_{0},0)+\rho_{P_{0}}$. Let $k\in K$ and $m\in M_{P_{0}}\cap K$.
Then $-\nu_{0} + \rho_{P_{0}} \in \g[a]_{\g[q]}^{*}(P_{0},0)$, so
that by Lemma \ref{FT lemma a^nu psi Schwartz} the function
$$
\g[a]_{\g[q]}\ni Y\mapsto
    e^{-\nu_{0}(Y)}\Ht[P_{0}]\phi(k^{-1}m)\big(\exp(Y)\big)
$$
belongs to $\Ss(\g[a]_{\g[q]})$ and is supported in
$B+\Gamma_{P_{0}}$. We now apply Proposition \ref{EFT
prop Paley-Wiener est for E(V)} with $\Gamma_{P_{0}}$ in place of
$\Gamma$, so that $-\Cone_{\Gamma}\supseteq\g[a]_{\g[q]}^{*}(\overline P_{0}, 0)\cap \g[a]_{\g[q]}^{*}$, and
with $u=\Ht[P_{0}]\phi(k^{-1}m)\circ\log$. This gives for each $N \in \N$ the existence of a constant $C_{\nu_{0}, N}>0$, such that for all
$\lambda\in\nu_{0}+\g[a]_{\g[q]}^{*}(\overline P_{0} ,0)$
$$
\Big|\Ft_{A_{q}}\Big(\Ht[P_{0}]\phi(k^{-1}m)\Big)(\lambda)\Big|
\leq C_{N,m^{-1}k}(\phi)(1+\|\lambda\|)^{-N}e^{H_{B}(-\Re(\lambda))}.
$$
Here
$$
C_{N,m^{-1}k}(\phi)
=C_{\nu_{0}, N}
    \big\|(1+\Delta_{A_{\g[q]}})^{N}\Big(e^{-\nu_{0}\circ\log}\Ht[P_{0}]\phi(k^{-1}m)\Big)\big\|_{L^{1}(A_{\g[q]})}
$$
We note that the constant $C_{\nu_{0}, N}$ is
independent of the function $\phi$.
In view of the definition of the unnormalized Fourier transform,
(\ref{FT eq def Ft^un}), and Proposition
\ref{FT prop FT_e=FT_A circ Rt} we now obtain
the estimate
\begin{eqnarray*}
\lefteqn{
    \|\Ft^{\un}_{P_{0},e}\phi(\zeta:\lambda)\eta(k)\|
    =\Big\|\int_{M_{P_{0}}\cap K}\Ft_{A_{q}}\Big(\Ht[P_{0}]\phi(k^{-1}m)\Big)(\lambda)\zeta(m)\eta\,dm\Big\|
    }\\
&&\leq\int_{M_{P_{0}}\cap K}\Big|\Ft_{A_{q}}\Big(\Ht[P_{0}]\phi(k^{-1}m)\Big)(\lambda)\Big|\ \|\zeta(m)\eta\|\,dm
\leq\widetilde{C}_{N,k}(\phi) (1+\|\lambda\|)^{-N}e^{H_{B}(-\Re(\lambda))}\|\eta\|
\end{eqnarray*}
for all $k\in K$ and $\lambda\in\nu_{0}+\g[a]_{\g[q]}^{*}(\overline P_{0},0)$.
For the last inequality we used that $\zeta$ is unitary and
we wrote
$$
\widetilde{C}_{N,k}(\phi) =
 \int_{M_{P_{0}} \cap K} C_{N,m^{-1}k}(\phi)\, dm.
$$
Using Leibniz' rule, the fact that $M_{P_{0}}\cap K$
centralizes $A_{\g[q]}$, and the fact that $a^{\rho_{P_{0}}-\nu_{0}}$ is bounded for $a\in\exp(B +
\Gamma_{P_{0}})$, we infer that there exists a constant ${\widetilde C}_{\nu_{0}, N}$ and a finite subset $F_{N}'
\subseteq S_{2N}(\g[a]_{\g[q]})$, such that
\begin{align*}
\widetilde{C}_{N,k}(\phi)
 &\leq{\widetilde C}_{\nu_{0}, N}\sum_{u \in F_{N}'} \int_{M_{P_{0}}\cap K}
 \big\|u\Big(a\mapsto \Rt[P_{0}]\big(L_{k}\phi\big)(ma\cdot\xi_{P_{0}})\Big)\big\|_{L^{1}(A_{\g[q]})}\,dm\\
&\leq
\widetilde{C}_{\nu_{0}, N}\sum_{u \in F_{N}'} \int_{M_{P_{0}}\cap K}
           \int_{A_{\g[q]}}\int_{N_{P_{0}}}
                |u\big(L_{k}\phi\big)(man\cdot x_{0})|\,dn\,da\,dm.
\end{align*}
We now note that $u(L_{k}\phi) = L_{k}([\Ad(k)^{-1} u]\phi)$ and that
$\Ad(k)^{-1} u$ is expressible in terms of a basis of $\U_{2N}(\g)$,
with coefficients that are continuous, hence bounded, functions of
$k \in K$. Combining this observation with
 (\ref{FT eq int_(P_0 cdot x0)=sum_W int_LcapK int_A_q int_N_P})
we see that there exists a finite subset $F_{N} \subseteq\U(\g)$
such that $\widetilde{C}_{N,k}(\phi)\leq \widetilde{C}_{\nu_{0}, N}\sum_{u \in F_{N}} \|u \phi\|_{L^1(X)}$.
In view of the previous estimates, we now conclude that for all
$\lambda\in\nu_{0}+\g[a]_{\g[q]}^{*}(\overline{P}_{0},0)$
$$
\|\Ft^{\un}_{P_{0},e}\phi(\zeta:\lambda)\eta(k)\| \leq
\widetilde{C}_{\nu_{0},N}\sum_{u\in F_{N}}\|u
\phi\|_{L^{1}(X)}(1+\|\lambda\|)^{-N}e^{H_{B}(-\Re(\lambda))}.
$$
Since
$\Gamma\subset\g[a]_{\g[q]}^{*}(\overline{P}_{0},0)$ is a
cone in $\g[a]_{\g[q],\C}^{*}$ generated by a compact subset, we have
for sufficiently large $R>0$ that $\{\lambda\in\Gamma:\|\lambda\|>R\}
\subseteq\nu_{0}+\g[a]_{\g[q]}^{*}(\overline{P}_{0},0)$.
The first statement now follows.

We address the second statement. Let $U$ be an open subset of
$\g[a]_{\g[q],\C}^{*}$ with compact closure in
$\g[a]_{\g[q]}^{*}(\overline P_{0}, 0)+\rho_{P_{0}}$. Then it
suffices to prove the smoothness and holomorphy of (\ref{FT eq (k,
lambda) mapsto Ft^(un)_(P,e)phi(lambda)(k)}) on $K \times U$.

Put $S = B + \Gamma_{P_{0}}$.
Then $\g[a]_{\g[q]}^{*}(\overline P_{0}, 0)\subseteq-\Cone_{S}$.
We note that
$\Ss_{S}(A_{\g[q]})=\{\psi \in \Ss(A_{\g[q]})\, : \; \supp(\psi) \subseteq S\}$
is a closed subspace of $\Ss(A_{\g[q]})$.
As in the proof of Proposition \ref{FT prop FT_e=FT_A circ Rt},
it follows that (\ref{Ft eq a_q ni Y mapsto Rt phi(k exp(Y))}) is a continuous family in $\Ss_{S}(A_{\g[q]})$,
with family parameter $k\in K$. As this also applies to $u \phi$, for every $u \in \U(\g[k])$,
it actually follows that (\ref{Ft eq a_q ni Y mapsto Rt phi(k exp(Y))}) is a smooth family in $\Ss_{S}(A_{\g[q]})$.
We now note that the Euclidean Fourier transform
defines a continuous linear map
$$
\Ft_{A_{\g[q]}}: \Ss_{S}(A_{\g[q]}) \to \mathcal{O}(-\Cone_{S})
=\mathcal{O}(\g[a]_{\g[q]}^{*}(\overline P_{0}, 0)).
$$
Here $\mathcal{O}(-\Cone_{S})$ denotes the space of holomorphic functions on $-\Cone_{S}$,
equipped with the usual Fr\'echet topology.
Combining this with the above assertion about smooth families we infer that $k \mapsto \Ft_{A_{\g[q]}} \Big(e^{-\lambda_{0}\circ\log}\Ht[P_{0}]\phi(k)\Big)$
is a smooth function on $K$ with values in the Fr\'echet space
$\mathcal{O}(\g[a]_{\g[q]}^{*}(\overline P_{0}, 0))$. Fix $\lambda_{0}\in\g[a]_{\g[q]}^{*}(\overline{P}_{0},0)+\rho_{P_{0}}$
such that $U\subseteq \g[a]_{\g[q]}^{*}(\overline{P}_{0},0)+\lambda_{0}$.
It follows that
$$
(k,\lambda)\mapsto
\Ft_{A_{\g[q]}}\Big(\exp^{-\lambda_{0}\circ\log}\Ht[P_{0}]\phi(k)\Big)(\lambda)
=\Ft_{A_{\g[q]}}\Big(\Ht[P_{0}]\phi(k)\Big)(\lambda+\lambda_{0})
$$
is smooth on $K\times\g[a]_{\g[q]}^{*}(\overline{P}_{0},0)$ and in addition holomorphic in the second variable.
As $U\subseteq \lambda_{0}+\g[a]_{\g[q]}^{*}(\overline{P}_{0},0)$, it now follows from (\ref{FT eq def Ft^un}) and Proposition \ref{FT prop FT_e=FT_A circ Rt} that $\Ft^{\un}_{P_{0},e}\phi(\zeta:\lambda) \eta(k)$ is smooth in
$(k,\lambda) \in K \times U$ and in addition holomorphic in $\lambda \in U$.
\end{proof}

\subsection{The $\tau$-spherical Fourier transform $\Ft_{\overline{P}_{0},\tau}$}
\label{subsection The tau-spherical Fourier transform}
Let $\vartheta\subset\widehat{K}$ be finite. For a representation $(\pi,V)$ of $K$ we write
$V_{\vartheta}$ for the space of $K$-finite vectors with isotypes contained in $\vartheta$.

Let $V_{\tau}=C(K)_{\vartheta}$, where the set $\vartheta$ of isotypes is taken with respect to the left-regular representation of $K$ on $C(K)$,
and let $\tau=\tau_{\vartheta}$ be the
representation of $K$ on $V_{\tau}$ obtained from the
right-action. We equip $V_{\tau}$ with the inner product
induced from $L^{2}(K)$. With respect to this inner product,
$\tau$ is unitary.

As before, let $P_{0} \in \Psg(\g[a]_{\g[q]})$. In this section we will consider
the $\tau$-spherical Fourier transform $\Ft_{\overline P_{0},\tau}$ as defined in
Section 6 of
\cite{vdBan&Schlichtkrull_FourierTransformOnASemisimpleSymmetricSpace}.
Before we can write down the definition of this Fourier transform,
we need to introduce some notation.

We denote the restriction of $\tau$ to $M_{P_{0}}\cap K$ by
$\tau_{M_{P_{0}}}$. Let $\leftsuperscript{\circ}{\Cs}(\tau)$ be the formal
direct sum of Hilbert spaces
$$
\leftsuperscript{\circ}{\Cs}(\tau)=
\bigoplus_{w\in\WGs}\leftsuperscript{0}{\Cs(\tau)}_{w}.
$$
Here
$$
\leftsuperscript{0}{\Cs(\tau)}_{w}=C^{\infty}(M_{P_{0}}/w(M_{P_{0}}\cap H)w^{-1}:\tau_{M_{P_{0}}})
$$
is the finite dimensional Hilbert space of $\tau_{M_{P_{0}}}$-spherical functions on $M_{P_{0}}/w(M_{P_{0}}\cap H)w^{-1}$,
i.e., the Hilbert space of smooth functions $f:M_{P_{0}}/w(M_{P_{0}}\cap H)w^{-1}\to V_{\tau}$
satisfying
$$
f(k\cdot x)=\tau(k)f(x)\qquad(k\in M_{P_{0}}\cap K, x\in M_{P_{0}}/w(M_{P_{0}}\cap H)w^{-1}).$$
The inner product on
$\leftsuperscript{0}{\Cs(\tau)}_{w}$ is induced from the inner
product on the space of square integrable functions.

Let
 $\zeta\in \widehat{M_{P_{0}}}_{H}$. The space of smooth
functions $f:K\to\H_{\zeta}\otimes V_{\tau}$ satisfying
$$
f(mk_{0}k)
=\big(\zeta(m)\otimes\tau(k)^{-1}\big)f(k_{0})\quad\text{for}\quad
    k,k_{0}\in K, m\in M_{P_{0}}\cap K
$$
is denoted by $\Es(K:\zeta:\tau)$.
Note that evaluation at the identity element induces a linear isomorphism
$$
f \mapsto f(e),\quad
\Es(K:\zeta:\tau)\to \big(\H_{\zeta}\otimes V_{\tau}\big)^{M_{P_{0}} \cap K}.
$$
Let $\overline{V(\zeta)}$ be
the conjugate vector space of $V(\zeta)$.
Following \cite[p. 528]{vdBan&Schlichtkrull_FourierTransformOnASemisimpleSymmetricSpace}, we
define a linear map
$$
\Es(K:\zeta:\tau)\otimes\overline{V(\zeta)}\to\leftsuperscript{0}{\Cs(\tau)};
\quad T\mapsto\psi_{T}$$
by
$$
\big(  \psi_{f\otimes\eta}\big  )_{w}(m) =
\langle f(e) | \zeta(m)\eta_{w} \rangle_{\H_{\zeta}}
\qquad (m \in M_{P_{0}}/w(M_{P_{0}}\cap H )w^{-1}),
$$
for $f\in \Es(K:\zeta:\tau)$ and $\eta\in\overline{V(\zeta)}$.
Here $\langle\cdot|\cdot\rangle_{\H_{\zeta}}$ denotes the inner product on $\H_{\zeta}$.
Let $\Ds(X:\tau)$ be the space of compactly supported smooth
functions $f:X\to V_{\tau}$ satisfying $f(k\cdot x)=\tau(k)f(x)$.
We define $\varsigma:\Ds(X)_{\vartheta}\to \Ds(X:\tau)$ by $\varsigma(\phi)(x)(k)=\phi(kx)$.
This map is a bijection. (See \cite[Lemma
5]{vdBan&Schlichtkrull_FourierTransformOnASemisimpleSymmetricSpace}.)

Restriction to $K$ induces a linear isomorphism from
$\Es(P_{0}:\zeta:\lambda)_{\vartheta}$ onto $\Es(K:\zeta:\tau)$.
Using this isomorphism and the linear isomorphism
$\overline{V(\zeta)} \to V(\zeta)^*$
(defined via the Hermitian inner product on $V(\zeta)$) we may view
$\Ft^{\un}_{P_{0}}(\phi)(\zeta : \lambda)$ as an element of
$\Es(K:\zeta:\tau)\otimes\overline{V(\zeta)}$, for all
$\lambda \in \g[a]_{\g[q]}^{*}(\overline P_{0}, 0) + \rho_{P_{0}}$.
It thus makes sense
to consider inner products between
$\Ft_{P_{0}}^{\un}\phi(\zeta:\lambda)$ and elements in
$\Es(K:\zeta:\tau)\otimes\overline{V(\zeta)}$.

The $\tau$-spherical Fourier Transform $\Ft_{\overline{P}_{0},\tau}$
is the linear transform from the space $\Ds(X:\tau)$ to the space of
meromorphic $\leftsuperscript{0}{\Cs(\tau)}$-valued functions on
$\g[a]_{\g[q],\C}^{*}$,
defined in \cite[(59)]{vdBan&Schlichtkrull_FourierTransformOnASemisimpleSymmetricSpace}.
For our purposes, it is sufficient to use the following characterization
in terms of the unnormalized Fourier transform discussed in the previous section.
Let $\widehat{M_{P_{0}}}_{H}(\tau)$ denote the finite collection
of representations $\zeta \in \widehat{M_{P_{0}}}_{H}$ such that
$\zeta|_{M_{P_{0}}\cap K}$ and $\tau|_{M_{P_{0}} \cap K}$ have a
$(M_{P_{0}}\cap K)$-type in common.
Then the $\tau$-spherical Fourier transform is completely determined by the requirement
that
\begin{equation}\label{FT eq def tau-sph Ft}
\langle \Ft_{\overline{P}_{0},\tau}\varsigma(\phi)(\lambda)|\psi_{f\otimes\eta}\rangle
=\langle \Ft_{P_{0}}^{\un}\phi(\zeta:\lambda)|\big(A(\overline{P}_{0}:P_{0}:\zeta:\overline{\lambda})^{-1}f\big)\otimes\eta\rangle
\end{equation}
for $\phi\in \Ds(X)_{\vartheta}$, $\zeta \in\widehat{M_{P_{0}}}_{H}(\tau)$, $f\in \Es(K:\zeta:\tau)$,
$\eta\in\overline{V(\zeta)}$ and generic $\lambda \in \g[a]_{\g[q],\C}^{*}$.
Here
$$
A(\overline{P}_{0}:P_{0};\zeta:\lambda):
\Es(P_{0}:\zeta:\lambda)\to\Es(\overline{P}_{0}:\zeta:\lambda)
$$
is the standard intertwining operator. It is initially defined
for $\lambda\in\g[a]_{\g[q]}^{*}(\overline{P}_{0},-R)$,
with  $R>0$ sufficiently large, by the
absolutely convergent integral
$$
A(\overline{P}_{0}:P_{0};\zeta:\lambda)\phi(g) =
\int_{\overline{N_{P_{0}}}}\phi(ng)\,dn, \qquad
(\phi\in\Es(P_{0}:\zeta:\lambda),\; g \in G).
$$
For the remaining $\lambda \in \g[a]_{\g[q],\C}^{*}$ it is defined
by meromorphic continuation. (See \cite[Theorem 4.2]{vdBan_ThePrincipalSeriesForAReductiveSymmetricSpaceI}, or alternatively \cite[Theorem 1.13]{VoganWallach_IntertwiningOperatorsForRealReductiveGroups}.) In (\ref{FT eq def tau-sph Ft}) the topological linear isomorphisms $\Es(Q:\zeta:\lambda)_{\vartheta}\to \Es(K:\zeta:\tau)$,
given by $f \mapsto \varsigma(f\big|_K)$, were used for $Q=P_{0}$ and
$Q=\overline{P}_{0}$ to view
$A(\overline{P}_{0}:P_{0};\zeta:\lambda)$ as an endomorphism of the space
$\Es(K:\zeta:\tau)$.

To see that the definition for $\Ft_{\overline{P}_{0},\tau}$ is in
fact equivalent to the defining identity
\cite[(59)]{vdBan&Schlichtkrull_FourierTransformOnASemisimpleSymmetricSpace},
use subsequently loc. cit. (59), (50) with $P$ and $P'$ replaced
by $\overline{P}_{0}$ and $P_{0}$ respectively, (47) and the
identity similar to the one in Proposition 3 for the unnormalized
Fourier transforms. The last mentioned identity is obtained from
the proof of Proposition 3 by using (30) instead of (53).

If $\phi\in \Es^{1}(X)_{\vartheta}$ satisfies $\supp(\Rt[P_{0}]\phi)\subseteq\Xi_{P_{0}}(B+\Gamma_{P_{0}})$,
then by definition the
unnormalized Fourier transform
$\Ft^{\un}_{P_{0},e}\phi(\zeta:\lambda)$ is an element of
the space $\Hom\big(V(\zeta,e), \Es(P_{0}:\zeta:\lambda)\big)$ for
$\lambda\in\g[a]^{*}_{\g[q]}(\overline{P}_{0},0)+\rho_{P_{0}}$.
In accordance with (\ref{FT eq def tau-sph Ft})
we now define the (first component of the) spherical
Fourier transform
$\Ft_{\overline{P}_{0},\tau,e}\varsigma(\phi)(\lambda)$ of such a
function $\phi$ to be the meromorphic
$\leftsuperscript{0}{\Cs(\tau)}_{e}$-valued
function on $\g[a]_{\g[q]}^{*}(\overline P_{0}, 0) + \rho_{P_{0}}$  given by
\begin{equation}\label{FT eq def tau-sph Ft_e}
\langle\Ft_{\overline{P}_{0},\tau,e}\varsigma(\phi)(\lambda)|\psi_{f\otimes\eta}\rangle
=\langle\Ft^{\un}_{P_{0},e}\phi(\zeta:\lambda)|\big(A(\overline{P}_{0}:P_{0}:\zeta:\overline{\lambda})^{-1}f\big)\otimes\eta\rangle
\end{equation}
for $\zeta\in\widehat{M_{P_{0}}}_{H}(\tau)$,
$f\in \Es(K:\zeta:\tau)$, $\eta\in\overline{V(\zeta,e)}$ and
generic  $\lambda\in \g[a]_{\g[q]}^{*}(\overline{P}_{0},0)+\rho_{P_{0}}$.
Note that this definition is compatible with (\ref{FT eq def tau-sph Ft}).
If $\phi\in\Ds(X)_{\vartheta}$, then
$\Ft_{\overline{P}_{0},\tau,e}\varsigma(\phi)(\lambda)
=\pr_{e}\Ft_{\overline{P}_{0},\tau}\varsigma(\phi)(\lambda)$,
for generic $\lambda \in \g[a]_{\g[q]}^{*}(\overline P_{0}, 0) +
\rho_{P_{0}}$. Here $\pr_{e}$ denotes the projection
$\leftsuperscript{0}{\Cs(\tau)}\to\leftsuperscript{0}{\Cs(\tau)}_{e}$.

Let $A(\overline P_{0}: P_{0}: \zeta : \lambda)_\tau$ denote
the standard intertwining operator, viewed as an endomorphism
of $\Es(K: \zeta : \tau)$. It follows from \cite[Lemma
16.6]{vdBan_ThePrincipalSeriesForAReductiveSymmetricSpaceII}
that there exists a polynomial function
$\pi: \g[a]_{\g[q],\C}^{*} \to \C$, which is
a product of linear factors
$\lambda\mapsto\langle\lambda,\alpha\rangle-c$ with
$\alpha\in\Sigma(\g,\g[a]_{\g[q]})$ and $c\in\R$, such that $\lambda\mapsto
\pi(\lambda)A(\overline{P}_{0}:P_{0}:\zeta:\overline{\lambda})_\tau^{-1}$
is a holomorphic $\End(\Es(K : \zeta : \tau))-$valued function
on $\g[a]_{\g[q]}^{*}(\overline{P}_{0},0)$ and
satisfies a polynomial estimate of the form
$$
\|\pi(\lambda)A(\overline{P}_{0}:P_{0}:\zeta:\overline{\lambda})_\tau^{-1}\|
\leq C ( 1 + \|\lambda\|)^N\qquad (\lambda \in \g[a]_{\g[q]}^{*}(\overline P_{0} : 0)),
$$
for some $C > 0$ and $N \in \N$. Here we note that the space
$\End(\Es(K : \zeta : \tau))$ is finite dimensional.

The following proposition is a direct corollary of
Proposition \ref{FT prop F^un PW-est} and
(\ref{FT eq def tau-sph Ft_e}).

\begin{prop}\label{FT prop PW-est for Ft_tau}
Let $B$ be a compact subset of $\g[a]_{\g[q]}$. If
$\phi\in \Es^{1}(X)_{\vartheta}$ satisfies
\begin{equation}\label{FT eq supp Rt phi subseteq Xi(B+Gamma)}
\supp(\Rt[P_{0}]\phi)\subseteq\Xi_{P_{0}}(B+\Gamma_{P_{0}}),
\end{equation}
then the map
$$
\g[a]^{*}_{\g[q]}(\overline{P}_{0},0)+\rho_{P_{0}}\to\leftsuperscript{0}{\Cs(\tau)}_{e};\quad
    \lambda\mapsto \pi(\lambda)\Ft_{\overline{P}_{0},\tau,e}\varsigma(\phi)(\lambda)
$$
is holomorphic. Let  $\Gamma$ be
any cone in $\g[a]_{\g[q],\C}^{*}$ generated by a
compact subset of $\g[a]_{\g[q]}^{*}(\overline{P}_{0},0)$.
Then there exists a  constant $R>0$
and for every $N\in\N$ a constant $C_{N}>0 $
and a finite subset $F_{N} \subseteq \U(\g)$
such that for all $\phi\in \Es^{1}(X)_{\vartheta}$ satisfying
(\ref{FT eq supp Rt phi subseteq Xi(B+Gamma)}) the estimate
$$
\|\pi(\lambda)\Ft_{\overline{P}_{0},\tau,e}\varsigma(\phi)(\lambda)\|
\leq C_{N}\sum_{u\in F_{N}}\|u\phi\|_{L^{1}(X)}
(1+\|\lambda\|)^{-N}e^{H_{B}(-\Re(\lambda))}
$$
 is valid for all $\lambda\in\Gamma$ with $\|\lambda\|>R$.
\end{prop}

\subsection{Function Spaces}
\label{subsection Function Spaces}
As before, we assume that $P_{0}$ is a minimal $\sigma\circ\theta$-stable parabolic subgroup
containing $A$. Let
$$\invCone{P_{0}}=\bigcap_{w\in\WG[K\cap H]}\Gamma_{P_{0}^{w}}.$$
For a subset $S$ of $\g[a]_{\g[q]}$, we define $\Es^{1}(X;S)
=\big\{\phi\in\Es^{1}(X):\supp(\phi)\subseteq X(S)\big\}$.
We further define
$$
\Es^{1}_{P_{0}}(X)
=\RDs(X)+\bigcup_{\genfrac{}{}{0pt}{}{B\subset\g[a]_{\g[q]}}{B\textnormal{compact}}}
\Es^{1}(X;B+\invCone{P_{0}}).
$$
Here $\RDs(X)$ denotes the space of rapidly decreasing functions
on $X$, which is the intersection of the Harish-Chandra
$L^{p}$-Schwartz spaces $\Cs^{p}(X)$ for $p>0$. (See \cite[Section
17]{vdBan_ThePrincipalSeriesForAReductiveSymmetricSpaceII}.)

\begin{rem}
If $X$ is a Riemannian symmetric space (hence $\sigma=\theta$) or $X$ is a
Lie group (i.e., $G=G_{0}\times G_{0}$ for some reductive Lie
group $G_{0}$ of the Harish-Chandra class and
$H=\textnormal{diag}(G_{0})$), then $\WG[K\cap H]$ equals the full
Weyl group $W$. In these cases the cone $\invCone{P_{0}}$ is the
trivial cone $\{0\}$ so that $\Es^{1}_{P_{0}}(X)=\RDs(X)$ is
independent of $P_{0}$.
\end{rem}

\begin{prop}\label{FuncSpaces prop Ls^1_P G-invariant}
$\Es^{1}_{P_{0}}(X)$ is a $G$-invariant subspace of $\Es^{1}(X)$.
\end{prop}

\begin{proof}
Let $B \subseteq \g[a]_{\g[q]}$ be a compact subset.
Since $\RDs(X)$ and $\Es^{1}(X)$ are $G$-invariant,
it suffices to show that for every  $g \in G$
there exists a compact subset $B' \subseteq \g[a]_{\g[q]}$ such that
\begin{equation}\label{FuncSpaces eq g cdot X(B+Gamma^W_(K capH)) subseteq X(B'+Gamma^W_(K capH))}
g\cdot X(B + \Gamma_{P_{0}}^{W_{K\cap H}})\subseteq X(B' + \Gamma_{P_{0}}^{W_{K \cap H}}).
\end{equation}
As $Kg\cdot X(B + \Gamma_{P_{0}}^{W_{K\cap H}})$ is a
$K$-invariant subset of $X$, there exists a unique $W_{K\cap
H}$-invariant compact subset $S \subseteq \g[a]_{\g[q]}$ such that $Kg\cdot X(B+\invCone{P_{0}})=X(S)$.
We will finish the proof by showing that $S$ is contained in a set of the form
$B'+\Gamma_{P_{0}}^{W_{K\cap H}}$ with $B'$ compact.

Recall the definition of $\Acomp_{P_{0}}$ in (\ref{RadTrans eq A_P def}). Let $a\in A$ be such that $g\in KaK$.
Then by the convexity theorem of Van
den Ban (Theorem \ref{RadTransSupp Thm Convexity theorem}),
$$
\ch(S)+\Gamma\big(\Sigma_{-}(\g,\g[a]_{\g[q]};P_{0})\big) =
\Acomp_{P_{0}}\big(aK\exp(B+\invCone{P_{0}})H\big).
$$
By Lemma \ref{RadTrans lemma A_P=pi_(a_P cap q)A_KAN}, the set on the right-hand side is contained in $\Acomp_{P_{0}}(aK)+\Acomp_{P_{0}}\big(\exp(B+\invCone{P_{0}})H\big)$.
If we apply Theorem \ref{RadTransSupp Thm Convexity theorem} to the
second term, we obtain
\begin{equation}\label{FuncSpaces eq ch(S)+Gamma= pi_q(W a)+pi_q ch( W_(K cap H) B)+Gamma}
\ch(S)+\Gamma\big(\Sigma_{-}(\g,\g[a]_{\g[q]};P_{0})\big)
\subseteq\Acomp_{P_{0}}(aK)+\ch(\WG[K\cap H]\cdot B)
    +\invCone{P_{0}}+\Gamma\big(\Sigma_{-}(\g,\g[a]_{\g[q]};P_{0})\big).
\end{equation}
Put $B''=\Acomp_{P_{0}}(aK)+\ch(\WG[K\cap H]\cdot B)$. Note that $B''$ is compact. Both cones that appear on the right-hand side of (\ref{FuncSpaces eq ch(S)+Gamma= pi_q(W a)+pi_q ch( W_(K cap H) B)+Gamma}) are contained in $\Gamma_{P_{0}}$.
Hence, it follows from (\ref{FuncSpaces eq ch(S)+Gamma= pi_q(W a)+pi_q ch( W_(K cap H) B)+Gamma}) that
$S \subseteq B'' + \Gamma_{P_{0}}$.
By $W_{K\cap H}$-invariance of $S$ this implies that
$$
S
\subseteq\bigcap_{w\in\WG[K\cap H]}w\cdot(B''+\Gamma_{P_{0}})
=\bigcap_{w\in\WG[K\cap H]}(B''+\Gamma_{P_{0}^{w}}).
$$
The cones $\Gamma_{P_{0}^{w}}$
are finitely generated and $B''$ is compact, so that we may complete the proof by application of
Lemma \ref{FuncSpaces lemma cap_k (B+Gamma_k) subseteq B'+cap_k Gamma_k}.
\end{proof}

\begin{rem}
By inspection of the above proof, one readily sees that for every
compact subset $C\subseteq G$ there exists a compact $B'
\subseteq \g[a]_{\g[q]}$ such that (\ref{FuncSpaces eq g cdot
X(B+Gamma^W_(K capH)) subseteq X(B'+Gamma^W_(K capH))}) is valid
for all $g \in C$.
\end{rem}

\subsection{Inversion formula}
\label{subsection Inversion Formula}
We continue in the setting of Section \ref{subsection The tau-spherical Fourier transform}.
Let $X_{+}$ be the union of disjoint open subsets of $X$
$$X_{+}=\bigcup_{w\in\WGs}X\big(\g[a]_{\g[q]}^{+}(P_{0}^{w})\big).$$
Let $E_{+}(\overline{P}_{0}:\lambda:\cdot):X_{+}\to\Hom(\leftsuperscript{0}{\Cs(\tau)},V_{\tau})$
be defined by
$$
E_{+}(\overline{P}_{0}:\lambda:kaw\cdot x_{0})(\psi)=
    \tau(k)\Phi_{\overline{P}_{0},w}(\lambda:a)\psi_{w}(e)
$$
where $\Phi_{\overline{P}_{0},w}(\lambda,\cdot)$ is the
$\textnormal{End}\big(V_{\tau}^{K\cap M_{P_{0}}\cap wHw^{-1}}\big)$-valued
function on $A_{\g[q]}^{+}(\overline{P}_{0})$ defined in \cite[Section
10]{vdBan&Schlichtkrull_ExpansionsForEisensteinIntegralsOnSemisimpleSymmetricSpaces}.
Let $\nu\in\g[a]_{\g[q]}^{*}(\overline{P}_{0},0)$.
By \cite[Theorem
4.7]{vdBan&Schlichtkrull_FourierInversionOnAReductiveSymmetricSpace},
$$
\phi(x)
=|\WG|\int_{t\nu+i\g[a]_{\g[q]}^{*}}E_{+}(\overline{P}_{0}:\lambda:x)
    \Ft_{\overline{P}_{0},\tau}\varsigma(\phi)(\lambda)\,d\lambda
$$
if $\phi\in\RDs(X)_\vartheta$,
$x\in X_{+}$ and if $t>0$ is sufficiently large. This result can be partially extended
to the $K$-finite functions in the larger space $\Es^{1}_{P_{0}}(X)$.

\begin{prop}[Inversion Formula]\label{InvFor prop Inversion Formula}
If $\phi\in \Es^{1}_{P_{0}}(X)_{\vartheta}$ and $\nu\in\g[a]_{\g[q]}^{*}(\overline{P}_{0},0)$,
then for $k\in K$, $a\in A_{\g[q]}^{+}(\overline{P}_{0})$ and sufficiently large $t>0$
$$
\phi(ka\cdot x_{0})
=|\WG|\int_{t\nu+i\g[a]_{\g[q]}^{*}}
    \tau(k)\Phi_{\overline{P}_{0},e}(\lambda:a)\Big(\Ft_{\overline{P}_{0},\tau,e}\varsigma(\phi)(\lambda)(e)\Big)
        \,d\lambda.
$$
\end{prop}

\begin{proof}
It suffices to prove the proposition for
$\phi\in\Es^{1}(X;B+\invCone{P_{0}})_{\vartheta}$, where $B$ is a
compact subset of $\g[a]_{\g[q]}$. As
$\Es^{1}(X;B+\invCone{P_{0}})_{\vartheta}\cap\Ds(X)_{\vartheta}$
is dense in $\Es^{1}(X;B+\invCone{P_{0}})_{\vartheta}$, there
exists a sequence $(\phi_{j})_{j\in\N}$ in
$\Es^{1}(X;B+\invCone{P_{0}})_{\vartheta}\cap\Ds(X)_{\vartheta}$
converging to $\phi$ in
$\Es^{1}(X;B+\invCone{P_{0}})_{\vartheta}$. According to
\cite[Theorem
9.1]{vdBan&Schlichtkrull_ExpansionsForEisensteinIntegralsOnSemisimpleSymmetricSpaces},
the function $i\g[a]_{\g[a]}^{*}\ni\lambda\mapsto\Phi_{\overline{P}_{0},e}(t\nu+\lambda:a)$
is bounded if $t>0$ is sufficiently large.
Corollary \ref{RadTransSupp cor supp(Rt_Q phi)subseteq Xi_Q,B => supp(Rt_P phi)subseteq Xi_P,B} and the Paley-Wiener estimate in Proposition \ref{FT prop PW-est for Ft_tau} therefore imply that,
for $t > 0$ sufficiently large,
$$
\lim_{j\to\infty}\int_{t\nu+i\g[a]_{\g[q]}^{*}}\left|
    \tau(k)\Phi_{\overline{P}_{0},e}(\lambda:a)\Big(\Ft_{\overline{P}_{0},\tau,e}\varsigma(\phi-\phi_{j})(\lambda)(e)\Big)\right|\,d\lambda
=0.
$$
\end{proof}

\subsection{A support theorem for the horospherical transform for functions}
\label{subsection Support Theorem for the Horospherical transform}

We can now prove a preliminary support theorem for the Radon
transform $\Rt[P_{0}]$ associated to a minimal $\sigma\circ
\theta$-stable parabolic subgroup $P_{0} \in \Psg(\g[a]_{\g[q]})$.
The proof is based on a Paley-Wiener type shift argument. At the
end of this section we will sharpen the preliminary result by
invoking the equivariance of the Radon transform.
We start with a lemma.

\begin{lemma}\label{SuppThm lemma W cdot B identity}
Let $B$ be a compact subset of $\g[a]_{\g[q]}$. Then
$$
\big\{Y\in\g[a]_{\g[q]}:
    \nu(Y)+H_{B}(-\nu)\geq0
\textnormal{ for all }\nu\in\g[a]_{\g[q]}^{*}
(\overline{P}_{0},0)\cap\g[a]_{\g[q]}^{*}\big\}
=\ch(B)+\Gamma_{P_{0}}.
$$
\end{lemma}

\begin{proof}
The lemma is a direct corollary of Lemma \ref{ConvexGeom lemma W cdot B identity} as the closure of $\g[a]_{\g[q]}^{*}
(\overline{P}_{0},0)\cap\g[a]_{\g[q]}^{*}$ equals $-\Cone_{\Gamma_{P_{0}}}$.
\end{proof}

\begin{prop}\label{SuppThm prop supp Rt phi subset Xi_P(B+Gamma)->supp phi cap X(a^+(Pc)) subset X(B)}
Let $B$ be a convex compact subset of $\g[a]_{\g[q]}$ and let
$\phi\in\Es^{1}_{P_{0}}(X)$. If
$$\supp(\Rt[P_{0}]\phi)\subseteq \Xi_{P_{0}}(B+\Gamma_{P_{0}}),$$
then
$$\supp(\phi)\cap A_{\g[q]}^{+}(\overline{P}_{0})\cdot x_{0}
\subseteq \Big(\exp(B+\Gamma_{P_{0}})\cap A_{\g[q]}^{+}(\overline{P}_{0})\Big)\cdot x_{0}.$$
\end{prop}

\begin{proof}
Because of equivariance and continuity of $\Rt[P_{0}]$ it
suffices to prove the claim for $K$-finite functions
$\phi$. We will therefore assume that
$\phi$ is $K$-finite with isotypes contained in a finite subset $\vartheta$ of $\widehat{K}$.

Assume that $a\in A_{\g[q]}^{+}(\overline{P}_{0})$, but $\log
a\notin B+\Gamma_{P_{0}}$. By Lemma \ref{SuppThm lemma W cdot B
identity} there exists a
$\nu\in\g[a]_{\g[q]}^{*}(\overline{P}_{0},0)\cap\g[a]_{\g[q]}^{*}$
such that $\nu(\log a)+H_{\log(B)}(-\nu)<0$.
According to \cite[Theorem
9.1]{vdBan&Schlichtkrull_ExpansionsForEisensteinIntegralsOnSemisimpleSymmetricSpaces}
there exists a constant $c>0$ such that
$\|\Phi_{\overline{P}_{0},e}(\lambda:a)\|<c a^{t\nu}$
for all sufficiently large $t>0$ and $\lambda\in t\nu+i\g[a]_{\g[q]}^{*}$.
The Paley-Wiener estimate for $\Ft_{\overline{P}_{0},\tau}\phi$ (Proposition \ref{FT prop PW-est for Ft_tau}) and the inversion formula (Proposition \ref{InvFor prop Inversion Formula}) imply that for every integer $N$ there exists a constant
$C_{N}>0$ such that for sufficiently large $t>0$
$$
|\phi(a\cdot x_{0})|
\leq C_{N}e^{t(\nu(\log a)+H_{B}(-\nu))}
    \int_{t\nu+i\g[a]_{\g[q]}^{*}}(1+\|\lambda\|)^{-N}\,d\lambda.
$$
Let $N \geq \dim \g[a]_{\g[q]} + 1;$ then by taking the limit for
$t\to\infty$ we find $\phi(a\cdot x_{0})=0$.
\end{proof}

Let $\WG(\g[a])$ be the Weyl group of the root system $\Sigma(\g;\g[a])$.

\begin{lemma}\label{SuppThm lemma kak Xi_P(B+Gamma) subseteq Xi( ch W cdot a+B+Gamma)}
Let $S\subseteq \g[a]_{\g[q]}$, let $a \in A$ and let $g\in KaK$.
Then
$$
g\cdot\Xi_{P_{0}}(S)
\subseteq\Xi_{P_{0}}\Big(\pi_{\g[q]}\,\ch \big(\WG(\g[a])\cdot\log a\big)+S\Big).
$$
\end{lemma}

\begin{proof}
Let $P_{m}$ be a minimal parabolic subgroup of $G$ such that $A\subseteq P_{m}\subseteq P_{0}$. By Kostant's convexity theorem
(\cite[Theorem
4.1]{Kostant_OnConvexityTheWeylGroupAndTheIwasawaDecomposition}),
$gK\subseteq K\exp\Big(\ch\big(W(\g[a])\cdot\log a\big)\Big) N_{P_{m}}$.
Using that $N_{P_{m}} = N_{P_{m}}^{P_{0}} N_{P_{0}}$ and that
$N_{P_{m}}^{P_{0}}\subseteq M_{P_{0}}\cap H $ (see Lemma \ref{PSubgrp lemma n_MAN^(P_0)=n cap h})  we now find that
\begin{align*}
g\cdot\Xi_{P_{0}}(S)
&=gK\exp(S)\cdot\xi_{P_{0}}
\subseteq K\exp\Big(\ch\big(W(\g[a])\cdot\log a\big)\Big) N_{P_{m}}\exp(S)\cdot\xi_{P_{0}}\\
&= K\exp\Big(\ch\big(W(\g[a])\cdot\log a\big) +S\Big)\cdot\xi_{P_{0}}
= \Xi_{P_{0}}\Big(\pi_{\g[q]}\,\ch\big(W(\g[a])\cdot\log a\big)+S\Big).
\end{align*}
\end{proof}

We can now sharpen Proposition \ref{SuppThm prop supp Rt phi subset Xi_P(B+Gamma)->supp phi cap X(a^+(Pc)) subset X(B)} by using
the equivariance of the Radon transform.

\begin{thm}[Support theorem for the horospherical transform]\label{SuppThm thm Support Theorem for functions and minimal P_0}
Let $B$ be a convex compact subset of $\g[a]_{\g[q]}$ and let
$\phi\in\Es^{1}_{P_{0}}(X)$. If
$$\supp(\Rt[P_{0}]\phi)\subseteq \Xi_{P_{0}}(B+\Gamma_{P_{0}}),$$
then
\begin{equation}\label{SuppThm eq supp(phi)subseteq X(cap_w (B+Gamma_P^w))}
\supp(\phi)\subseteq X\Big(\bigcap_{w\in\WG[K\cap H]}w\cdot(B+\Gamma_{P_{0}})\Big).
\end{equation}
\end{thm}

\begin{proof}
Assume that $\phi$ satisfies the hypothesis.
We will first show that
\begin{equation}\label{SuppThm eq supp(phi)cap A_q cdot x_0 subseteq exp(B+Gamma_P)cdot x_0}
\supp(\phi)\cap A_{\g[q]}\cdot x_{0}
\subseteq\exp(B+\Gamma_{P_{0}})\cdot x_{0}.
\end{equation}
Let $Y_{0} \in \g[a]_{\g[q]}$ be such that $\exp
Y_{0}\in\supp(\phi)$. Then there exists a
$Y\in\g[a]_{\g[q]}^{+}(\overline{P}_{0})$ such that $Y_{0}+Y\in
\g[a]_{\g[q]}^{+}(\overline{P}_{0})$. By equivariance of
$\Rt[P_{0}]$ and by application of Lemma \ref{SuppThm lemma kak
Xi_P(B+Gamma) subseteq Xi( ch W cdot a+B+Gamma)} we find that
\begin{align*}
\supp\big(\Rt[P_{0}](\ltpb{\exp(-Y)}\phi)\big)
&=\exp(Y)\cdot \supp(\Rt[P_{0}]\phi)
\subseteq\exp(Y)\cdot\Xi_{P_{0}}(B+\Gamma_{P_{0}})\\
&\subseteq\Xi_{P_{0}}\Big(\pi_{\g[q]}\,\ch \big(\WG(\g[a])\cdot Y\big)+B+\Gamma_{P_{0}}\Big).
\end{align*}
From Propositions \ref{FuncSpaces prop Ls^1_P G-invariant} and
\ref{SuppThm prop supp Rt phi subset Xi_P(B+Gamma)->supp phi cap X(a^+(Pc)) subset X(B)} it now follows that
$$
\exp(Y_{0}+Y)\cdot x_{0}
\in\supp\big(\ltpb{\exp(-Y)}\phi\big)\cap A_{\g[q]}^{+}(\overline{P}_{0})\cdot x_{0}
\subseteq\exp\Big(\pi_{\g[q]}\,\ch \big(\WG(\g[a])\cdot Y\big)+B+\Gamma_{P_{0}}\Big)\cdot x_{0}.
$$
We conclude that $Y_{0}+Y\in\pi_{\g[q]}\,\ch \big(\WG(\g[a])\cdot
Y\big)+B+\Gamma_{P_{0}}$. On the other hand,
since $Y\in\g[a]_{\g[q]}^{+}(\overline{P}_{0})$,
$\pi_{\g[q]}\,\ch\big(\WG(\g[a])\cdot Y\big)\subseteq Y+\Gamma_{P_{0}}$.
We thus see that $Y_{0} + Y \in Y + B + \Gamma_{P_{0}}$, so that
$Y_{0} \in B + \Gamma_{P_{0}}$. We have proved (\ref{SuppThm eq
supp(phi)cap A_q cdot x_0 subseteq exp(B+Gamma_P)cdot x_0}).

If $k \in K$, then by equivariance of the Radon transform, the
function $\ltpb{k}\phi$ satisfies the same hypotheses as $\phi$,
so that (\ref{SuppThm eq supp(phi)cap A_q cdot x_0 subseteq
exp(B+Gamma_P)cdot x_0}) is valid with $\ltpb{k}\phi$ in place of
$\phi$. This implies that
$$
\supp(\phi)\cap  A_{\g[q]}\cdot x_{0}
\subseteq\bigcap_{w\in\mathcal{N}_{K\cap H}(\g[a]_{\g[q]})}w\cdot\exp(B+\Gamma_{P_{0}})\cdot x_{0}
=\exp \Big(\bigcap_{w\in W_{K\cap H}}w\cdot(B+\Gamma_{P_{0}})\Big)\cdot x_{0}.
$$
Invoking the $K$-equivariance of the Radon transform once more in
a similar way, we conclude that (\ref{SuppThm eq supp(phi)subseteq
X(cap_w (B+Gamma_P^w))}) holds.
\end{proof}

\begin{rem}
Let $B$ be a $\WG[K\cap H]$-invariant closed convex subset of $\g[a]_{\g[q]}$. If $\Cone_{\invCone{P_{0}}}$ equals $\bigcup_{w\in\WG[K\cap H]}\Cone_{\Gamma_{P_{0}^{w}}}$
then
\begin{equation}\label{SuppThm eq B+invCone=bigcap(B+Gamma_(P^w))}
\bigcap_{w\in \WG[K\cap H]}B+\Gamma_{P_{0}^{w}}=B+\invCone{P_{0}}
\end{equation}
by Lemma \ref{ConvexGeom Lemma S=cap(S+Gamma)}. This is in particular the case if $\WG=\WG[K\cap
H]$. (See Lemma \ref{RadTransSupp Lemma B+Gamma_P=cap(B+Gamma_P_0)}.)
In general there exists a compact subset $B'$ of $\g[a]_{\g[q]}$
such that
$$B+\invCone{P_{0}}
\subseteq \bigcap_{w\in \WG[K\cap H]}B+\Gamma_{P_{0}^{w}}
\subseteq B'+\invCone{P_{0}}
$$
(see Lemma \ref{FuncSpaces lemma cap_k (B+Gamma_k) subseteq
B'+cap_k Gamma_k}), but if $\Cone_{\invCone{P_{0}}}\neq\bigcup_{w\in\WG[K\cap H]}\Cone_{\Gamma_{P_{0}^{w}}}$, then (\ref{SuppThm eq
B+invCone=bigcap(B+Gamma_(P^w))}) is not necessarily true. The
following is a counterexample.

Let $G=\LG{SL}(3,\R)$. Let
$\theta$ be a Cartan involution for $G$ and $G=KAN$ an Iwasawa
decomposition such that $G^{\theta}=K$. The root system
$\Sigma(\g, \g[a])$ is of type $A_{2}$. Let $\Sigma^{+}(\g,\g[a])$
be a system of positive roots and let $\alpha$ and $\beta$ be the
simple roots in that system. Then $\Sigma^{+}(\g,\g[a])=\{\alpha, \beta,\alpha+\beta\}$.
\vspace{-5pt}
\setlength{\unitlength}{2 cm}
\begin{center}
\begin{picture}(1,1)(0,0)
\includegraphics[width=\unitlength]{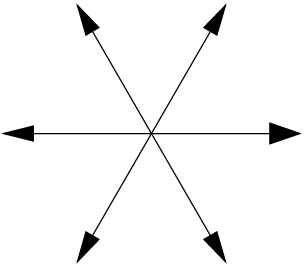}
\end{picture}
\begingroup\makeatletter\ifx\SetFigFont\undefined
\gdef\SetFigFont#1#2#3#4#5{
  \reset@font\fontsize{#1}{#2pt}
  \fontfamily{#3}\fontseries{#4}\fontshape{#5}
  \selectfont}
\fi\endgroup
\begin{picture}(1,1)(0,0)
\put(-0.3,0){\makebox(0,0)[lb]{\smash{{\SetFigFont{10}{12}{\rmdefault}{\mddefault}{\updefault}{$\beta$}}}}}
\put(-0.1,0.4){\makebox(0,0)[lb]{\smash{{\SetFigFont{10}{12}{\rmdefault}{\mddefault}{\updefault}{$\alpha+\beta$}}}}}
\put(-0.3,0.8){\makebox(0,0)[lb]{\smash{{\SetFigFont{10}{12}{\rmdefault}{\mddefault}{\updefault}{$\alpha$}}}}}
\end{picture}
\end{center}
Let $\epsilon:\Sigma(\g,\g[a])\to\{\pm1\}$ be given by $\epsilon(\pm\alpha)=\epsilon(\pm\beta)=-1$ and $\epsilon\big(\pm(\alpha+\beta)\big)=1$.
Let $\theta_{\epsilon}:\g\to\g$ be the Lie algebra involution given by
$$
\theta_{\epsilon}(Y)=
\begin{cases}
-Y&(Y\in\g[a])\\
\epsilon(\gamma)\theta(Y)&\big(Y\in\g_{\gamma}, \gamma\in\Sigma(\g,\g[a])\big).
\end{cases}
$$
The Lie algebra involution $\theta_{\epsilon}$ lifts to a Lie
group involution of $G$, which we also denote by
$\theta_{\epsilon}$. Let $K_{\epsilon}=G^{\theta_{\epsilon}}$ and
let $X=G/K_{\epsilon}$. (Then $K_{\epsilon}\simeq SO(2,1)$.) We claim that (\ref{SuppThm eq
B+invCone=bigcap(B+Gamma_(P^w))}) does not hold for every compact
convex Weyl-group invariant subset $B$ of $\g[a]_{\g[q]}$  in this
case.

The group $\WG[K\cap K_{\epsilon}]$ equals the Weyl group for the
root system
$\Sigma_{+}(\g^{\theta\circ\theta_{\epsilon}},\g[a])=\{\pm(\alpha+\beta)\}$.
The reflection $s$ in $\alpha+\beta$ maps $\alpha$ to $-\beta$ and
$\beta$ to $-\alpha$. Let $P$ be the minimal parabolic subgroup of
$G$ such that $A_{P}=A$ and
$\Sigma(\g,\g[a];P)=\{\alpha,\beta,\alpha+\beta\}$. Then $P$
is $\theta_{\epsilon}\circ\theta$-stable and $P^{s}=\overline{P}$.
Therefore, $\invCone{P}=\Gamma_{P}\cap\Gamma_{\overline{P}}=\{0\}$.
Let $B$ be the closed ball in $\g[a]_{\g[q]}$ with radius $r$,
centered at the origin. The angle between the root vectors
$H_{\alpha}$ and $H_{\alpha+\beta}$ equals the angle between
$H_{\beta}$ and $H_{\alpha+\beta}$; both are equal to
$\frac{\pi}{3}$. Let $v$ be a vector perpendicular to
$H_{\alpha+\beta}$ and with length $r$, then a straightforward
calculation shows that $(B+\Gamma_{P})\cap(B+\Gamma_{\overline{P}})=\ch\big(B\cup\{\pm2v\}\big)$.
In pictures:

\setlength{\unitlength}{35pt}
\begin{center}\begin{tabular}{ccccc}
\parbox[c]{1 \unitlength}{
    \includegraphics[width=\unitlength]{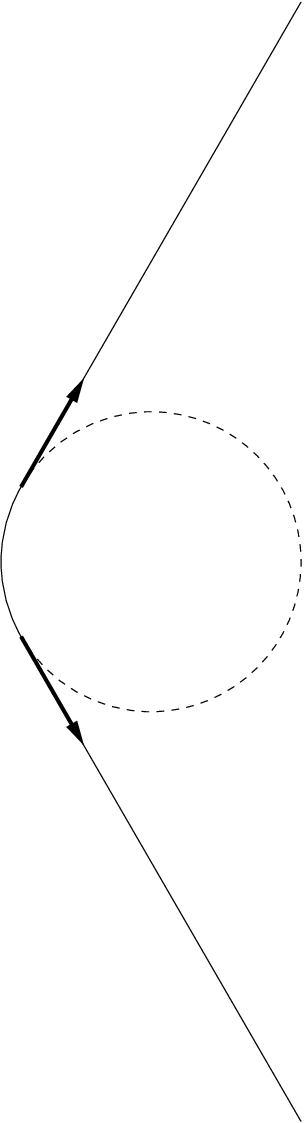}
    \begingroup\makeatletter\ifx\SetFigFont\undefined
    \gdef\SetFigFont#1#2#3#4#5{
    \reset@font\fontsize{#1}{#2pt}
    \fontfamily{#3}\fontseries{#4}\fontshape{#5}
    \selectfont}
    \fi\endgroup
    \begin{picture}(1,0)(0,0)
     \put(-0.5,2.7){\makebox(0,0)[lb]{\smash{{\SetFigFont{10}{12}{\rmdefault}{\mddefault}{\updefault}{$H_{\alpha}$}}}}}
     \put(-0.5,1.6){\makebox(0,0)[lb]{\smash{{\SetFigFont{10}{12}{\rmdefault}{\mddefault}{\updefault}{$H_{\beta}$}}}}}
    \end{picture}
    }
&\parbox[c]{\unitlength}{\begin{center}$\bigcap$\end{center}}
&\parbox[c]{47.7pt}{\reflectbox{
    \includegraphics[width=\unitlength]{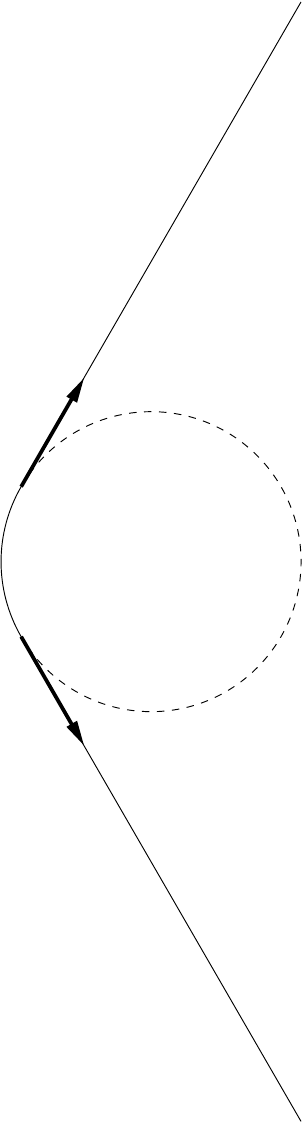}}
    \begingroup\makeatletter\ifx\SetFigFont\undefined
    \gdef\SetFigFont#1#2#3#4#5{
    \reset@font\fontsize{#1}{#2pt}
    \fontfamily{#3}\fontseries{#4}\fontshape{#5}
    \selectfont}
    \fi\endgroup
    \begin{picture}(1,0)(0,0)
     \put(0.7,1.6){\makebox(0,0)[lb]{\smash{{\SetFigFont{10}{12}{\rmdefault}{\mddefault}{\updefault}{$-H_{\alpha}$}}}}}
     \put(0.7,2.7){\makebox(0,0)[lb]{\smash{{\SetFigFont{10}{12}{\rmdefault}{\mddefault}{\updefault}{$-H_{\beta}$}}}}}
    \end{picture}
     }
&\parbox[c]{0.5\unitlength}{\begin{center}$=$\end{center}}
&\parbox[c]{\unitlength}{
    \includegraphics[width=\unitlength]{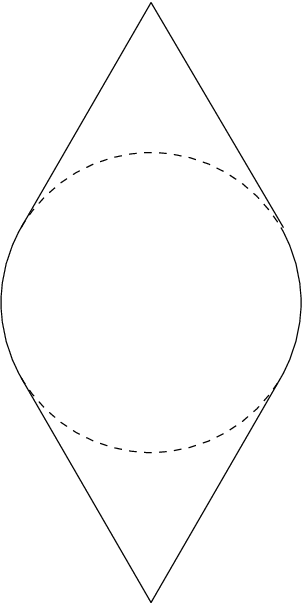}
    \begin{picture}(1,0)(0,0)
    \end{picture}
    }
\end{tabular}
\end{center}
\end{rem}

\section{Support theorems}
\label{section Support theorems}
The support theorem (Theorem \ref{SuppThm thm Support Theorem for functions and
minimal P_0}) for the horospherical transform for functions can be generalized to a support theorem for the
Radon transform $\Rt[P]$ corresponding to a (possibly non-minimal)
$\sigma\circ\theta$-stable parabolic subgroup $P$ for
distributions in a suitable subspace of $\Es_{b}'(X)$. In Section \ref{subsection Spaces of distributions}
we describe the spaces of distributions needed to formulate the
support theorem in Section \ref{subsection Support theorems}. The
support theorem implies injectivity of the Radon transform on
these spaces of distributions. In Section \ref{subsection
Injectivity} we discuss some implications of this for generalizing
the support theorem to even larger spaces of distributions.

Throughout this section $P$ is assumed to be a
$\sigma\circ\theta$-stable parabolic subgroup of $G$ containing
$A$.

\subsection{Spaces of distributions}
\label{subsection Spaces of distributions}
We define the convolution product $\theta*\phi$ of
$\theta\in\Ds(G)$ and $\phi\in\Es_{b}(X)$ to be the function on
$X$ given by
$$\theta*\phi(x)=\int_{G}\theta(g)\phi(g^{-1}\cdot x)\,dg\qquad(x\in X).$$
Since the left-regular representation of $G$ on the Fr\'echet
space $\Es_{b}(X)$ is continuous, it follows from standard
representation theory that convolution with a compactly supported
smooth function $\theta$ on $G$ defines a continuous operator from
$\Es_{b}(X)$ to itself.

For $\theta\in\Ds(G)$ define $\check{\theta}\in\Ds(G)$ by $\check{\theta}(g)=\theta(g^{-1})$. For $\theta\in\Ds(G)$ and  $\mu\in\Es_{b}'(X)$ we define the convolution product $\theta*\mu\in\Es_{b}'(X)$ by $\theta*\mu(\phi)=\mu(\check{\theta}*\phi)$ for $\phi\in\Es_{b}(X)$. Note that $\Es_{b}'(X)\ni\mu\mapsto\theta*\mu$ is continuous. Note further that the distribution $\theta*\mu$,
with $\theta\in\Ds(G)$ and $\mu\in\Es_{b}'(X)$, defines a smooth
function.

For $\psi\in\Es_{b}(\Xi_{P},J_{P}^{-1})$ and $\theta\in\Ds(G)$ we
define the convolution product $\theta*\psi$ to be the
function on $\Xi_{P}$ given by
$$\theta*\psi(\xi)=\int_{G}\theta(g)\psi(g^{-1}\cdot\xi)\,dg.$$
As above it follows that $\psi\mapsto\theta*\psi$ is a continuous endomorphism of $\Es_{b}(\Xi_{P},J_{P}^{-1})$.

Let $\invCone{P}$ be the maximal $\WG[K\cap H]$-invariant subcone of $\Gamma_{P}$
and let $\Es^{1}_{P}(X)$ be the subspace of $\Es^{1}(X)$ given by
$$
\Es^{1}_{P}(X)
=\RDs(X)+\bigcup_{\genfrac{}{}{0pt}{}
       {B\subset\g[a]_{\g[q]}}{B\textnormal{ compact}}}\Es^{1}(X;B+\invCone{P}).
$$
For $P\in\Psg(\g[a]_{\g[q]})$ these definitions agree with the definitions given
in the beginning of Section \ref{subsection Function Spaces}.

\begin{prop}\label{FuncSpaces prop Es^1_P=bigcap_(P_0)Es^1_(P_0)}
Let $\mathscr{C}$ denote the collection of $P_{0}\in \Psg(\g[a]_{\g[q]})$
contained in $P$.
Then
$$
\Es^{1}_{P}(X)=\bigcap_{P_{0}\in\mathscr{C}}\Es^{1}_{P_{0}}(X).
$$
In particular, $\Es^{1}_{P}(X)$ is invariant under the left action by $G$.
\end{prop}

\begin{proof}
If $P_{0} \in \mathscr{C}$ then $\invCone{P}\subseteq \invCone{P_{0}}$, hence
$\Es^{1}_{P}(X)\subseteq\Es^{1}_{P_{0}}(X)$.
It follows that $\Es^{1}_{P}(X)$ is contained
in the given intersection.

For the remaining inclusion, assume that
$$
\phi\in \bigcup_{\genfrac{}{}{0pt}{}{B\subset\g[a]_{\g[q]}}{B\textnormal{compact}}}
\Es^{1}(X;B+\invCone{P_{0}})
$$
for each $P_{0}\in \mathscr{C}$. Then for every such  $P_{0}$
there exists a compact subset $B_{P_{0}}$ of $\g[a]_{\g[q]}$ such that
$\supp(\phi)\subseteq X(B_{P_{0}}+\invCone{P_{0}})$.
Let $B$ be a $\WG[K\cap H]$-invariant compact subset of
 $\g[a]_{\g[q]}$ containing the (finite) union of the sets $B_{P_{0}}$.
Then
$$
\supp(\phi)
\subseteq X\Big(\bigcap_{P_{0}\in\mathscr{C}}(B+\invCone{P_{0}})\Big)
\subseteq X\Big(\bigcap_{P_{0}\in\mathscr{C}}\bigcap_{w\in\WG[K\cap H]}(B+\Gamma_{P_{0}^{w}})\Big).$$
In view of Lemma \ref{RadTransSupp Lemma B+Gamma_P=cap(B+Gamma_P_0)}
applied to $P^{w}$ and $\mathscr{C}^{w}$  it follows that $\supp(\phi)\subseteq X\Big(\bigcap_{w\in\WG[K\cap H]}(B+\Gamma_{P^{w}})\Big)$.
According to Lemma \ref{FuncSpaces lemma cap_k (B+Gamma_k)
subseteq B'+cap_k Gamma_k} there exists a compact subset $B'$ of
$\g[a]_{\g[q]}$ such that the support of $\phi$ is contained in
$X(B'+\invCone{P})$ and thus we conclude that
$\phi\in\Es^{1}_{P}(X)$.

The last assertion follows from the fact that each of the spaces
$\Es^{1}_{P_{0}}(X)$ is $G$-invariant by Proposition
\ref{FuncSpaces prop Ls^1_P G-invariant}.
\end{proof}

We define $\Vs_{P}(X)=\{\mu\in\Es_{b}'(X):\theta*\mu\in\Es^{1}_{P}(X)
\text{ for every }\theta\in\Ds(G)\}$.

\begin{prop}\label{FuncSpaces prop Vs G-invariant, Vs_P=bigcap_(P_0) Vs_(P_0)}
The space $\Vs_{P}(X)$ is a $G$-invariant subspace of
$\Es_{b}'(X)$. Furthermore, let $\mathscr{C}$ be as in Proposition
\ref{FuncSpaces prop Es^1_P=bigcap_(P_0)Es^1_(P_0)}. Then
$$
\Vs_{P}(X)=\bigcap_{P_{0}\in\mathscr{C}}\Vs_{P_{0}}(X).
$$
\end{prop}

\begin{proof}
Let $\mu\in\Vs_{P}(X)$ and let $g_{0}\in G$. We will prove that $\theta*(\ltpb{g_{0}}\mu)\in\Es^{1}_{P}(X)$ for every $\theta\in\Ds(G)$. To do so, let $\theta\in\Ds(G)$. If $\phi\in\Es_{b}(X)$, then by unimodularity of $G$
$$
\ltpb{g_{0}^{-1}}(\check{\theta}*\phi)
=\int_{G}\theta(g^{-1})\ltpb{g^{-1}g_{0}^{-1}}\phi\,dg
=\int_{G}\theta(g_{0}^{-1}g^{-1}g_{0})\ltpb{g_{0}^{-1}g^{-1}}\phi\,dg
=\check{\theta}^{g_{0}}*(\ltpb{g_{0}^{-1}}\phi),
$$
where $\theta^{g_{0}}$ is the function given by $\theta^{g_{0}}(g)=\theta(g_{0}^{-1}gg_{0})$.
Hence for every $\phi\in\Es_{b}(X)$
$$
\Big(\theta*(\ltpb{g_{0}}\mu)\Big)(\phi)
=\mu\big(\ltpb{g_{0}^{-1}}(\check{\theta}*\phi)\big)
=\mu\big(\check{\theta}^{g_{0}}*(\ltpb{g_{0}^{-1}}\phi)\big)
=\big(\theta^{g_{0}}*\mu\big)\big(\ltpb{g_{0}^{-1}}\phi\big)
=\ltpb{g_{0}}\big(\theta^{g_{0}}*\mu\big)(\phi).
$$
Since $\theta^{g_{0}}\in\Ds(G)$ and $\mu\in\Vs_{P}(X)$, we have
$\theta^{g_{0}}*\mu\in\Es^{1}_{P}(X)$. The latter space is
$G$-invariant by Proposition \ref{FuncSpaces prop
Es^1_P=bigcap_(P_0)Es^1_(P_0)}. This proves the first statement of
the proposition.

The second statement is a direct corollary of Proposition
\ref{FuncSpaces prop Es^1_P=bigcap_(P_0)Es^1_(P_0)}.
\end{proof}

We finally define
$$\Vs(X)=\{\mu\in\Es'_{b}(X):\theta*\mu\in\RDs(X)\text{ for every }\theta\in\Ds(G)\}.$$
Let $\WGs_{P}$ be a set of representatives in $K$ for $\WG[M_{P}\cap K]\setminus\WG/\WG[K\cap H]$.

\begin{prop}
$\Vs(X)=\bigcap_{w\in\WGs_{P}}\Vs_{P^{w}}(X)$.
In particular, $\Vs(X)$ is a $G$-invariant subspace of $\Es_{b}'(X)$.
\end{prop}

\begin{proof}
It is clear that $\Vs(X)$ is contained in each of the spaces
$\Vs_{P^{w}}(X)$.
Using that $M_{P}\cap K$ normalizes $P$, we obtain
$$
\bigcap_{w\in\WGs_{P}}\invCone{P^{w}}
=\bigcap_{w''\in\WG[M_{P}\cap K]}\bigcap_{w\in\WGs_{P}}\bigcap_{w'\in\WG[K\cap H]}\Gamma_{P^{w''ww'}}
=\bigcap_{w\in\WG}\Gamma_{P^{w}}
=\{0\}.
$$
For $w\in\WGs_{P}$, let $B_{w}$ be a compact subset of $\g[a]_{\g[q]}$. Then, by Lemma \ref{FuncSpaces lemma cap_k (B+Gamma_k) subseteq B'+cap_k Gamma_k}, the intersection $\bigcap_{w\in\WGs_{P}}(B_{w}+\invCone{P^{w}})$ is compact. This proves that the other inclusion also holds.

The second statement is now a direct corollary of Proposition \ref{FuncSpaces prop Vs G-invariant, Vs_P=bigcap_(P_0) Vs_(P_0)}.
\end{proof}

\begin{rem}
Note that $\Es'(X),\RDs(X)\subseteq\Vs_{P}(X)\cap\Vs(X)$.
Furthermore, the spaces
 $\Vs_{P}(X)$ and $\Vs(X)$ contain all
integrable functions $\phi$ on $X$ that are of rapid decay, i.e.,
the functions $\phi$ with the property that
if $C$ is a compact subset of $G$, then for every $n\in\N$
$$
\sup_{x\in X}\|x\|^{n}\int_{C} |\ltpb{g}\phi(x)|\,dg<\infty.$$
Here $\|\cdot\|:X\to\R$ denotes the function
given by $\|ka\cdot x_{0}\|=e^{\|\log a\|}$ for $k \in K$ and $a \in A_{\g[q]}$.
The subspace of $L^{1}(X)$ consisting of the
functions with support contained in $X(B+\invCone{P})$
for some compact subset $B$ of $\g[a]_{\g[q]}$, is a subspace of
$\Vs_{P}(X)$ as well.
\end{rem}

\subsection{Support theorems}
\label{subsection Support theorems}
We start with a lemma.

\begin{lemma}
Let $\theta\in\Ds(G)$, $\mu\in\Es'_{b}(X)$ and
$\psi\in\Es_{b}(\Xi_{P},J_{P}^{-1})$, then $\Rt[P](\theta*\mu)(\psi)=\Rt[P]\mu(\check{\theta}*\psi)$.
\end{lemma}

\begin{proof}
We denote the left regular representation on $\Es_{b}(X)$ and
$\Es_{b}(\Xi_{P}, J_{P}^{-1})$ both by $L$. Using equivariance and
continuity of $\dRt[P]$, we obtain
\begin{align*}
\Rt[P](\theta*\mu)(\psi)
&=(\theta* \mu)(\dRt[P]\psi)
=\mu\big(\check{\theta}*(\dRt[P]\psi)\big)
=\int_{G}\theta(g^{-1})\mu\big(L_{g}\dRt[P]\psi\big)\,dg\\
&=\int_{G}\theta(g^{-1})\Rt[P]\mu(L_{g}\psi)\,dg
=\Rt[P]\mu(\check{\theta}*\psi).
\end{align*}
\end{proof}

\begin{thm}[Support Theorem]\label{SuppThm thm Support theorem for distributions and non-minimal P}
Let $B$ be a $\WG[M_{P}\cap K\cap H]$-invariant convex compact subset of $\g[a]_{\g[q]}$ and let $\mu\in\Vs_{P}(X)$. If
$$\supp(\Rt[P]\mu)\subseteq \Xi_{P}(B+\Gamma_{P}),$$
then
$$\supp(\mu)\subseteq
X\Big(\bigcap_{w\in \WG[K\cap H]}w\cdot(B+\Gamma_{P})\Big).
$$
\end{thm}

\begin{rem}
Note that if $P=P_{0}$ is a minimal $\sigma\circ\theta$-stable parabolic subgroup, then any subset $B$ of $\g[a]_{\g[q]}$ is $\WG[M_{P_{0}}\cap K\cap H]$-invariant since $M_{P_{0}}$ centralizes $\g[a]_{\g[q]}$.

If $P=G$, then $\Rt[P]=\Rt[G]$ equals the identity operator $\Vs_{G}(X)\to\Vs_{G}(X)$. In this case the support theorem reduces to the following tautology.  Let $B$ be a $\WG[K\cap H]$-invariant convex compact subset of $\g[a]_{\g[q]}$ and let $\mu\in\Vs_{P}(X)$. Then $\supp(\mu)\subseteq X(B)$ implies $\supp(\mu)\subseteq X(B)$.
\end{rem}

\begin{proof}[Proof for Theorem \ref{SuppThm thm Support theorem for distributions and non-minimal P}]
Let $\mu\in\Vs_{P}(X)$ and assume that $\supp(\Rt[P]\mu)\subseteq\Xi_{P}(B+\Gamma_{P})$.
Let $B_{U}$ be a closed ball in $\g[a]$  centered at $0$ and
let $U=K\exp(B_{U})K$. Note that $U$ is symmetric in the sense that $\{u^{-1}:u\in U\}=U$.
Let $\theta\in\Ds(G)$  and assume that $\supp(\theta)\subseteq U$.
If
$\psi\in\Ds(\Xi_{P})$ satisfies $\supp(\psi)\cap
U\cdot \Xi_{P}(B+\Gamma_{P})=\emptyset$,
then $\supp(\check{\theta}*\psi)\cap\Xi_{P}(B+\Gamma_{P}) =\emptyset$ and thus $\Rt[P](\theta*\mu)(\psi)
=\Rt[P]\mu(\check{\theta}*\psi) =0$. As this holds for all $\psi$ as above, $\supp\big(\Rt[P](\check{\theta}*\mu)\big)\subseteq U\cdot\Xi_{P}(B+\Gamma_{P})$.
(Here we used that $U$ is compact, so that $U\cdot\Xi_{P}(B+\Gamma_{P})$ is closed.)

Let $\mathscr{C}$ be the set of $P_{0}\in\Psg(\g[a]_{\g[q]})$ that are
contained in $P$. Let $P_{0}\in\mathscr{C}$.
By Lemma \ref{RadTrans lemma A_P=pi_(a_P cap q)A_KAN} we have
$$
\Acomp_{P_{0}}\Big(UK\exp(B+\Gamma_{P})(L_{P}\cap H)\Big) =\Acomp_{P_{0}}\Big(\exp(B_{U})K\Big)+\Acomp_{P_{0}}\Big(\exp(B+\Gamma_{P})(L_{P}\cap H)\Big).
$$
Again by Lemma \ref{RadTrans lemma A_P=pi_(a_P cap q)A_KAN} and by Kostant's convexity theorem (\cite[Theorem 4.1]{Kostant_OnConvexityTheWeylGroupAndTheIwasawaDecomposition}) the first term on the right-hand side is contained in $B_{U}\cap \g[a]_{\g[q]}$. By Theorem \ref{RadTransSupp Thm Convexity theorem} applied to $L_{P}/(L_{P}\cap H)$ and $P_{0}\cap L_{P}$, the second term is contained in $B+\Gamma_{P_{0}}$. Since this holds for all $P_{0}\in\mathscr{C}$ and $(B_{U}\cap \g[a]_{\g[q]})+B$ is a $W_{M_{P}\cap K\cap H}$-invariant compact convex subset of $\g[a]_{\g[q]}$, we can apply Lemma \ref{RadTransSupp lemma relation Xi_(P,B), Acomp} and thus obtain $U\cdot\Xi_{P}(B+\Gamma_{P})
\subseteq\Xi_{P}\big((B_{U}\cap \g[a]_{\g[q]})+B+\Gamma_{P}\big)$.
Proposition \ref{FuncSpaces prop Vs G-invariant, Vs_P=bigcap_(P_0) Vs_(P_0)}
implies that $\mu\in\Vs_{P_{0}}(X)$, hence $\theta*\mu\in\Es^{1}_{P_{0}}(X)$. From Corollary
\ref{RadTransSupp cor supp(Rt_Q phi)subseteq Xi_Q,B => supp(Rt_P
phi)subseteq Xi_P,B} and Proposition \ref{RadTransRel prop
Rt_P=Rt_P^Q circ Rt_Q} it follows that
$$
\supp\Big(\Rt[P_{0}](\theta*\mu)\Big)
\subseteq\Xi_{P_{0}}\Big((B_{U}\cap \g[a]_{\g[q]})+B+\Gamma_{P_{0}}\Big).
$$
Now Theorem \ref{SuppThm thm Support Theorem for functions and minimal P_0}
can be applied and thus we conclude that
$$
\supp(\theta*\mu)\subseteq X\Big(\bigcap_{w\in\WG[K\cap H]}
w\cdot\Big((B_{U}\cap \g[a]_{\g[q]})+B+\Gamma_{P_{0}}\big)\Big).
$$

For each $j \in \N$, let $B_{j}\subseteq\g[a]$ be the ball of radius $1/j$ and centered at $0$ and let $U_{j} = K\exp(B_{j})K$. Let
$(\theta_{j}\in\Ds(G))_{j\in\N}$ be a sequence such that $\supp(\theta_{j})\subseteq U_{j}$
and $\theta_{j}\to\delta$ in $\Es'(G)$ (with respect to the weak
topology) for $j\to\infty$. Since convolution is sequentially
continuous with respect to each variable separately, the sequence
$\theta_{j}*\mu$ converges to $\mu$ in $\Ds'(X)$ (with respect to
the weak topology) for $j \to \infty$. As the $B_{U_{j}}$ form a decreasing sequence of sets, we therefore have
$$
\supp(\mu)
\subseteq\bigcap_{j\in\N}X\Big(\bigcap_{w\in\WG[K\cap H]}
    w\cdot\big((B_{U_{j}}\cap \g[a]_{\g[q]})+B+\Gamma_{P_{0}}\big)\Big)
=X\Big(\bigcap_{w\in\WG[K\cap H]}w\cdot(B+\Gamma_{P_{0}})\Big).
$$
Since this is true for
each $P_{0}\in\mathscr{C}$, it follows that
$$
\supp(\mu)
\subseteq X\Big(\bigcap_{P_{0}\in\mathscr{C}}\bigcap_{w\in\WG[K\cap H]}
                                    w\cdot(B+\Gamma_{P_{0}})\Big).
$$
The theorem now follows by application of Lemma
\ref{RadTransSupp Lemma B+Gamma_P=cap(B+Gamma_P_0)}.
\end{proof}

\begin{rem}
With essentially the same proof, it is seen that the support
theorem can be generalized to distributions $\mu\in\Es_{b}'(X)$
for which there exist a sequence $(\theta_{j}\in\Ds(G))_{j\in\N}$
such that
\begin{enumerate}[(i)]
\item
    $\supp(\theta_{j})$ is contained in $B_{j}$,
\item
    $\theta_{j}\to\delta$ in $\Es'(G)$ for $j\to\infty$ (with respect to the weak topology),
\item
    $\theta_{j}*\mu\in\Es^{1}_{P}(X)$.
\end{enumerate}
It is not clear to us whether the subset
of these distributions forms a subspace of $\Es_{b}'(X)$,
nor are we certain that the set of these distributions is
actually strictly larger than $\Vs_{P}(X)$.
\end{rem}

\begin{cor}
\label{SuppThm cor supp Rt_P phi(xi)=0 for xi cap KBH=empty and all P ->supp phi subset KBH}
Let $\mu\in\Vs(X)$, let $B$ be a $\WG$-invariant
convex compact subset of $\g[a]_{\g[q]}$ and let $g\in G$. Then the
following statements are equivalent.
\begin{enumerate}[(i)]
\item
    $\supp(\Rt[P^{w}]\mu)\subseteq g\cdot \Xi_{P^{w}}(B+\Gamma_{P^{w}})$
             for every  $w\in\WGs_{P}$.
\item
    $\supp(\mu)\subseteq g\cdot X(B)$.
\end{enumerate}
\end{cor}

\begin{proof}\ \\
{\em (i)$\Rightarrow$(ii):}
If $\supp(\Rt[P^{w}]\mu)$ is contained in $g\cdot \Xi_{P^{w}}(B+\Gamma_{P^{w}})$, then
$\supp\big(\Rt[P^{w}](\ltpb{g}\mu)\big)$ is contained in $\Xi_{P^{w}}(B+\Gamma_{P^{w}})$,
hence
$$
   \supp(\ltpb{g}\mu)\subseteq X\Big(\bigcap_{w'\in\WG[K\cap H]}
                  (B+\Gamma_{P^{ww'}})\Big)
$$
by Theorem
\ref{SuppThm thm Support theorem for distributions and non-minimal P}.
Since  this holds for all
$w\in\WGs_{P}$, it follows that
$$
\supp(\ltpb{g}\mu)
\subseteq X\Big(\bigcap_{w\in\WGs_{P}}\bigcap_{w'\in\WG[K\cap H]}(B+\Gamma_{P^{ww'}})\Big).
$$
As $P$ is stable under $\WG[M_{P}\cap K]$, it follows that the last equals $X\Big(\bigcap_{w\in\WG}(B+\Gamma_{P^{w}})\Big)$.
According to Lemma \ref{ConvexGeom Lemma S=cap(S+Gamma)}
the latter set equals $X(B)$.
We thus obtain $\supp(\mu)\subseteq g\cdot X(B)$.
\\
{\em (ii)$\Rightarrow$(i):}
    This is a consequence of Corollary
    \ref{RadTransSupp cor supp(Rt_Q phi)subseteq Xi_Q,B => supp(Rt_P phi)subseteq Xi_P,B}.
\end{proof}

If $X$ is a Riemannian symmetric space (hence $\sigma=\theta$) and
$P=P_{0}$ is a minimal parabolic subgroup, then
Theorem \ref{SuppThm thm Support theorem for distributions and non-minimal P}
reduces to the support theorem
\cite[Lemma 8.1]{Helgason_TheSurjectivityOfInvariantDifferentialOperatorsOnSymmetricSpaces} of Helgason for the horospherical transform on $X$. (See also Theorem 1.1,  Corollary 1.2 and the subsequent Remark in chapter IV of \cite{Helgason_GeometricAnalyisOnSymmetricSpaces}.) The support theorem can in this case be described in a purely geometrical setting as follows.

Suppose $X$ is a Riemannian symmetric space. A horosphere in $X$
is a closed submanifold of $X$ by Proposition \ref{RadTrans prop
N_P cdot xi_Q closed submanifold}. Therefore the Riemannian
structure on $X$ induces a Riemannian structure and thus a measure
on every horosphere. Let $\Rt$ be the transform mapping a function
$\phi\in\RDs(X)$ to the function on the set $\Hor(X)$ of
horospheres in $X$
$$
\Rt\phi:\xi\mapsto\int_{x\in\xi}\phi(x)\,dx.
$$
In this case $\Hor(X)$ is in bijection with $\Xi_{P_{0}}$ where
$P_{0}$ is a minimal parabolic subgroup of the identity component
$G$ of the isometry group of $X$. In this way $\Hor(X)$ can be
given the structure of a $G$-manifold. Let $x\in X$. The
stabilizer $G_{x}$ of $x$ in $G$ (i.e., the isotropy group of $G$
at $x$) acts transitively on the set of horospheres containing
$x$. Therefore this set carries a unique normalized
$G_{x}$-invariant measure $d\xi$. The dual transform of $\Rt$ is
the transform $\dRt$ mapping a function $\psi\in\Es(X)$ to
$$\dRt\psi:x\mapsto\int_{\xi\ni x}\psi(\xi)\,d\xi.$$
The Radon transform $\Rt$ is defined on $\Vs(X)$ to be the
transpose of $\dRt$. Let $d(\cdot,\cdot)$ be the distance-function
on $X$. For $x\in X$ and $R\geq0$ we define
$\beta_{R}(x)=\{\xi\in\Hor(X): d(\xi,x)\leq R\}$ and $B_{R}(x)=\{x'\in X:d(x',x)\leq R\}$.

\begin{cor}[{Riemannian case; \cite[Ch. IV, Corollary 1.2]{Helgason_GeometricAnalyisOnSymmetricSpaces}}]
Let $\mu\in\Vs(X)$, $x\in X$ and $R\geq 0$. Then $\supp(\Rt\mu)\subseteq \beta_{R}(x)$ if and only if $\supp(\mu)\subseteq B_{R}(x)$.
\end{cor}

\begin{proof}
Every closed ball in $X=G/K$  is of the form $g\cdot X(B)$,
where $g\in G$ and $B$ is a closed ball in $\g[a]$.
Therefore, the statement follows directly from
Lemma \ref{RadTransSupp lemma E(xi) cap KBH neq empty <-> xi in KBexp(Gamma) xi_P}
and Corollary \ref{SuppThm cor supp Rt_P phi(xi)=0 for xi cap KBH=empty and all P ->supp phi subset KBH}.
\end{proof}

\subsection{Injectivity}
\label{subsection Injectivity}
Theorem \ref{SuppThm thm Support theorem for distributions and non-minimal P} has the following corollary.

\begin{thm}\label{Inj thm Rt_P injective}
$\Rt[P]:\Vs_{P}(X)\to\Es_{b}'(\Xi_{P},J_{P}^{-1})$ is injective.
\end{thm}

In \cite[Theorem 5.5]{Krotz_HorosphericalTransformOnRealSymmetricVarieties:KernelAndCokernel}
it is claimed that the horospherical transform is injective on a
certain subspace of $\RDs(X)$. The proof for this theorem relies
in an essential way on the assumption that $\Hor(X)$ admits the
structure of an analytic manifold and the horospherical transforms
$\Rt[P^{w}]$ for $w\in\WGs$ together induce a transform $\Rt$ on
$\Hor(X)$ with the property that, for a real
analytic vector $\phi$ for the left-regular representation of $G$
on $L^{1}(X)$, $\Rt\phi$ is a real analytic function on $\Hor(X)$. As stated
in Remark \ref{RadTrans rem Kroetz} we believe that there are some
problems with this kind of reasoning.

\medbreak

A natural question is whether Theorem \ref{SuppThm thm Support
theorem for distributions and non-minimal P} can be generalized to
a support theorem for a larger subspace of $\Ds'(X)$. If so, the
Radon transform $\Rt[P]$ would be injective on that larger
subspace as well. We will now show that the support theorem does
not hold in general on the Harish-Chandra Schwartz spaces
$\Cs^{p}(X)$ for $0<p\leq 1$.

We will use the notations introduced in Section \ref{section
Support theorem for the horospherical transform}. Let $P_{0} \in
\Psg(\g[a]_{\g[q]})$ and let $0<p<1$.

\begin{lemma}\label{FurtherRemarks lemma int_X phi(x) j_P(x) dx abs convergent}
Let $\zeta \in \widehat{M_{P_{0}}}_{H}$, let
$1<c<\frac{2}{p}-1$ and let $\eta\in V(\zeta)$. If $\phi\in
\Cs^{p}(X)$ then for every $g\in G$ and $\lambda\in
c\rho_{P_{0}}+i\g[a]_{\g[q]}^{*}$
\begin{equation}\label{FurtherRemarks eq F_M F_A_q R_P phi}
\int_{X}\phi(x)j(P_{0}:\zeta:-\lambda)(\eta)(g\cdot x)\,dx
=\int_{M_{P_{0}}\cap K}\Ft_{A_{\g[q]}}\Big(\Ht[P_{0}]\phi(g^{-1}m)\Big)(\lambda)\zeta(m)\eta\,dm
\end{equation}
with absolutely convergent integrals.
\end{lemma}

\begin{proof}
Since $\Cs^{p}(X)$ is $G$-invariant, it suffices to prove the claim for $g=e$.

Let $\lambda$ be as in the lemma.
If $kah=ma'nwh'$ with  $k\in K$, $a, a'\in A_{\g[q]}$, $h,h'\in H$, $m\in
M_{P_{0}}\cap K$, $w\in\WGs$ and $n\in N_{P_{0}}$, then $w\cdot\log(a')=\Acomp_{P_{0}^{w}}(ahh'^{-1})$.
The last is contained in $\ch\big(\WG[K\cap H]\cdot
\log(a)\big)+\Gamma_{P_{0}^{w}}$ by the convexity theorem
of Van den Ban (Theorem \ref{RadTransSupp Thm Convexity theorem}). Hence $\log(a')\in\ch\big(\WG\cdot\log(a)\big)+\Gamma_{P_{0}}$
and therefore
$$
|(a')^{-\lambda+\rho_{P_{0}}}| \leq \max_{w\in
\WG}(waw^{-1})^{(1-c)\rho_{P_{0}}}.
$$

We define the function
$$
J:A_{\g[q]}\to\R_{\geq0};\quad a\mapsto
\prod_{\alpha\in\Sigma(\g,\g[a]_{\g[q]};P_{0})}
	|a^{\alpha}-a^{-\alpha}|^{m_{\alpha}^{+}}(a^{\alpha}+a^{-\alpha})^{m_{\alpha}^{-}},
$$
where $m_{\alpha}^{\pm}$ is the dimension of the $\pm1$-eigenspace
for $\sigma\circ\theta$ in $\g_{\alpha}$. Since
$\lambda\in\g[a]_{\g[q]}^{*}(\overline{P}_{0},0) +\rho_{P_{0}}$,
it follows by \cite[Proposition
5.6]{vdBan_ThePrincipalSeriesForAReductiveSymmetricSpaceI} that
$j(P_{0}:\zeta:\lambda)$ is continuous. In view of \cite[p.
149]{Schlichtkrull_HyperfunctionsAndHarmonicAnalysisOnSymmetricSpaces},
 there exists a constant $c>0$ such that
\begin{eqnarray*}
\lefteqn{\int_{X}\|\phi(x) j(P_{0}:\zeta:-\lambda)(\eta)(x)\|\,dx}\\
&=&c\int_{K}\int_{A_{\g[q]}}\|\phi((ka\cdot x_{0}))
j(P_{0}:\zeta:-\lambda)(\eta)((ka\cdot x_{0}))\|J(a)\,da\,dk\\
&\leq&c\int_{K}\int_{A_{\g[q]}}|\phi(ka\cdot x_{0})| J(a)
\max_{w\in \WG}(waw^{-1})^{(1-c)\rho_{P_{0}}}\,da\,dk.
\end{eqnarray*}

Following Harish-Chandra, we use the notation $\Xi$ for the
elementary spherical function on $G$ with spectral parameter $0$,
and we put
$$
\Theta:X\to\R_{>0};\quad x\mapsto
\sqrt{\Xi\big(x\sigma(x)^{-1}\big)}.
$$
By \cite[Theorem
17.1]{vdBan_ThePrincipalSeriesForAReductiveSymmetricSpaceII} there
exists a constant $C >0$ such that $|\phi|\leq C\Theta^{\frac{2}{p}}$.
Furthermore, by \cite[Corollary
17.6]{vdBan_ThePrincipalSeriesForAReductiveSymmetricSpaceII} it
follows that for sufficiently small $\epsilon>0$ there exists a
constant $C_{\epsilon}>0$ such that $J(a)\leq C_{\epsilon}\Theta(ka\cdot x_{0})^{-2-\epsilon}$.
We infer that there exists a constant $\widetilde{C}_{\epsilon}>0$
such that
\begin{eqnarray*}
\lefteqn{\int_{X}\|\phi(x) j(P_{0}:\zeta:-\lambda)(\eta)(x)\|\,dx}\\
&\leq &\widetilde{C}_{\epsilon}\int_{K}\int_{A_{\g[q]}}
    \Theta^{\frac{2}{p}-2-\epsilon}(ka\cdot x_{0})
       \max_{w\in \WG[K\cap H]}(waw^{-1})^{-(1-c)\rho_{P_{0}}}\,da\,dk\\
&\leq &\widetilde{C}_{\epsilon}|\WG|\int_{A^{+}_{\g[q]}(P_{0})}
     \Theta^{\frac{2}{p}-2-\epsilon}(a\cdot x_{0})a^{-(1-c)\rho_{P_{0}}}\,da.
\end{eqnarray*}
By \cite[Corollary
17.6]{vdBan_ThePrincipalSeriesForAReductiveSymmetricSpaceII}, for
every $\delta>0$ there exists a constant $c_{\delta}>0$ such that $\Theta(a\cdot x_{0})\leq c_{\delta}a^{(\delta-1)\rho_{P_{0}}}$ for $a\in A^{+}_{\g[q]}(P_{0})$. Hence for $\epsilon<\frac{2}{p}-1-c$ the last integral is convergent. The claimed equality follows from equation (\ref{FT eq int_(P_0
cdot x0)=sum_W int_LcapK int_A_q int_N_P}).
\end{proof}

Let $\vartheta$ be a finite subset of $\widehat{K}$ and put
$\tau=\tau_{\vartheta}$. For $x\in X$ and
$\lambda\in\g[a]_{\g[q],\C}^{*}$, let
$E_{P_{0}}(\,\cdot:\lambda:x)$ denote the (unnormalized)
$\tau$-spherical Eisenstein integral defined in \cite[Section
2]{vdBan&Schlichtkrull_FourierTransformOnASemisimpleSymmetricSpace},
i.e., the element of
$\Hom(\leftsuperscript{\circ}{\Cs}(\tau),V_{\tau})$ given by
$$E_{P_{0}}(\psi_{f\otimes\eta}:\lambda:x)(k)
=\int_{K}\langle f(l)(k),j(P_{0}:\zeta:\overline{\lambda})(\eta)(l\cdot x)\rangle\,dl
$$
for $\zeta\in\widehat{M_{P_{0}}}_{H}$, $f\in C(K:\zeta:\tau)$,
$\eta\in \overline{V(\zeta)}$ and $x\in X$.

Using the $K$-invariance of the measure on $X$, we obtain the
following immediate corollary of Lemma \ref{FurtherRemarks lemma
int_X phi(x) j_P(x) dx abs convergent}.

\begin{cor}\label{FurtherRemarks cor Eisenstein int regular,
equality of Fts} Let $\zeta\in\widehat{M_{P_{0}}}_{H}$, let
$\eta\in \overline{V(\zeta)}$, let $f\in C(K:\zeta:\tau)$ and let $1<c<\frac{2}{p}-1$. Then the $\tau$-spherical
Eisenstein integral $E_{P_{0}}(\psi_{f\otimes\eta}:\cdot:x)$ is
regular on $c\rho_{P_{0}}+i\g[a]_{\g[q]}^{*}$. Moreover, for every
$\phi\in\Cs^{p}(X)_{\vartheta}$ and $\lambda\in c\rho_{P_{0}}+i\g[a]_{\g[q]}^{*}$
$$
\int_{X}\int_{K}
    \varsigma(\phi)(x)(k)\,
    \overline{E_{P_{0}}(\psi_{f\otimes\eta}:-\overline{\lambda}:x)}(k)\,dk\,dx
=
\int_{X}\phi(x)\langle
j(P_{0}:\zeta:-\lambda)(\eta)(x),f\rangle\,dx,
$$
where the integrals are absolutely convergent.
\end{cor}

For $x\in X$ and $\lambda\in\g[a]_{\g[q],\C}$, let
$E^{\circ}_{\overline{P}_{0}}(\,\cdot:-\overline{\lambda}:x)$ be
the normalized $\tau$-spherical Eisenstein integral for
$\overline{P}_{0}$ defined in \cite[Section
5]{vdBan&Schlichtkrull_FourierTransformOnASemisimpleSymmetricSpace},
i.e., the element of
$\Hom(\leftsuperscript{\circ}{\Cs}(\tau),V_{\tau})$ given by
\begin{equation}\label{FurtherRemarks eq def E^0}
E^{\circ}_{\overline{P}_{0}}(\psi_{f\otimes\eta}:\lambda:x)
=E_{P_{0}}(\psi_{A(\overline{P}_{0}:P_{0}:\zeta:-\lambda)^{-1}f\otimes\eta}:\lambda:x).
\end{equation}

For $r\in\R$ we define $C_{r}(X,\tau)$ to be the space of
continuous functions $f:X\to V_{\tau}$ satisfying $f(k\cdot x)=\tau(k)f(x)$ for $k\in K$
and the estimate
$$\sup_{k\in K, a\in A_{\g[q]}}e^{-r \|\log a\|}|f(ka\cdot x_{0})|<\infty.$$
Let $R\in\R$ be such that $\rho_{P_{0}}\in
\g[a]_{\g[q]}^{*}(\overline{P}_{0},R)$ and let $\omega$ be a
connected and bounded open subset of
$\g[a]_{\g[q]}^{*}(\overline{P}_{0},R)$ containing both $0$ and
$\rho_{P_{0}}$. Let $E_{P_{0}}^{\circ}(\lambda:x)^{*}$ be the dual
of $E_{P_{0}}^{\circ}(\,\cdot:-\overline{\lambda}:x)$. Then, according to
\cite[Lemma
12.2]{vdBan&Schlichtkrull_FourierInversionOnAReductiveSymmetricSpace}
there exists an $r\in\R$ such that the
$\leftsuperscript{\circ}{\Cs}(\tau)$-valued integral
\begin{equation}\label{FurtherRemarks eq int_X E^circ* f dx}
\int_{X}E^{\circ}_{P_{0}}(-\overline{\lambda}:x)^{*}f(x)\,dx
\end{equation}
is absolutely convergent for every $f\in C_{r}(X:\tau)$ and
generic $\lambda\in\omega+i\g[a]_{\g[q]}^{*}$, and
(\ref{FurtherRemarks eq int_X E^circ* f dx}) depends
meromorphically on $\lambda$ in that region. Following
\cite[Section
12]{vdBan&Schlichtkrull_FourierInversionOnAReductiveSymmetricSpace}
we define the (normalized) $\tau$-spherical Fourier transform
$\Ft_{\overline{P}_{0},\tau}f(\lambda)$ of a function $f\in
C_{r}(X:\tau)$ for generic $\lambda\in\omega+i\g[a]_{\g[q]}^{*}$ to be
given by (\ref{FurtherRemarks eq int_X E^circ* f dx}). This
definition coincides with the definition for compactly supported
smooth functions $\phi$ given in Section \ref{subsection The
tau-spherical Fourier transform}. Recall that $\pr_{e}$ denotes the projection
$\leftsuperscript{0}{\Cs(\tau)}\to\leftsuperscript{0}{\Cs(\tau)}_{e}$.

If  $0<p<1$ is sufficiently small, then $\varsigma$ maps
$\Cs^{p}(X)_{\vartheta}$ into $C_{r}(X:\tau)$. Fix such a $p$.

\begin{prop}\label{FurtherRemarks prop pi_e F phi=0 <=> R phi=0}
Let $\phi\in\Cs^{p}(X)_{\vartheta}$. Then $\pr_{e}\Ft_{\overline{P}_{0},\tau}\varsigma(\phi)\big|_{i\g[a]_{\g[q]}^{*}}=0$ if and only if $\Rt[P_{0}]\phi=0$.
\end{prop}

\begin{proof}
Fix a $1<c<\frac{2}{p}-1$ such that $c\rho_{P_{0}}+i\g[a]_{\g[q]}^{*}\subseteq\omega+i\g[a]_{\g[q]}^{*}$ and $\{A(\overline{P}_{0}:P_{0}:\zeta:\overline{\lambda})^{-1}:\lambda\in c\rho_{P_{0}}+i\g[a]_{\g[q]}^{*}\}$
is a regular family of bijective operators. As
$\Ft_{\overline{P}_{0},\tau}\phi(\lambda)$ depends
meromorphically on $\lambda$, it follows that the restriction of
$\pr_{e}\Ft_{\overline{P}_{0},\tau}\phi$ to
$i\g[a]_{\g[q]}^{*}$  vanishes if and only if it
vanishes on $c\rho_{P_{0}}+i\g[a]_{\g[q]}^{*}$. From
(\ref{FurtherRemarks eq def E^0}) and Corollary
\ref{FurtherRemarks cor Eisenstein int regular, equality of Fts}
it follows that the latter is the case if and only if
the right-hand side of (\ref{FurtherRemarks eq F_M F_A_q R_P phi}) vanishes for every
$g\in G$, $\lambda\in\g[a]_{\g[q],\C}^{*}$,
$\zeta\in\widehat{M_{P_{0}}}_{H}$ and $\eta\in V(\zeta,e)$. We
will now show that the last is equivalent to $\Rt_{P_{0}}\phi=0$.
For this we observe from \cite[Lemma 3.5]{vdBan_ThePrincipalSeriesForAReductiveSymmetricSpaceI} that (\ref{FurtherRemarks eq F_M F_A_q R_P
phi}) is basically equal to the Fourier transform on the compact
homogeneous space $M_{P_{0}}/(M_{P_{0}}\cap H)$ applied to the
Euclidean Fourier transform on $A_{\g[q]}$ of the horospherical
transform of $\phi$.
For fixed $\lambda\in c\rho_{P_{0}}+i\g[a]_{\g[q]}^{*}$, the
function $(M_{P_{0}}\cap K)/(M_{P_{0}}\cap K\cap H)\to\C$;
$$
m\cdot (M_{P_{0}}\cap K\cap H)
    \mapsto\Ft_{A_{\g[q]}}\Big(\Ht[P_{0}]\phi(m)\Big)(\lambda)$$
is continuous and hence square integrable.
The Fourier transform on
$L^{2}\big(M_{P_{0}}/(M_{P_{0}}\cap H)\big)$ is injective.
Moreover, $\psi\mapsto\Ft_{A_{\g[q]}}\psi\big|_{c\rho_{P_{0}}+i\g[a]_{\g[q]}^{*}}$ is injective. Therefore, it follows
that $\Rt[P_{0}]\phi=0$ if and only if the restriction of
$\pr_{e}\Ft_{\overline{P}_{0},\tau}\phi$ to $i\g[a]_{\g[q]}^{*}$
vanishes.
\end{proof}

By the Plancherel decomposition (\cite[Th\'eor\`eme
3]{Delorme_FormuleDePlancherelPourLesEsPacesSymetriquesReductifs}
and \cite[Theorem
23.1]{vdBan&Schlichtkrull_ThePlancherelDecompositionForAReductiveSymmetricSpaceI}) the kernel of $\Ft_{\overline{P}_{0},\tau}$ is non-trivial if and only if there are discrete series or intermediate series of representations present. If this kernel has a non-trivial intersection with $\Cs^{p}(X)_{\vartheta}$, then the kernel of $\Rt[P_{0}]$ in $\Cs^{p}(X)$ is non-trivial and the support theorem cannot hold on this space. The estimates in the proof of \cite[Theorem
4.8]{Flensted_Jensen_DiscreteSeriesForSemisimpleSymmetricSpaces}
together with \cite[Theorem
7.3]{vdBan_AsymptoticBehaviourOfMatrixCoefficientsRelatedToReductiveSymmetricSpaces}
show that this is in particular the case if the rank condition
\begin{equation}\label{FurtherRemarks eq rank G/H=rank K/(K cap H)}
\textnormal{rank}(G/H)=\textnormal{rank}(K/(K\cap H))
\end{equation}
is satisfied, because in that case there exists a finite subset $\vartheta$ of $\widehat{K}$ such that the subspace of $\Cs^{p}(X)_{\vartheta}$
corresponding to the discrete series of representations is
non-trivial. We thus obtain the following proposition.

\begin{prop}
Let $0<p<1$. Assume that (\ref{FurtherRemarks eq rank G/H=rank
K/(K cap H)}) holds. Then the horospherical transform $\Rt_{P_{0}}$ is not injective on $\Cs^{p}(X)$ and the support theorem for the
horospherical transform for functions (Theorem \ref{SuppThm thm
Support Theorem for functions and minimal P_0}) is not valid with
$\Es^{1}_{P_{0}}(X)$ replaced by $\Cs^{p}(X)$.
\end{prop}

A similar result, but with a different proof, can be found in
\cite[Theorems 4.1 and 4.2]{Krotz_HorosphericalTransformOnRealSymmetricVarieties:KernelAndCokernel}.

Finally, we note that Theorem \ref{Inj thm Rt_P injective} and
Proposition \ref{FurtherRemarks prop pi_e F phi=0 <=> R phi=0}
imply that $\pr_{e}\Ft_{\overline{P}_{0},\tau}$ is injective on
$\RDs(X)_{\vartheta}$. However, the following stronger result
already follows from the inversion formula (Proposition \ref{InvFor
prop Inversion Formula}), together with (\ref{FT eq def tau-sph
Ft_e}) and the equivariance of $\Ft^{\un}_{P_{0},e}$.

\begin{thm}
Assume that $P_{0}$ is a minimal $\sigma\circ\theta$-stable
parabolic subgroup, $\vartheta$ is a finite subset of $K$-types
and $\tau=\tau_{\vartheta}$. Then the map
$$
\Es^{1}_{P_{0}}(X)_{\vartheta}\ni\phi\mapsto\Ft_{\overline{P}_{0},\tau,e}\phi
$$
is injective.
\end{thm}

This theorem might be a consequence of the symmetry relations
between the components of the Fourier transform. (See
\cite[Section
16]{vdBan&Schlichtkrull_TheMostContinuousPartOfThePlancherelDecompositionForAReductiveSymmetricSpace}.)
It is however not clear to us how to prove this result using only
those symmetries.

\addcontentsline{toc}{section}{\protect\numberline{}Appendix A. Normalization of Measures}
 \section*{Appendix A. Normalization of measures}
 \renewcommand{\thesection}{A}
 \setcounter{equation}{0}
 \setcounter{thm}{0}

In this appendix we list the normalizations of all relevant measures that we use in the main text. Throughout this appendix, let $P$ and $Q$ be parabolic subgroups such that $A\subseteq P\subseteq Q$.

On each compact subgroup, we normalize the Haar measure such that the total volume equals $1$. If $N$ is one of the groups $N_{P}^{Q}$, $N_{P}$, $A$ or $A_{\g[q]}$, then The Haar measure on $N$ is taken to be the pushforward along $\exp$ of the Lebesgue measure on the Lie algebra $\g[n]$ of $N$. The Lebesgue measure on $\g[n]$ is normalized such that a unit cube (with respect to $-B(\cdot,\theta\cdot)$ has volume $1$.

The measure on $\Xi_{P}$ is normalized such that if $P_{0}\in\Psg(\g[a]_{\g[q]})$ with $P_{0}\subseteq P$, then
$$
\int_{\Xi_{P}}\phi(\xi)J_{P}(\xi)\,d\xi
=\sum_{w\in\WGs[M_{P}]}\int_{M_{P_{0}}\cap K}\int_{A_{\g[q]}}\int_{N_{P_{0}^{w}}^{P}}
    \phi(wman\cdot\xi_{P})\,dn\,da\,dm
$$
for all $\phi\in \Ds(\Xi_{P})$. (See (\ref{FT eq int_(P_0 cdot x0)=sum_W int_LcapK int_A_q int_N_P}).)
By the $K$-invariance of the measure on $\Xi_{P}$ we then also have
$$
\int_{\Xi_{P}}\phi(\xi)J_{P}(\xi)\,d\xi
=\sum_{w\in\WGs[M_{P}]}\int_{K}\int_{A_{\g[q]}}\int_{N_{P_{0}^{w}}^{P}}
    \phi(kan\cdot\xi_{P})\,dn\,da\,dk
$$
(See Lemma \ref{RadTransExt lemma int_Xi phi J=sum int_K int_A int_N phi}.)
Subsequently the measure on $(L_{Q}\cap H)/(L_{P}\cap H)$ is normalized such that for every $\psi\in \Ds\big(G/(L_{P}\cap H)N_{Q}\big)$
\begin{align*}
&\int_{\Xi_{Q}}\int_{(L_{Q}\cap H)/(L_{P}\cap H)}\psi\big(gl\cdot(L_{P}\cap H)N_{Q}\big)
    \,d_{(L_{P}\cap H)}l\,d_{(L_{Q}\cap H)N_{Q}}g\\
&\qquad=\int_{\Xi_{P}}\int_{N_{P}^{Q}}\psi\big(gn\cdot(L_{P}\cap H)N_{Q}\big)\,dn\,d_{(L_{P}\cap H)N_{P}}g.
\end{align*}
(See Lemma \ref{RadTransExt lemma int_(Xi_Q)int_(L_Q cap H)/(L_P cap H)|psi(gl cdot xi_P) phi(g cdot xi_Q)|leq|psi|_(infty,J_P)|phi|_(1,J_Q)}.)
Finally, the measure on $L_{P}/(L_{P}\cap H)$ is normalized such that
$$
\int_{\Xi_{P}}\chi(\xi)J_{P}(\xi)\,d\xi
=\int_{K}\int_{L_{P}/(L_{P}\cap H)}\chi(kl\cdot\xi_{P})\,d_{L_{P}\cap H}l\,dk
$$
for every $\phi\in\Ds(\Xi_{P})$.
(See Lemma \ref{FuncSpaces lemma int_Xi phi J= int_K int_L/(L cap H)phi}.)

\addcontentsline{toc}{section}{\protect\numberline{}Appendix B. Transversality}
 \section*{Appendix B. Transversality}
 \renewcommand{\thesection}{B}
 \setcounter{equation}{0}
 \setcounter{thm}{0}

In this appendix we show that the subgroups
$(L_{Q}\cap H)N_{Q}$ and $(L_{P}\cap H)N_{P}$ of $G$,
that play an important role in the double fibration
(\ref{RadTrans eq double fibration}), are transversal.

Throughout this appendix we assume that $P$ and $Q$ are
$\sigma\circ\theta$-stable parabolic subgroups such that $A
\subseteq P\subseteq Q$.

\begin{lemma}\label{Transv lemma Nor(a cap q)=Nor(l)=Nor(a)}
$\Nor_{K\cap H}(\g[l]_{P}\cap\g[q])
    =\Nor_{K\cap H}(\g[l]_{P})
    =\Nor_{K\cap H}(\g[a]_{P})
    =\Nor_{K\cap H}(\g[a]_{P}\cap\g[a]_{\g[q]}).
$
\end{lemma}

\begin{proof}
We will first show that
\begin{equation}\label{Transv eq Cen_g(l_P) cap q ca p = a_P cap a_q}
\Cen_{\g}(\g[l]_{P}\cap\g[q])\cap\g[p]\cap\g[q]=\g[a]_{P}\cap\g[a]_{\g[q]}.
\end{equation}
Clearly the set on the right-hand side is contained in the set on
the left-hand side. Conversely, as $\g[l]_{P}\cap\g[q]$ contains $\g[a]_{\g[q]}$,
 it follows that
$$
\Cen_{\g}(\g[l]_{P}\cap\g[q])\cap\g[p]\cap\g[q]
\subseteq\Cen_{\g}(\g[a]_{\g[q]})\cap\g[p]\cap\g[q]
=\g[a]_{\g[q]}.
$$
Here we have used that $\g[a]_{\g[q]}$ is maximal abelian in
 $\g[p]\cap\g[q]$. Let $Y\in\g[a]_{\g[q]}$ and assume that
$Y$ centralizes $\g[l]_{P}\cap \g[q]$. $Y$ normalizes $\g[l]_{P}$,
hence also the $B$-orthocomplement of $\g[l]_{P}\cap\g[q]$ in
$\g[l]_{P}$, which is $\g[l]_{P} \cap \g[h]$. On the other hand,
$[\g[a]_{\g[q]}, \g[l]_{P}\cap \g[h]] \subseteq
[\g[q],\g[h]]\subseteq \g[q]$, and we see that $Y$
centralizes $\g[l]_{P}\cap\g[h]$. It follows that $Y$ centralizes
$\g[l]_{P}$, hence belongs to $\g[a]_{P}\cap\g[a]_{\g[q]}$. This
establishes (\ref{Transv eq Cen_g(l_P) cap q ca p = a_P cap a_q})

For the proof of the actual lemma, assume that $k \in K\cap H$.
Then $k$ normalizes both $\g[p]$ and $\g[q]$. Now assume that $k$
normalizes $\g[l]_{P}\cap \g[q]$. Then $k$ also normalizes
$\Cen_{\g}(\g[l]_{P} \cap \g[q])$, hence also $\Cen_{\g}(\g[l]_{P} \cap
\g[q]) \cap \g[p] \cap \g[q] = \g[a]_{P}\cap\g[a]_{\g[q]}$. If $k$
normalizes $\g[a]_{P}\cap\g[a]_{\g[q]}$, then it also normalizes
the centralizer of $\g[a]_{P}\cap \g[q]$, which is $\g[l]_{P}$. If
$k$ normalizes $\g[l]_{P}$, then it also normalizes the center of
$\g[l]_{P}$, hence also the $\g[p]$-part of the center of
$\g[l]_{P}$, which is $\g[a]_{P}$. If $k$ normalizes $\g[a]_{P}$
then also $\g[a]_{P}\cap\g[q] =\g[a]_{P}\cap\g[a]_{\g[q]}$.
Finally, if $k$ normalizes $\g[l]_{P}$ then also $\g[l]_{P}\cap
\g[q]$. This proves the lemma.
\end{proof}

\begin{prop}\label{Transv prop horosphere stabilizer}
The stabilizer in $G$ of $N_{P}^{Q}\cdot \xi_{Q}$ equals
$(L_{P}\cap H)N_{P}$.
\end{prop}

\begin{proof}
It is clear that $(L_{P}\cap H)N_{P}$ stabilizes $N_{P}^{Q}\cdot
\xi_{Q}$, hence it remains to prove that the stabilizer is contained
in $(L_{P}\cap H)N_{P}$.

Assume that $gN_{P}^{Q}\cdot \xi_{Q}=N_{P}^{Q}\cdot\xi_{Q}$.
Then $gN_{P}(L_{Q}\cap H)=N_{P}(L_{Q}\cap H)$,
hence $gN_{P}H=N_{P}H$. This implies that $g=nh$, with $n\in N_{P}$ and $h\in
H$ satisfying $hN_{P}H=N_{P}H$.
We will finish the proof by showing that $h \in L_{P}\cap H$.

We first note that
$N_{P}H$ is a submanifold
of $G$ containing $e$ (see Corollary \ref{RadTrans cor N_P(L_Q cap H) closed}) and that $hN_{P} h^{-1}\subseteq N_{P}H$.
Differentiating at $e$ and using that $\Ad(h)$ is a
linear isomorphism mapping $\g[h]$ onto itself, we see that
$\Ad(h)(\g[n]_{P}\oplus\g[h])=\g[n]_{P}\oplus\g[h]$.
Note that $\g[g]$ decomposes as
$\g[g]=(\g[h]\oplus\g[n]_{P})\oplus(\g[l]_{P}\cap\g[q])$. In fact,
$\g[l]_{P}\cap\g[q]$ is the orthocomplement of
$\g[h]\oplus\g[n]_{P}$ with respect to the non-degenerate bilinear form $B$. Therefore
$h$ normalizes $\g[l]_{P}\cap\g[q]$.

Write $h = k \exp Y$ with $k \in K$ and $Y \in \g[p]$.  As $H$ is
$\theta$-stable, it follows that $k \in K \cap H$ and $Y \in \g[p]
\cap \g[h]$. Moreover, since $\g[l]_{P} \cap \g[q]$ is
$\theta$-stable, so is the normalizer of this space in $G$ and it
follows that both $k$ and $Y$ normalize $\g[l]_{P} \cap \g[q]$. We
will finish the proof by showing that both $\exp Y$ and $k$ belong
to $L_{P}\cap H$.

We may write $Y = Y_{0} + \big(U + \sigma(U)\big)$, with $Y_{0}
\in \g[l]_{P}\cap\g[h]$ and $U \in \g[n]_{P}$. Since $Y_{0}$
normalizes $\g[l]_{P}\cap\g[q]$, we see that $U + \sigma(U)$
normalizes $\g[l]_{P} \cap \g[q]$. Let $Z\in\g[l]_{P}\cap\g[q]$.
Then $[U, Z] \in \g[n]_{P}$ and $[\sigma(U), Z] \in
\overline{\g[n]}_{P}$, so $[U + \sigma(U), Z] \in \g[l]_{P} \cap
(\g[n]_{P}\oplus\overline{\g[n]}_{P}) = \{0\}$. Thus, we see that
$U + \sigma(U)$ centralizes $\g[l]_{P} \cap \g[q]$. In particular,
$U + \sigma(U)$ centralizes $\g[a]_{P}$, which in turn implies
that $U = 0$. We now see that $Y \in \g[l]_{P} \cap \g[h] \cap
\g[p]$.
In particular, it follows that $\exp Y \in L_{P} \cap H$
so that it remains to prove the same statement for $k$.

Since both $h$ and $\exp Y$ stabilize $N_{P}\cdot x_{0}$
it follows $k=h(\exp Y)^{-1}$ stabilizes
$N_{P}\cdot x_{0}$ as well. We thus obtain
$$N_{P}H=kN_{P}H=kN_{P}k^{-1}H = k N_{P} H k^{-1}.$$
Hence, the closed submanifold $N_{P} H$ of $G$ is stable under
conjugation by $k$. It follows that its tangent space $\g[n]_{P}
\oplus \g[h]$ at $e$ is stable under $\Ad(k)$ and therefore that
the $B$-orthocomplement of this tangent space, $\g[l]_{P} \cap
\g[q]$, is stable under $\Ad(k)$ as well.
As $k\in\Nor_{K\cap H}(\g[l]_{P}\cap\g[q])$, it follows from Lemma
\ref{Transv lemma Nor(a cap q)=Nor(l)=Nor(a)} that
$k\in\Nor_{K\cap H}(\g[a]_{P})$. We will show that, in fact,
$k\in\Cen_{K\cap H}(\g[a]_{P})$. Aiming at a contradiction, assume
this not to be the case. The group $kPk^{-1}$ is a
$\sigma\circ\theta$-stable parabolic subgroup of $G$ with split
component $\g[a]_{P}$. By \cite[Proposition 7.86]{Knapp_LieGroups}
there exists an $\alpha\in\Sigma(\g,\g[a];P)$ such that
$-\alpha\in\Sigma(\g,\g[a];kPk^{-1})$. Fix
$Y_{\alpha}\in\g_{\alpha}\setminus\{0\}$. Then $\theta Y_\alpha
\in \g_{-\alpha} \subset \Ad(k) \g[n]_{P}$.
Since the set $N_{P}H$ is invariant under the left-actions of both $N_{P}$ and $kN_{P}k^{-1}$, it follows that $[Y_{\alpha}, \theta Y_{\alpha}] \subset \g[n]_{P} \oplus \g[h]$.
This implies that the orthogonal projection of $[Y_{\alpha},
\theta Y_{\alpha}]$ onto $\g[a]_{\g[q]}$ is zero. On the other
hand, the commutator $[Y_{\alpha},\theta Y_{\alpha}]$ is an
element of $\g[a]$ and therefore for every $X\in\g[a]$
$$
B([Y_{\alpha},\theta Y_{\alpha}],X)
=B(\theta Y_{\alpha},[X,Y_{\alpha}]) =-\|Y_{\alpha}\|^{2}\alpha(X)
=-\|Y_{\alpha}\|^{2} B(H_{\alpha},X).
$$
Here $H_{\alpha}$ is the element of $\g[a]_{\g[q]}$ given by
(\ref{PSubgrp eq alpha(Y)=B(H_alpha,Y)}). This implies $[Y_{\alpha},\theta Y_{\alpha}]=-\|Y_{\alpha}\|^{2}H_{\alpha}\in\g[a]_{\g[q]}\setminus\{0\}$, which gives a contradiction. We conclude that $k\in\Cen_{K\cap
H}(\g[a]_{P})=L_{P}\cap K\cap H$.
\end{proof}

\begin{prop}\label{Transv prop (L_Q cap H) cdot xi_P stabilizer}
The stabilizer of $(L_{Q}\cap H)\cdot\xi_{P}$ equals $(L_{Q}\cap H)N_{Q}$.
\end{prop}

\begin{proof}
Since $N_{Q}\subseteq N_{P}$ it is clear that $(L_{Q}\cap H)N_{Q}$
normalizes $(L_{Q}\cap H)\cdot\xi_{P}$. It remains to prove that
the stabilizer is contained in $(L_{Q}\cap H)N_{Q}$.

Let $g \in G$ and assume $g(L_{Q}\cap H)\cdot\xi_{P}=(L_{Q}\cap H)\cdot\xi_{P}$.
Since $L_{P}\subseteq L_{Q}$, this implies $g(L_{Q}\cap H)N_{P}=(L_{Q}\cap H)N_{P}$.
Hence there exist $l\in L_{Q}\cap H$ and $n\in N_{P}$ such that
$g=ln$ and $n(L_{Q}\cap H)N_{P}=(L_{Q}\cap H)N_{P}$.
We will finish the proof by showing that $n \in N_{Q}$.

By Corollary \ref{RadTrans cor N_P(L_Q cap H) closed}, the set $(L_{Q}\cap
H)N_{P}$ is a submanifold of $G$. Using that
$n(L_{Q}\cap H)n^{-1}$ is contained in $(L_{Q}\cap
H)N_{P}$, we find by differentiating at $e$ that $\Ad(n)\big((\g[l]_{Q}\cap \g[h])\oplus\g[n]_{P}\big) =(\g[l]_{Q}\cap \g[h])\oplus\g[n]_{P}$.
Here we used that $\Ad(n)$
is a linear isomorphism of $\g$ mapping $\g[n]_{P}$ onto itself.
The orthocomplement of $(\g[l]_{Q}\cap \g[h])\oplus\g[n]_{P}$ with
respect to $B$ equals $(\g[l]_{P}\cap \g[q])\oplus\g[n]_{Q}$.
Therefore $n$ normalizes $(\g[l]_{P}\cap \g[q])\oplus\g[n]_{Q}$.
In particular $\Ad(n)(\g[a]_{P}\cap\g[a]_{\g[q]})\subseteq (\g[l]_{P}\cap
\g[q])\oplus\g[n]_{Q}$.
Since $n\in N_{P}$, we also have $\Ad(n)(\g[a]_{P}\cap\g[a]_{\g[q]})\subseteq (\g[a]_{P}\cap
\g[q])\oplus\g[n]_{P}$, hence
\begin{equation}\label{Transv eq Ad(n)(a_q)subseteq (a_P cap a_q) oplus n_Q}
\Ad(n)(\g[a]_{P}\cap\g[a]_{\g[q]})\subseteq \big((\g[l]_{P}\cap
\g[q])\oplus\g[n]_{Q}\big)\cap\big((\g[a]_{P}\cap
\g[q])\oplus\g[n]_{P}\big)
=(\g[a]_{P}\cap\g[a]_{\g[q]})\oplus\g[n]_{Q}.
\end{equation}
Fix a minimal $\sigma\circ\theta$-stable parabolic subgroup
$P_{0}\in\Psg(\g[a]_{\g[q]})$ contained in $P$.
Let $S$ be the collection of simple roots
for the positive system $\Sigma^{+}=\Sigma(\g, \g[a]_{\g[q]};
P_{0})$. Let $S_{0}$ denote the set of roots $\alpha \in S$ which
vanish on $\g[a]_{Q}\cap \g[q]$ and put $S_{1} = S\setminus
S_{0}$. Let $S_{-1}$ be a finite subset of $\g[a]_{\g[q]}^{*}\setminus S$
such that $S_{-1} \cup S$ is
 a basis for $\g[a]_{\g[q]}^{*}$. Equip this basis with a total ordering
$<$ such that $S_{-1} < S_{0} < S_{1}$, and equip $\g[a]_{\g[q]}^{*}$
 with the associated lexicographic ordering, also denoted by $<$.
Since $Q$ is a $\sigma\circ\theta$-stable parabolic subgroup
containing $P_{0}$, a root $\alpha \in \Sigma^{+}$ vanishes on
$\g[a]_{Q}\cap\g[a]_{\g[q]}$ if and only if it is a sum of simple
roots from $S_{0}$.
Thus if $\Sigma_{0}^{+}$ is the set of roots in $\Sigma^{+}$
vanishing on $\g[a]_{Q}\cap\g[a]_{\g[q]}$ and $\Sigma_{1}^{+}$ its
complement, then $\Sigma_{0}^{+} < \Sigma_{1}^{+}$.
Let $\log(n)=\sum_{\alpha\in\Sigma^{+}} Y_{\alpha}$,
where $Y_{\alpha}\in\g_{\alpha}$. Then $Y_{\alpha} \in \g[n]_{P}$
for all $\alpha$. Indeed, if $\alpha|_{\g[a]_{P}\cap\g[a]_{\g[q]}}
= 0$, then $Y_{\alpha} = 0$. Let $\alpha_{0}$ be the smallest root
in $\Sigma^{+}$ such that $Y_{\alpha_{0}}\neq 0$. Then for every
$Y\in\g[a]_{P}\cap \g[q]$
$$
Y-\Ad(n)Y
=\sum_{\alpha\in\Sigma^{+}}\alpha(Y)Y_{\alpha}-\sum_{k=2}^{\infty}
\frac{\ad(\log n)^{k}Y}{k!}.
$$
The sum on the right-hand side decomposes as a sum of terms
$Z_\beta(Y) \in \g_{\beta}$ with $\beta \in \Sigma^{+}$,
$\beta\geq \alpha_{0}$. If $Y$ is a $P$-regular element, then the
lowest order part of the sum equals $Z_{\alpha_{0}}(Y) =
\alpha_{0}(Y) Y_{\alpha_{0}}$ and is different from zero. From
$(\ref{Transv eq Ad(n)(a_q)subseteq (a_P cap a_q) oplus n_Q})$ it
now follows that $Y_{\alpha_{0}}\in\g[n]_{Q}$, hence $\alpha_{0}$
does not vanish on $\g[a]_{Q}\cap\g[a]_{\g[q]}$, so $\alpha_{0}
\in \Sigma_{1}^{+}$. This implies that any root $\alpha \in
\Sigma^{+}$ with $\alpha \geq \alpha_{0}$ belongs to
$\Sigma_{1}^{+}$. Thus, if $Y_{\alpha} \neq 0$, then $\alpha \in
\Sigma_{1}^{+}$ and hence $Y_{\alpha}\in \g[n]_{Q}$. We conclude
that $n\in N_{Q}$.
\end{proof}

\begin{cor}\label{Transv cor (L_Q cap H)N_Q and (L_P cap H)N_P transversal}
$(L_{Q}\cap H)N_{Q}$ and $(L_{P}\cap H)N_{P}$ are transversal.
\end{cor} 
     \addcontentsline{toc}{section}{\protect\numberline{}References}
     \bibliographystyle{alpha}
\bibliography{Article}
\def\adritem#1{\hbox{\small #1}}
\def\distance{\hbox{\hspace{3.5cm}}}
\def\apetail{@}
\def\addKuit{\vbox{
    \adritem{J.~J.~Kuit}
    \adritem{Department of Mathematical Sciences}
    \adritem{University of Copenhagen}
    \adritem{Universitetsparken 5}
    \adritem{DK-2100 Copenhagen \O}
    \adritem{Denmark}
    \adritem{E-mail: j.j.kuit{\apetail}gmail.nl}
    }}
\mbox{\ }
\vfill
\hbox{\vbox{\addKuit}}
 \end{document}